\documentclass[fleqn,reqno,11pt,a4paper,final]{amsart}
\usepackage[a4paper,left=20mm,right=20mm,top=20mm,bottom=20mm,marginpar=20mm]{geometry}
\usepackage{amsaddr}
\usepackage{amsmath}
\usepackage{amssymb}
\usepackage{amsthm}
\usepackage{mathtools}
\usepackage{amscd}
\usepackage{cite}
\usepackage{bbm}
\usepackage{color}
\usepackage[english=american]{csquotes}
\usepackage[final]{graphicx}
\usepackage{hyperref}
\usepackage{calc}
\usepackage{mathptmx}
\usepackage{tikz}
\usepackage{graphicx}
\usepackage{stix}
\usepackage[normalem]{ulem}

\graphicspath{{../Pictures/}}

\numberwithin{equation}{section}

\newtheoremstyle{thmlemcorr}{10pt}{10pt}{\itshape}{}{\bfseries}{.}{10pt}{{\thmname{#1}\thmnumber{ #2}\thmnote{ (#3)}}}
\newtheoremstyle{thmlemcorr*}{10pt}{10pt}{\itshape}{}{\bfseries}{.}\newline{{\thmname{#1}\thmnumber{ #2}\thmnote{ (#3)}}}
\newtheoremstyle{remexample}{10pt}{10pt}{}{}{\bfseries}{.}{10pt}{{\thmname{#1}\thmnumber{ #2}\thmnote{ (#3)}}}
\newtheoremstyle{ass}{10pt}{10pt}{}{}{\bfseries}{.}{10pt}{{\thmname{#1}\thmnumber{ A#2}\thmnote{ (#3)}}}

\theoremstyle{thmlemcorr}
\newtheorem{theorem}{Theorem}
\numberwithin{theorem}{section}
\newtheorem{lemma}[theorem]{Lemma}

\theoremstyle{thmlemcorr*}
\newtheorem{theorem*}{Theorem}
\newtheorem{lemma*}[theorem]{Lemma}
\newtheorem{corollary*}[theorem]{Corollary}
\newtheorem{proposition*}[theorem]{Proposition}
\newtheorem{problem*}[theorem]{Problem}
\newtheorem{conjecture*}[theorem]{Conjecture}
\newtheorem{definition*}[theorem]{Definition}
\newtheorem{assumption*}[theorem]{Assumption}

\theoremstyle{remexample}
\newtheorem{remark}[theorem]{Remark}

\theoremstyle{ass}

\newcommand{\Dcal}{\mathcal{D}}

\newcommand{\T}{\mathbb{T}}

\newcommand{\norm}[1]{\|#1\|}

\newcommand{\abs}[1]{|#1|}

\newcommand{\R}{\mathbb{R}}
\newcommand{\C}{\mathbb{C}}

\newcommand{\Z}{\mathbb{Z}}

\newcommand{\eps}{\epsilon}

\providecommand{\de}[1]{\partial_{\theta}^{#1}}
\providecommand{\dfi}[1]{\partial_{\varphi}^{#1}}
\providecommand{\abs}[1]{\lvert#1\rvert}
\providecommand{\e}[1]{^{#1}}
\providecommand{\dhil}[1]{\dot{H}^{#1}(\T)}

\providecommand{\norm}[1]{\lVert#1\rVert}

\def\XXint#1#2#3{{\setbox0=\hbox{$#1{#2#3}{\int}$}
\vcenter{\hbox{$#2#3$}}\kern-.5\wd0}}

\renewcommand{\eps}{\varepsilon}
\renewcommand{\epsilon}{\varepsilon}
\renewcommand{\phi}{\varphi}

\begin{document}


\title[Non-Newtonian two-fluid Taylor--Couette flow]{Analysis of a two-fluid Taylor--Couette flow \\with one non-Newtonian fluid}

\author{Christina Lienstromberg, Tania Pernas-Casta\~{n}o, Juan J. L. Vel\'{a}zquez}
\address{Institute of Applied Mathematics, University of Bonn, Endenicher Allee~60, 53115 Bonn, Germany}
\email{lienstromberg@iam.uni-bonn.de}
\email{pernas@iam.uni-bonn.de (Corresponding author)}
\email{velazquez@iam.uni-bonn.de}

\begin{abstract}
We study the dynamic behaviour of two viscous fluid films confined between two concentric cylinders rotating at a small relative velocity. It is assumed that the fluids are immiscible and that the volume of the outer fluid film is large compared to the volume of the inner one. Moreover, while the outer fluid is considered to have constant viscosity, the rheological behaviour of the inner thin film is determined by a strain-dependent power-law.
Starting from a Navier--Stokes system, we formally derive evolution equations for the interface separating the two fluids. Two competing effects drive the dynamics of the interface, namely, the surface tension and the shear stresses induced by the rotation of the cylinders.  
When the two effects are comparable, the solutions behave, for large times, as in the Newtonian regime. 
We also study the regime in which the surface tension effects dominate the stresses induced by the rotation of the cylinders. In this case, we prove local existence of positive weak solutions both for shear-thinning and shear-thickening fluids. In the latter case, we show that interfaces which are initially close to a circle converge to a circle in finite time and keep that shape for later times.
\end{abstract}
\vspace{4pt}







\maketitle
\bigskip

\noindent\textsc{MSC (2010): 76A05, 76A20, 35B40, 35Q35, 35K35, 35K65;}

\noindent\textsc{Keywords: Taylor--Couette flow, non-Newtonian fluid, power-law fluid, degenerate parabolic equation, weak solution, long-time asymptotics, thin-film equation}
\bigskip

\bigskip

\section{Introduction}

Taylor--Couette flows describe the dynamics of viscous fluids confined between two concentric cylinders.
By the end of the 19th century, M. Couette experimentally observed that the fluid flow is steady when the relative velocity of the rotating cylinders is small and the gap between the cylinders is small compared to their radii. This is the so-called Couette flow. In 1923, G. I. Taylor proved mathematically that the Couette flow becomes unstable as soon as the relative angular velocity of the cylinders exceeds a certain critical value  \cite{Taylor}. The more the relative angular velocity of the cylinders is increased, the more turbulent becomes the behaviour of the flow.

There is a rich literature dealing with the dynamics of one single Newtonian fluid between two concentric rotating cylinders, both in the mathematical and the physical literature, c.f. \cite{Baumert,Chandra,Chossat,Drazin,Schlichting,renardy}, to mention only a few contributions.
Much less has been done for the two-fluid Taylor--Couette flow. The dynamics of two immiscible Newtonian fluids in a Taylor--Couette geometry has been studied in \cite{renardy} with a combination of analytical and numerical methods. In  particular, the stability of the flows for different ranges of viscosities, densities and surface tensions of the fluids in the absence of gravity is considered. The particular setting in which one of the fluids is localised in a thin layer is considered in \cite{PV1} for the case in which both fluids are Newtonian.

In the present work, we formally derive a model for the dynamics of the interface separating two immiscible viscous fluids between two concentric cylinders, where one of the fluids occupies a rather thin layer and is characterised by a non-Newtonian rheology. We also study rigorously the well-posedness of the resulting model and the long-time asymptotics of its solutions. 

Two physical assumptions are crucial for the derivation of the model. First, as in \cite{PV1}, we assume that the dynamics of the two-fluid system is described by a small perturbation of the Taylor--Couette flow for one single fluid confined between two cylinders. Second, while the outer fluid is assumed to be Newtonian, the inner fluid is assumed to be a non-Newtonian fluid with a strain-dependent viscosity $\mu$.
We consider the setting in which the inner cylinder is at rest, while the outer cylinder rotates at a fixed angular velocity. 

Originally, the dynamics of both immiscible fluids are described by a Navier--Stokes system in which gravitational effects are neglected. For different regimes of the surface tension and
under the assumption that the Reynolds number is small enough to avoid the aforementioned Taylor-instabilities, we study the formal asymptotic limit of a vanishing thickness of the inner fluid film. 
To this end, we apply the so-called lubrication approximation which has been used extensively in the literature on fluid mechanics, c.f. for instance \cite{Ockendon}. We also refer the reader to the papers \cite{giacomelli_otto,Gunther} for rigorous mathematical results concerning the derivation of the classical Newtonian thin-film equation, taking as a starting point the Navier--Stokes equations. 

The evolution equation that we derive and study in this paper has the general form
\begin{equation} \label{eq:intro1_new}
    \partial_t h + \partial_\theta 
	\left(h^2 \int_0^1 z \psi\left(\beta + 
	z\, h \bigl(\partial_\theta h + \partial^3_\theta h\bigr)\right)\, dz\right) = 0, \quad t > 0,\ \theta \in S^1=[0,2\pi],
\end{equation}
for the thickness $h=h(t,\theta)$ of the thin strain-dependent fluid film which separates the thicker Newtonian fluid from the internal cylinder.
We assume that the non-Newtonian fluid has a general strain-dependent viscosity $\mu=\mu\bigl(\tau_{\text{char}}\left\|\,\textbf{D} \textbf{u}\right\|\bigr)$, where $\textbf{u}$ is the velocity field of the inner non-Newtonian fluid, $\textbf{D} \textbf{u}=\frac{1}{2}\bigl(\nabla \textbf{u}+\nabla (\textbf{u})\e{T}\bigr)$ is the corresponding symmetric gradient and $\left\|\,\textbf{D} \textbf{u}\right\| = \sqrt{\text{tr}(|\textbf{D} \textbf{u}|^2)}$. Moreover, $\tau_{\text{char}}$ is the characteristic time of the non-Newtonian fluid, i.e. it can be thought of as a characteristic value of the strains for which the nonlinear effects in the viscosity become relevant.
The key assumption in the derivation of \eqref{eq:intro1_new} is that the function $s\mapsto\mu(|s|)s,\ s \in \R$, is strictly increasing. Then, the function $\psi$ in \eqref{eq:intro1_new} is defined by means of $\psi\bigl(\mu(|s|)s\bigr) = s$ for $s\in\R$. 
It can be seen in the derivation of \eqref{eq:intro1_new} that the evolution of the interface separating the two fluids is driven by the combined action of surface tension, of the shear stress induced by the rotation of the outer cylinder, and of the characteristic stress of the non-Newtonian rheology. We are interested in the interaction of these forces and on their influence on the structure of the evolutionary equation for the thickness $h$ of the non-Newtonian fluid. Different scaling limits for these three effects are encoded in the function $\psi$ and the parameter $\beta$, respectively. 

We first comment on the choice of the function $\psi$.
If the effect of surface tension is comparable with the characteristic stress of the non-Newtonian fluid, we can derive and study \eqref{eq:intro1_new} for general smooth functions $\psi$. If either surface tension dominates the characteristic stress of the non-Newtonian fluid or vice versa, we chose the function $\psi$ such that
\begin{equation}\label{eq:intropsi}
  \psi(s) = |s|^\frac{1-p}{p} s, \quad s \in \R,\ p>0,
\end{equation} 
in order to allow for an appropriate time scaling.

We mention that the definition of $\psi$ in \eqref{eq:intropsi} does also correspond to the case in which the function $\mu$ characterising the strain-dependent viscosity of the non-Newtonian fluid is given by 
\begin{equation}\label{eq:intromu}
    \mu(\left|s\right|) = \left|s\right|^{p-1}, \quad s \in \R.
\end{equation}
Fluids whose rheology is defined by \eqref{eq:intromu} are called Ostwald--de Waele fluids. The parameter $p$ denotes the flow behaviour exponent. These fluids are Newtonian if $p=1$. For $p > 1$ the corresponding fluids are called shear-thickening as their viscosity increases with increasing shear rate. Conversely, if $p<1$, the viscosity decreases with increasing shear rate and the fluids are called shear-thinning fluids.

The parameter $\beta$ in \eqref{eq:intro1_new} measures the ratio of surface tension and shear forces induced by the rotation of the cylinders and plays a crucial role in our analysis. 

In the regime in which the surface tension, the shear stress induced by the rotation of the outer cylinder, and the characteristic stress of the non-Newtonian fluid are of the same order, we have that $\beta > 0$ is a positive constant that depends on the radii of the two cylinders, their relative velocity and the characteristic viscosities of the two fluids.  In these cases the derivation of \eqref{eq:intro1_new} is valid for general smooth functions $\mu$ and $\psi$, respectively, under the sole assumption that the map $s\mapsto\mu(|s|)s,\ s \in \R$, is strictly increasing.
Using center manifold theory, as developed for instance in \cite{Mielke,centermanif}, we prove in this paper that solutions to \eqref{eq:intro1_new} behave, for long times, after a suitable rescaling of the variables, in a manner analogous to the solutions 
in the Newtonian case  ($\mu\equiv 1$ and $\psi(s)\equiv s$ or $p=1$) in which \eqref{eq:intro1_new} reduces to the equation
\begin{equation}\label{eq:intro_Newtonian}
    \partial_t h + \partial_\theta\left(\frac{h^2}{2}\right) +\de{}\bigl(h\e{3}(\de{}h+\de{3}h)\bigr)=0, \quad t > 0,\ \theta \in S^1.
\end{equation}
Equation \eqref{eq:intro_Newtonian} is studied in \cite{PV1}, where the authors observe the same asymptotic behaviour for initial interfaces close to a circle. In particular, it is shown that in the Newtonian case the solution is globally defined and the interface approaches quickly a circle which is initially not concentric with the rotating cylinders. For larger times, the center of this circle spirals towards the common center of the cylinders as time tends to infinity. The equation \eqref{eq:intro_Newtonian} has also been obtained in \cite{Kerchman} describing the motion of a single thin fluid layer evolving on the exterior of a solid cylinder.

By a straightforward adaptation of the methods used in \cite{PV1}, we prove in this paper that solutions to \eqref{eq:intro1_new} feature the same asymptotic behaviour as solutions to \eqref{eq:intro_Newtonian}. More precisely, we prove that if the interface is initially close to a circle in $H^1(S^1)$, it is globally defined and converges in $H^1(S^1)$ to a circle at rate $1/t$ which is not concentric with the rotating cylinders. However, the center of this circle spirals at rate $1/\sqrt{t}$ to the common center of the cylinders as time $t$ tends to infinity.


If either surface tension effects dominate the effects of the characteristic stresses of the non-Newtonian rheology or vice versa, we have that $\psi$ is given by \eqref{eq:intropsi}. For both settings we study the case
when the effects of surface tension dominate the shear effects due to the rotation of the cylinders. This corresponds to the asymptotic limit $\beta \to 0$, and we obtain the evolution equation

\begin{equation} \label{eq:intro_power_law_beta=0}
    \partial_t h + \partial_\theta 
	\left(h(\theta)^{\frac{1+2p}{p}}  \left|\partial_\theta h(\theta) + \partial^3_\theta h(\theta)\right|^\frac{1-p}{p} \bigl(\partial_\theta h(\theta) + \partial^3_\theta h(\theta)\bigr)\right) = 0, \quad t > 0,\ \theta \in S^1.
\end{equation}
In this regime, the effect of the shear forces induced by the rotation of the cylinders is negligible and the whole dynamics of the interface is driven by the combination of surface tension and the non-Newtonian rheology of the thin fluid film.

We prove local existence of positive weak solutions to \eqref{eq:intro_power_law_beta=0} for general positive initial data for both shear-thinning and shear-thickening fluids. 
Moreover, in the shear-thickening regime ($p>1$), we show that if the initial interface is close to a circle, there exists a global weak solution of \eqref{eq:intro_power_law_beta=0} with the property that it converges to a circle in finite time $t^\ast < \infty$. The proof is based on the derivation of a differential inequality for a certain energy functional which implies that the energy drops down to zero as $t\to t^\ast$. The main 
obstruction in the proof of the global existence result is the possibility of the interface touching the interior cylinder,
i.e. to have $\min h\left(t,\cdot\right) =0$ at some positive time $t$. In this paper, we consider only solutions of positive thickness $h$ since we are interested in the analysis of the stability of circular interfaces. In the shear-thinning case $p<1$, we expect the solutions to \eqref{eq:intro_power_law_beta=0} to approach asymptotically a circle as $t\to \infty$ with a correction given by the power-law $\mathcal{O}(t\e{-\frac{p}{1-p}})$. However, since the techniques required to obtain this result differ much from the ones used in this paper, we do not consider this case here. 

Thin-film equations the solutions of which allows for film rupture have been extensively studied in the literature, c.f.  \cite{BF:1990,BBD:1995,Kn:2015, Gi:1999}, to mention only a selection. Global well-posedness for non-negative initial data has first been proved in the seminal paper \cite{BF:1990}. This topic has additionally been pursued in \cite{BBDK:1994}, where the authors also study numerically the existence of singularities in finite and infinite time. The existence of global in time weak solutions to \eqref{eq:intro_Newtonian}  which allow for film rupture has been studied in \cite{Taranets} for a cylindrical geometry.


We remark that the surface tension forces tend to drive the interface towards a circular shape. On the contrary, the shear induced by the rotation of the
cylinders has the tendency to generate 'fingering' and to form interfaces which differ much from a circular interface. In this paper, we consider only situations in which, for large times, the contribution due to the surface tension dominates the contribution due to the shear induced by the rotation of the cylinders. Therefore, for large times the
interfaces behave asymptotically as a circle.

Note that when the shape of the interface becomes close to a circle, then the surface tension forces, reflected in the term $(\partial_\theta h + \partial_\theta^3 h)$, do no longer yield a considerable effect on the dynamics. Therefore, for sufficiently long times, the parameter $\beta$ reflecting the shear stress induced by the rotating cylinders gives the main contribution to the deformation of the interface. Consequently, for large times the solution behaves always as in the Newtonian case, with a viscosity coefficient depending on $\beta$. This indicates in particular that, when the shape of the interface becomes close to a circle, then the model with $\beta=0$ is not a good approximation anymore. Moreover, even if the interface is not close to a circle, the term $(\partial_\theta h + \partial_\theta^3 h)$ induced by the surface tension vanishes at some points. Near those points, the effect of the term $\beta$, reflecting the shear forces, becomes the dominant one. Consequently, this might lead to small localised effects in the solution and to the creation of boundary layer regions in which the shear stress is dominant. However, these questions are addressed in future works.

The evolution equations \eqref{eq:intro1_new}, respectively \eqref{eq:intro_power_law_beta=0}, belong to a class of non-Newtonian thin-film equations with strain-dependent viscosity. Similar equations have been studied in different settings for instance in \cite{AG:2002,AG:2004,K:2001a,King:2001b,LM:2020}.
In \cite{AG:2004}, the authors consider a single thin film occupied by a power-law fluid. The governing equation is 
\begin{equation} \label{eq:AG:2004}
    \partial_t h + \partial_x
	\left(h^\frac{2p+1}{p}  \left|\partial^3_x h\right|^\frac{1-p}{p}\right) = 0, \quad t > 0,\ x \in \Omega \subset \R.
\end{equation}
Note that this equation is very similar to the equation \eqref{eq:intro_power_law_beta=0} with $\beta=0$, except that in \eqref{eq:AG:2004} the nonlinearity depends only on the third-order derivative, instead of $(\de{}h+\de{3}h)$. The authors use a two-step regularisation scheme to prove global existence of non-negative weak solutions to \eqref{eq:AG:2004} for general non-negative initial data.
The papers \cite{AG:2002} and \cite{LM:2020} deal with an Ellis thin-film equation instead of a power-law thin-film equation. This is a constitutive rheological law for shear-thinning fluids which combines a power-law behaviour with a Newtonian plateau, cf. \cite{MBB:1965,SW:1994}. In \cite{AG:2002}, the authors analyse a class of quasi-self-similar solutions describing the spreading of a droplet, whose thickness $h$ is determined by the equation
\begin{equation}\label{eqself}
    \partial_t h + \partial_x\left(h^3 \left(1 + \bigl|h \partial_x^3h\bigr|^\frac{1-p}{p}\right) \partial_x^3 h\right) = 0, \quad t > 0,\, x \in \Omega \subset \R.
\end{equation}
For these solutions the presence of the non-Newtonian rheology plays a fundamental role removing the well-known no-slip paradox which arises for Newtonian fluids in the presence of contact lines. 
However, the non-Newtonian terms become negligible except in a small region close to the contact lines. 
In \cite{LM:2020}, the authors prove local existence of strong solutions to \eqref{eqself} in the case $p \in (1/2,1)$, in which the coefficients of the highest-order terms depend only H\"older continuously on the solution. 

For two-phase thin-film equations we refer the reader to the works \cite{Belinchon,Escher,Laurencot} dealing with the Newtonian case. Moreover, a wide variety of two-fluid viscous flows in many different geometrical settings is described in \cite{renardybook}.

The introduction is closed by a brief outline of our work. In Section \ref{sec:modelling}, we use lubrication approximation to formally derive the evolution equations \eqref{eq:intro1_new} and \eqref{eq:intro_power_law_beta=0}, respectively, for the interface separating the strain-dependent thin fluid film from the Newtonian fluid film. At the end of the section, we discuss the different asymptotic limits, reflected in the choice of the function $\psi$ and the parameter $\beta$. The resulting evolution equations are analysed in Sections \ref{sec:different_time_scales} and \ref{sec:comparable_timescales}. More precisely, the asymptotic limit $\beta=0$ is treated in Section \ref{sec:different_time_scales}. In Section \ref{sec:local_existence} we prove local existence of positive weak solutions in the shear-thinning as well as in the shear-thickening regime. In Section \ref{sec:global_existence} we prove, for initial interfaces close to a circle, the existence of a global weak solution of \eqref{eq:intro_power_law_beta=0} with the property that it converges to a circle in finite time. Finally, in Section \ref{sec:comparable_timescales} we study the equation for $\beta > 0$ of order one. We use center manifold theory, to prove the aforementioned convergence to a circular interface for long times.

\section{Physical model and derivation of the equations}\label{sec:modelling}

In this section we describe the physical setting of our problem and derive the evolution equations \eqref{eq:intro1_new} and \eqref{eq:intro_power_law_beta=0} for the interface separating the two fluids. These evolution equations, which contain a parameter $\beta$ that can take the value $\beta>0$ and $\beta=0$ in different asymptotic limits, are analysed rigorously in the subsequent sections. 

\bigskip

\subsection{Navier--Stokes system for one Newtonian and one non-Newtonian fluid}

We consider two immiscible fluid films confined between two concentric cylinders rotating at different angular velocities. More precisely, we denote by $R_-, R_+ > 0$ the radius of the internal and external cylinder, respectively, both cylinders being centered at the origin. The hydrodynamic behaviour of the two fluids can be described by the Navier--Stokes equations
\begin{equation}\label{eq:Navier-Stokes}
	\begin{cases}
		\rho_- \left(\tilde{\textbf{u}}^-_{\tilde{t}} + (\tilde{\textbf{u}}^- \cdot \nabla) \tilde{\textbf{u}}^-\right)
		=
		- \nabla \tilde{p}^-
		+ 2\nabla\cdot \bigl(\tilde{\mu}_-\bigl(\left\| \tilde{\textbf{D}}^- \tilde{\textbf{u}}^-\right\|\bigr) \tilde{\textbf{D}}^-\tilde{\textbf{u}}^-\bigr)
		&
		\text{in } \tilde{\Omega}_-(\tilde{t})
		\\
		\rho_+ \left(\tilde{\textbf{u}}^+_{\tilde{t}} + (\tilde{\textbf{u}}^+ \cdot \nabla) \tilde{\textbf{u}}^+\right)
		=
		- \nabla \tilde{p}^+ + \mu_+ \Delta \tilde{\textbf{u}}^+
		&
		\text{in } \tilde{\Omega}_+(\tilde{t})
		\\
		\nabla\cdot \tilde{\textbf{u}}^{\pm} 
		=
		0
		&
		\text{in } \tilde{\Omega}_\pm(\tilde{t}),
	\end{cases}
\end{equation}
where $\tilde{\textbf{u}}^\pm(\tilde{t},\tilde{x}) = (\tilde{u}^\pm(\tilde{t},\tilde{x}),\tilde{v}^\pm(\tilde{t},\tilde{x}))$ denotes the velocity field at time $\tilde{t} > 0$ and position $\tilde{x} \in \R^2$, $\tilde{p}^\pm$ is the pressure and $\rho_{\pm} \geq 0$ is the density of the inner (-), respectively outer (+) fluid.
The fluid next to the internal cylinder is assumed to be non-Newtonian with a shear-dependent viscosity $\tilde{\mu}_-\bigl(\left\|\tilde{\textbf{D}}^- \tilde{\textbf{u}}^-\right\|\bigr) > 0$. Here $\tilde{\textbf{D}}^- \tilde{\textbf{u}}^-=\frac{1}{2}\bigl(\nabla \tilde{\textbf{u}}^- + (\nabla \tilde{\textbf{u}}^-)^T\bigr)$ denotes the symmetric gradient of the velocity field $\tilde{\textbf{u}}^-$ of the inner fluid and $\left\|\tilde{\textbf{D}}^- \tilde{\textbf{u}}^-\right\| = \sqrt{\text{tr}(|\tilde{\textbf{D}}^- \tilde{\textbf{u}}^-|^2)}$. We assume that the shear stress is monotonically increasing in the shear rate, i.e. the function $s\mapsto \mu(s)s$ is monotonically increasing on $\R$. The fluid next to the external cylinder is assumed to be Newtonian with constant viscosity $\mu_+ > 0$. In order to describe the spatial position between the two cylinders we use polar coordinates $\tilde{\textbf{x}} = (\tilde{x},\tilde{z}) = (r \cos \theta, r \sin \theta) \in \R^2$. Thus, denoting by $d > 0$ the average height of the inner fluid film and by $h(\tilde{t},\theta) > 0$ the function defining the interface, the regions filled by the respective fluid may be described by
\begin{equation*}
    \begin{cases}
        \tilde{\Omega}_-(\tilde{t}) = \left\{\tilde{\textbf{x}} \in \R^2; R_- < r < R_- + d h(\tilde{t},\theta)\right\} & \\
        \tilde{\Omega}_+(\tilde{t}) = \left\{\tilde{\textbf{x}} \in \R^2; R_- + d h(\tilde{t},\theta) < r < R_+\right\}. &
    \end{cases}
\end{equation*}
Note that we assume the function $h(\tilde{t},\theta)$ being strictly positive, i.e. the interface of the two fluid films cannot touch the inner cylinder. A sketch of the problem setting may be found in Figure \ref{figdom}.
\begin{center}
   \begin{figure}[h]
    \center
    \includegraphics[width=90mm]{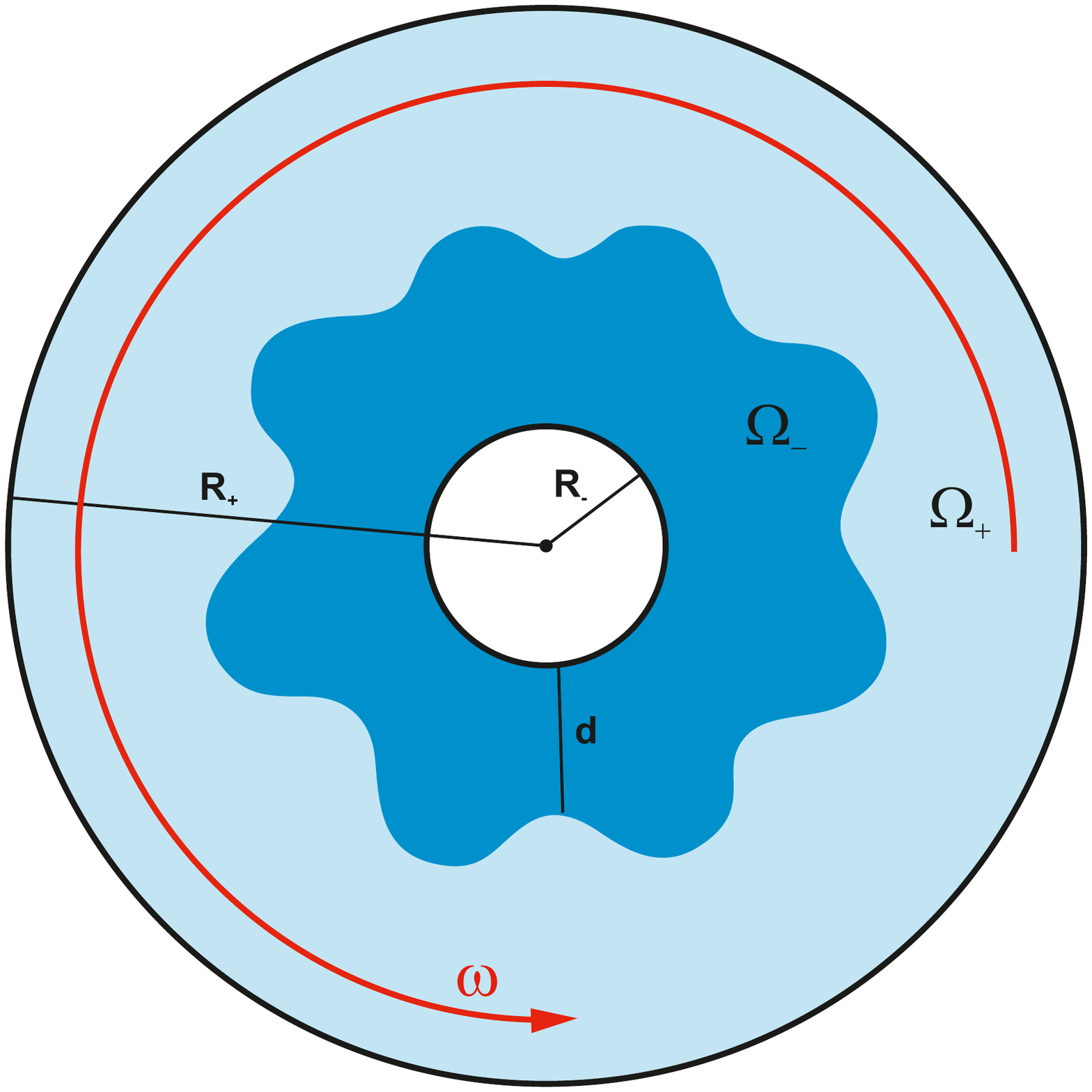}
    \caption{Taylor-Couette flow for two fluids}
    \label{figdom}
    \end{figure}
\end{center}
The Navier--Stokes system \eqref{eq:Navier-Stokes} is complemented by the following boundary conditions. We suppose that the internal cylinder is at rest, while the external cylinder rotates counterclockwise at angular velocity $\omega > 0$. Moreover, we assume that the fluid velocities $\tilde{\textbf{u}}^\pm$ of the two fluids coincide at the respective cylinders with the angular velocities at which the respective cylinder rotates. That is, we have the boundary conditions
\begin{equation*}
    \begin{cases}
    \tilde{\textbf{u}}^-
	=
	0,
    &
	\tilde{\textbf{x}} \in \partial B_{R_-}(0)
	\\
	\tilde{\textbf{u}}^+
	=
	\omega (-\tilde{x}_2,\tilde{x}_1), 
	&
	\tilde{\textbf{x}} \in \partial B_{R_+}(0).
	\end{cases}
\end{equation*}
Moreover, we assume that at the interface $\partial \tilde{\Omega}$ the normal velocities of the fluids coincide with the normal velocity $V_n$ of the interface and the tangential velocities of the fluids coincide, i.e.
\begin{equation*}
    \begin{cases}
        \tilde{\textbf{u}}^-\cdot \tilde{\textbf{n}}
    	=
	    \tilde{\textbf{u}}^+\cdot \tilde{\textbf{n}}
	    =
	    V_n, &
	    \tilde{\textbf{x}} \in \partial\tilde{\Omega}
	    \\
	    \tilde{\textbf{u}}^-\cdot \tilde{\textbf{t}}
	    =
	    \tilde{\textbf{u}}^+\cdot \tilde{\textbf{t}},
	    &
	    \tilde{\textbf{x}} \in \partial\tilde{\Omega}.
    \end{cases}
\end{equation*}
Finally, denoting by $\tilde{\Sigma}^\pm(\tilde{\textbf{u}},\tilde{p})$ the stress tensor of the respective fluid, we require the tangential stress balance condition and the normal stress balance condition to be satisfied at the interface $\partial \tilde{\Omega}$. These conditions read
\begin{equation*}
    \begin{cases}
        \tilde{\textbf{t}} \left(\tilde{\Sigma}^2 - \tilde{\Sigma}^1\right) \cdot \tilde{\textbf{n}} = 0,
	    &
	    \tilde{\textbf{x}} \in \partial\tilde{\Omega}
	    \\
	    \tilde{\textbf{n}} \left(\tilde{\Sigma}^2 - \tilde{\Sigma}^1\right) \cdot \tilde{\textbf{n}} = \tilde{\gamma}\tilde{\kappa},
	    & 
	    \tilde{\textbf{x}} \in \partial\tilde{\Omega}.
    \end{cases}
\end{equation*}
Here, we use the notation $\tilde{\textbf{n}}$ and $\tilde{\textbf{t}}$ for the normal vector pointing from the region $\tilde{\Omega}_-$ occupied by the inner fluid to the region $\tilde{\Omega}_+$ occupied by the outer fluid and the tangential vector at the interface, respectively. Furthermore, $\tilde{\kappa}$ denotes the mean curvature of the interface and $\tilde{\gamma}$ is the constant surface tension.


\bigskip
\noindent\textsc{The dimensionless Navier--Stokes system. } 
In this paper we assume that the rheology of the non-Newtonian fluid is given by a viscous coefficient $\mu_-\bigl(\tau\left\|\textbf{D}^- \textbf{u}^-\right\|\bigr) = \mu_0 \tilde{\mu}_-\bigl(\tau_{\text{char}}\left\|\tilde{\textbf{D}}^- \tilde{\textbf{u}}^-\right\|\bigr)$. The function $s \mapsto \mu_-(s)$ is a function that can describe very complicated nonlinear behaviours. The parameter $\mu_0$ is the characteristic viscosity of the fluid when $\tau_{\text{char}}\left\|\tilde{\textbf{D}}^- \tilde{\textbf{u}}^-\right\|$ is of order one.
Moreover, the parameter $\tau_{\text{char}}$ is the characteristic time of the non-Newtonian fluid that must have unit of time for dimensional reasons. Under this assumption on $\tilde{\mu}_-$ we have that the viscous stresses are given by $ \mu_0\tilde{\mu}_-\bigl(\tau_{\text{char}}\left\|\tilde{\textbf{D}}^- \tilde{\textbf{u}}^-\right\|\bigr) =\frac{\mu_0}{\tau_\text{char}}\Lambda(\tau_\text{char}\tilde{\textbf{D}}\e{-}\tilde{\textbf{u}}\e{-})$, where $\Lambda(\textbf{A})=\mu_-(\norm{\textbf{A}})\textbf{A}$ with $\textbf{A}\in M_{3\times 3}(\R)$. Note that the order of magnitude of these viscous stresses is $\frac{\mu_0}{\tau_{\text{char}}}$ if $\tau_{\text{char}}\left\|\tilde{\textbf{D}}^- \tilde{\textbf{u}}^-\right\|$ is of order one.

We now rescale the original variables in order to obtain the above Navier--Stokes system in dimensionless form. To this end, we set 
\begin{equation}\label{eq:scaling}
\begin{split}
	\textbf{x} &= \frac{\tilde{\textbf{x}}}{R_-}, 
	\quad
	t = \omega \tilde{t},
	\quad 
	\textbf{u}^\pm = \frac{\tilde{\textbf{u}}^\pm}{\omega R_-},
	\quad 
	p^\pm = \frac{\tilde{p}^\pm}{\rho_+ \omega^2 R_-^2},
	\quad 
	\gamma = 
	\frac{\tilde{\gamma}}{\rho_+ R_-^3 \omega^2},
	\quad
	\eps = \frac{d}{R_-},
	\quad
	\eta = \frac{R_+}{R_-}
	\\
	\rho &= \frac{\rho_-}{\rho_+},
	\quad
	\text{Re} = \frac{\rho_+ \omega R_-^2}{\mu_+},
	\quad
	\mu_-\bigl(\tau \left\|\textbf{D}^-\textbf{u}^-\right\|\bigr) = \frac{1}{\mu_0} \tilde{\mu}_-\bigl(\tau_{\text{char}} \left\|\tilde{\textbf{D}}^-\tilde{\textbf{u}}^-\right\|\bigr),
	\quad
	\tau=\tau_{\text{char}}\omega,
	\quad
	\mu = \frac{\mu_0}{\mu_+}.
\end{split}
\end{equation}
Consequently, $\eps > 0$ is the dimensionless thickness of the inner fluid film. By $\text{Re} \geq 0$ we denote the Reynolds number which defines the ratio of inertial to viscous forces. 

The external characteristic time $\tau$ is induced by the relative angular velocity of the rotating cylinders. Observe that the system is now scaled such that the internal cylinder has radius $1$, while the external cylinder has radius $\eta$ and rotates with angular velocity $1$, c.f. Figure \ref{figdomthin}. Note that the non-dimensional quantities $\rho, \text{Re}, \tau, \eta$ and $\mu$ do not have indices.
\begin{center}
  \begin{figure}[h]
    \center
    \includegraphics[width=90mm]{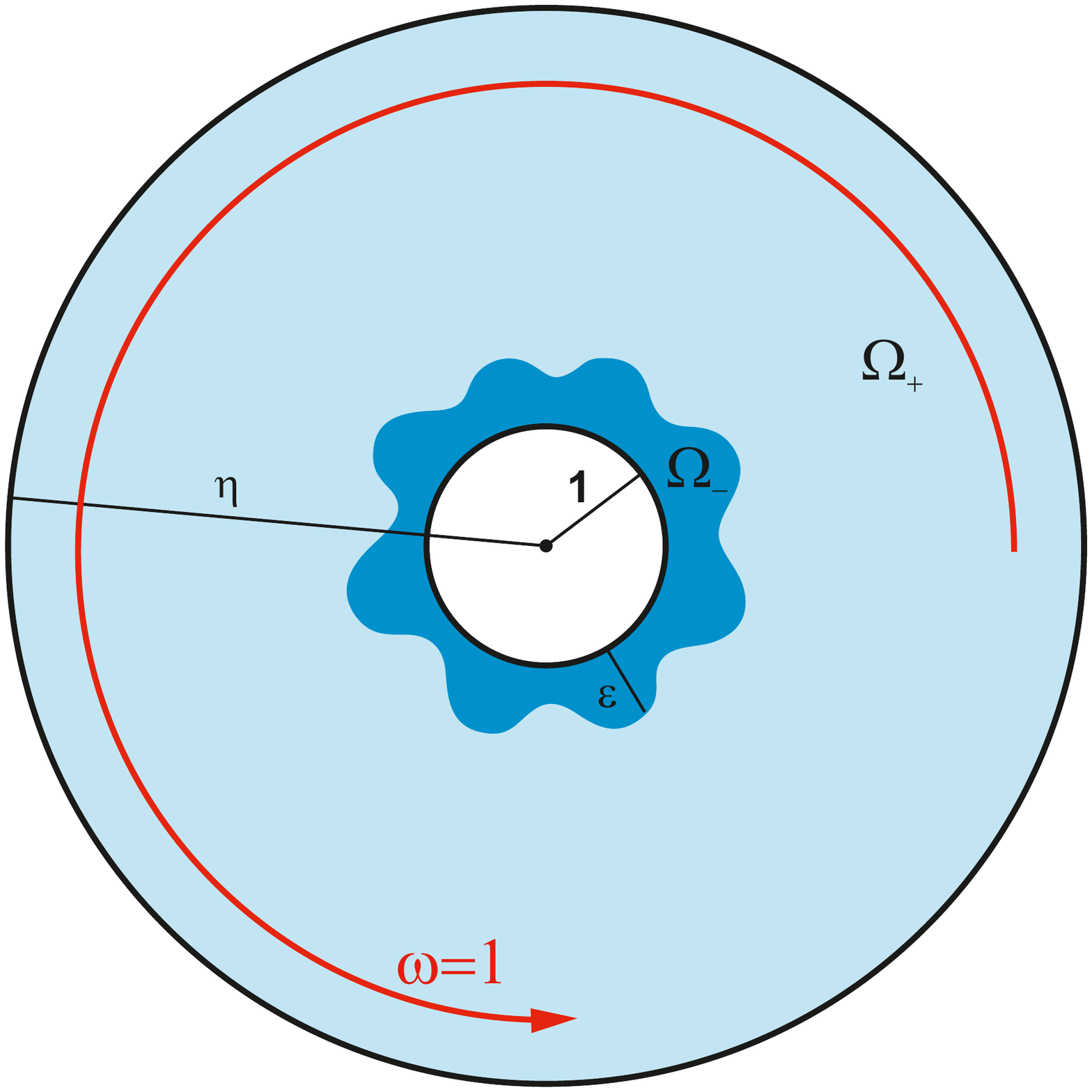}\caption{Taylor-Couette flow with a thin-film layer near the internal cylinder (in non-dimensional units)}
    \label{figdomthin}
\end{figure}
\end{center}
With this change of variables, the original Navier--Stokes system \eqref{eq:Navier-Stokes} becomes
\begin{equation*}
	\begin{cases}
		\rho \left(\textbf{u}^-_t + (\textbf{u}^-\cdot \nabla) \textbf{u}^-\right)
		=
		- \nabla p^- 
		+\frac{2\mu}{\text{Re}} \nabla\cdot \bigl( \mu_-\bigl(\tau \left\|\textbf{D}^-\textbf{u}^-\right\|\bigr) \textbf{D}^-\textbf{u}^-\bigr)
		&
		\text{in } \Omega_-(t)
		\\
		\bigl(\textbf{u}^+_t + (\textbf{u}^+\cdot \nabla) \textbf{u}^+\bigr)
		=
		- \nabla p^+ 
		+\frac{1}{\text{Re}} \Delta \textbf{u}^+
		&
		\text{in } \Omega_+(t)
		\\
		\nabla\cdot \textbf{u}^\pm
		=
		0
		&
		\text{in } \Omega_\pm(t),
	\end{cases}
\end{equation*}
where the regions $\Omega_-(t)$ and $\Omega_+(t)$, filled by the inner, respectively the outer fluid, are now given by
\begin{equation*}
    \begin{cases}
        \Omega_-(t) = \left\{\textbf{x} \in \R^2;\ 1 < r < 1 + \eps h(t,\theta)\right\} &
	    \\
	    \Omega_+(t) = \left\{\textbf{x} \in \R^2;\ 1 + \eps h(t,\theta) < r < \eta\right\}. &
    \end{cases}
\end{equation*}
The dimensionless boundary conditions read
\begin{equation*}
    \begin{cases}
	    \textbf{u}^-
	    =
	    0,
	    &
	    \textbf{x} \in \partial B_1(0)
	    \\
    	\textbf{u}^+
    	=
    	(-x_2,x_1),
	    &
	    \textbf{x} \in \partial B_{\eta}(0)
	    \\
	    
	    \textbf{u}^-\cdot \textbf{t}
	    =
	    \textbf{u}^+\cdot \textbf{t},
	    &
	    \textbf{x} \in \partial\Omega\\\textbf{u}^-\cdot \textbf{n}
	    =
	    \textbf{u}^+\cdot \textbf{n}
	    =
	    V_n,
	    &
	    \textbf{x} \in \partial\Omega
	    
	    \\
	    \textbf{t} \left(\Sigma^+ - \Sigma^-\right) \cdot \textbf{n} = 0,
	    &
	    \textbf{x} \in \partial\Omega
	    \\
	    \textbf{n} \left(\Sigma^+ - \Sigma^-\right) \cdot \textbf{n} = \gamma \kappa,
	    &
	    \textbf{x} \in \partial\Omega.
	\end{cases}
\end{equation*}


\subsection{Taylor--Couette flow and thin-film approximation}
As already mentioned in the introduction we are interested in the case in which the volume of the liquid film $\Omega_-(t)$ next to the internal cylinder is rather small compared to the volume of the film $\Omega_+(t)$ next to the external cylinder. Mathematically this corresponds to the asymptotic limit $\eps = \frac{d}{R_-} \to 0$.\footnote{In \cite{PV1},  the case in which the layer of fluid closer to the external cylinder  is much thinner than the inner one, has also been considered in the Newtonian case.  However, since the analysis is similar we restrict ourselves in  this paper to the case in which the thin layer is close to the internal cylinder.} 
Taking this limit and using formal matched asymptotic expansions, we are able to derive explicit expressions for the pressure as well as for the velocity field of each of the fluids. Consequently, we are left with a single equation for the  for the interface $h$ separating the two fluids.

Since the numerical results in \cite{renardy}, as well as the analytic results in \cite{PV1} show that the laminar flow solution, i.e. the concentric circle centered at the origin, is stable for small Reynolds number, we require $\text{Re} = \frac{\rho_+ \omega R_-^2}{\mu_+}$ to be of order one, but small enough to avoid the appearance of the Taylor instabilities. In addition, we require the parameters $\mu = \frac{\mu_0}{\mu_+}, \eta = \frac{R_+}{R_-}$ and $\rho = \frac{\rho_-}{\rho_+}$ to be of order one. 
\bigskip

\noindent\textsc{The dimensionless Navier--Stokes system in polar coordinates. } 
In order to perform the formal asymptotic analysis, we first introduce polar coordinates $\textbf{x} = (r \cos \theta, r \sin \theta)$ and write the velocity fields $\textbf{u}^\pm$ as 
\begin{equation*}
    \textbf{u}^\pm = u^\pm_r(r,\theta) \textbf{e}_r + u^\pm_\theta(r,\theta) \textbf{e}_\theta
    \quad \text{with} \quad
    \textbf{e}_r = (\cos \theta, \sin \theta)
    \quad \text{and} \quad
    \textbf{e}_\theta = (-\sin \theta,\cos \theta).
\end{equation*}
Componentwise the conservation of momentum equations for the fluid next to the inner cylinder, i.e. in $\Omega_-(t)$ in these variables read
\begin{equation}\label{eq:NS_polar_inner}
	\begin{cases}
		\rho \left(\partial_t u^-_r + u_r^- \partial_r u^-_r + \frac{1}{r} u^-_\theta \partial_\theta u^-_r 
		- \frac{1}{r} (u^-_\theta)^2\right) 
		= & \\
		- \partial_r p^- 
		+ \frac{\mu}{\text{Re}} 
		\left(
		2 \partial_r \mu_- \partial_r u^-_r 
		+ 
		\mu_- 
		\left(\partial_r\left[\frac{1}{r} \partial_r(r u^-_r)\right] + \frac{1}{r^2} \partial^2_\theta u^-_r 
		- \frac{2}{r^2} \partial_\theta u^-_\theta\right)\right. &
		\\
		\left.
		+
		\partial_\theta \mu_- \left[\partial_r\left(\frac{1}{r} u^-_\theta\right) + \frac{1}{r^2}\partial_\theta u^-_r\right]\right)
		\quad
		\text{in } \Omega_-(t)
		&\\
		\rho \left(\partial_t u^-_\theta + u_r^- \partial_r u^-_\theta + \frac{1}{r} u^-_\theta \partial_\theta u^1_\theta 
		+ \frac{1}{r} u^-_r u^-_\theta\right) 
		= &\\
		- \frac{1}{r}\partial_\theta p^- 
		+ \frac{\mu}{\text{Re}} 
		\left(
		2 \partial_\theta \mu_- \left(\frac{1}{r^2}\partial_\theta u^-_\theta + \frac{1}{r^2} u^-_r\right)\right.
		&\\
		\left.
		+ 
		\mu_- 
		\left(\partial_r\left[\frac{1}{r} \partial_r(r u^-_\theta)\right] + \frac{1}{r^2} \partial^2_\theta u^-_\theta 
		+ \frac{2}{r^2} \partial_\theta u^-_r\right)
		+
		\partial_r \mu_- \left(r \partial_r\left(\frac{1}{r} u^-_\theta\right) + \frac{1}{r} \partial_\theta u^-_r\right)\right)
		\quad
		\text{in } \Omega_-(t).
	\end{cases}
\end{equation}
The conservation of momentum equation for the outer fluid in $\Omega_+(t)$ becomes, also componentwise,
\begin{equation}\label{eq:NS_polar_outer}
    \begin{cases}
		\rho \left(\partial_t u^+_r + u_r^+ \partial_r u^+_r + \frac{1}{r} u^+_\theta \partial_\theta u^+_r 
		- \frac{1}{r} (u^+_\theta)^2\right) 
		=
		- \partial_r p^+ 
		+ \frac{1}{\text{Re}} 
		\left(\partial_r\left[\frac{1}{r} \partial_r(r u^+_r)\right] + \frac{1}{r^2} \partial^2_\theta u^+_r 
		- \frac{2}{r^2} \partial_\theta u^+_\theta\right)
		&
		\text{in } \Omega_+(t)
		\\
		\rho \left(\partial_t u^+_\theta + u_r^2 \partial_r u^+_\theta + \frac{1}{r} u^+_\theta \partial_\theta u^+_\theta 
		+ \frac{1}{r} u^+_r u^+_\theta\right) 
		=
		- \frac{1}{r}\partial_\theta p^+ 
		+ \frac{1}{\text{Re}}  
		\left(\partial_r\left[\frac{1}{r} \partial_r(r u^+_\theta)\right] + \frac{1}{r^2} \partial^2_\theta u^+_\vartheta 
		+ \frac{2}{r^2} \partial_\theta u^+_r\right)
		&
		\text{in } \Omega_+(t),
	\end{cases}
\end{equation}
and the continuity equation transforms into
\begin{equation}\label{eq:NS_continuity_polar}
		\partial_r (r u^\pm_r) + \partial_\theta u^\pm_\theta
		=
		0
		\quad \text{in } \Omega_\pm(t).
\end{equation}
Finally, the boundary conditions in polar coordinates are given by
\begin{equation}\label{eq:NS_boundary_polar}
    \begin{cases}
    u^-_r(1,\theta) = 0, \quad u^-_\theta(1,\theta) = 0 & \\
      u^-_r \eps \partial_\theta h + u^-_\theta (1+\eps h) = u^+_r \eps \partial_\theta h + u^+_\theta (1+\eps h)& \\
    u^-_r (1+\eps h) - u^-_\theta \eps \partial_\theta h = u^+_r (1+\eps h) - u^+_\theta \eps \partial_\theta h = \eps(1+\eps h) \partial_t h & \\
    \left(\sigma_{rr}]^+_- - \sigma_{\theta\theta}]^+_-\right) \eps \partial_\theta h (1+\eps h) + \sigma_{r\theta}]^+_- \left((1+\eps h)^2 - \eps^2 \partial_\theta h\right) = 0 & \\
    \sigma_{rr}]^+_- (1+\eps h)^2 + \sigma_{\theta\theta}]^+_- \eps^2 (\partial_\theta h)^2 - 2 \sigma_{r\theta}]^+_- (1+\eps h) \eps \partial_\theta h =\gamma\kappa. &
    \end{cases}
\end{equation}

For convenience, we perform another change of variables and set 
\begin{equation}\label{eq:change_of_var}
	 u^\pm_r = \eps^2 w^\pm_{\xi},
	 \quad
	 u^\pm_\theta = \eps w^\pm_\theta,
	 \quad
	 p^\pm = \frac{1}{\eps} P^\pm, 
	 \quad \text{and} \quad
	 \xi = \frac{r-1}{\eps}.
\end{equation}
Therewith the system \eqref{eq:NS_polar_inner}--\eqref{eq:NS_boundary_polar} transforms as follows:
The conservation of momentum equation for the inner fluid in $\Omega_-(t)$ becomes
\begin{equation*}
	\begin{cases}
		&\rho \left(\eps^4 \partial_t w^-_\xi + \eps^5 w^-_\xi \partial_\xi w^-_\xi 
		+ \frac{\eps^5}{1+\eps\xi}  w^-_\theta \partial_\theta w^-_\xi 
		- \frac{\eps^4}{1+\eps\xi} (w^-_\theta)^2\right) 
		=
		\\
		&
		- \partial_\xi P^- 
		+ \frac{\mu}{\text{Re}} 
		\left(
		2 \eps^2 \partial_\xi \mu_- \partial_\xi w^-_\xi
		+ 
		\mu_- 
		\left[\eps^2 \partial_\xi\left(\frac{1}{1+\eps\xi} \partial_\xi\bigl((1+\eps\xi) w^-_\xi\bigr)\right) 
		+ \frac{\eps^4}{(1+\eps\xi)^2} \partial^2_\theta w^-_\xi 
		- \frac{2\eps^3}{(1+\eps\xi)^2} \partial_\theta w^-_\theta\right]\right.
		\\
		&
		\left.
		+
		\partial_\theta \mu_- \left[\eps^2 \partial_\xi\left(\frac{1}{1+\eps\xi} w^-_\theta\right) 
		+ \frac{\eps^4}{(1+\eps\xi)^2}\partial_\theta w^-_\xi\right]\right)
		\quad
		\text{in } \Omega_1(t)
		\\
		&\rho \left(\eps^2 \partial_t w^-_\theta + \eps^3 w^-_\xi \partial_\xi w^-_\theta 
		+ \frac{\eps^3}{1+\eps\xi} w^-_\theta \partial_\theta w^-_\theta
		+ \frac{\eps^4}{1+\eps\xi} w^-_\xi w^-_\theta\right) 
		=
		\\
		&- \frac{1}{1+\eps\xi}\partial_\theta P^-
		+ \frac{\mu}{\text{Re}} 
		\left(
		2\partial_\theta \mu_- \left[\frac{\eps^2}{(1+\eps\xi)^2} \partial_\theta w^-_\theta 
		+ \frac{\eps^3}{(1+\eps\xi)^2} w^-_\xi\right]\right.
		\\
		&
		\left.
		+ 
		\mu_- 
		\left[\partial_\xi\left(\frac{ 1}{1+\eps\xi}\partial_\xi\bigl((1+\eps \xi) w^-_\theta\bigr)\right)
		+ \frac{\eps^2}{(1+\eps\xi)^2} \partial^2_\theta w^-_\theta 
		+
		\frac{2\eps^3}{(1+\eps\xi)^2} \partial_\theta w^-_\xi\right]\right.
		\\
		&
		\left.
		+
		\partial_\xi \mu_- \left[ (1+\eps\xi) \partial_\xi\left(\frac{1}{1+\eps\xi} w^-_\theta\right) + \frac{\eps^3}{1+\eps\xi} \partial_\theta w^-_\xi\right]\right)
		\quad
		\text{in } \Omega_-(t).
	\end{cases}
\end{equation*}
Similarly, for the outer fluid, we obtain the conservation of momentum equations
\begin{equation*}
	\begin{cases}
		&\eps^4 \partial_t w^-_\xi + \eps^5 w^+_\xi \partial_\xi w^+_\xi 
		+ \frac{\eps^5}{1+\eps\xi} w_\theta^+ \partial_\theta w^+_\xi 
		- \frac{\eps^4}{1+\eps\xi} (w^+_\theta)^2
		=
		\\
		&
		- \partial_\xi P^+ 
		+ \frac{1}{\text{Re}} 
		\left(\eps^2 \partial_\xi\left[\frac{1}{1+\eps\xi} \partial_\xi\bigl((1+\eps\xi) w^+_\xi\bigr)\right] 
		+ \frac{\eps^4}{(1+\eps\xi)^2} \partial^2_\theta w^+_\xi 
		- \frac{2\eps^3}{(1+\eps\xi)^2} \partial_\theta w^+_\theta\right)
		\quad
		\text{in } \Omega_+(t)
		\\
		&\eps^2 \partial_t w^+_\vartheta + \eps^3 w^+_\xi \partial_\xi w^+_\theta 
		+ \frac{\eps^3}{1+\eps\xi} w^+_\theta \partial_\theta w^+_\theta
		+ \frac{\eps^4}{1+\eps\xi}w^+_\xi w^+_\theta
		=
		\\
		&- \frac{1}{1+\eps\xi}\partial_\theta P^+
		+ \frac{1}{\text{Re}} 
		\left(\partial_\xi\left[\frac{1}{1+\eps\xi} \partial_\xi\bigl((1+\eps\xi) w^+_\theta\bigr)\right] 
		+ \frac{\eps^2}{(1+\eps\xi)^2} \partial^2_\theta w^+_\theta 
		+\frac{2\eps^3}{(1+\eps\xi)^2} \partial_\theta w^+_\xi\right)
		\quad
		\text{in } \Omega_+(t).
	\end{cases}
\end{equation*}
The continuity equation in the new variables reads
\begin{equation*}
	\partial_\xi w^\pm_\xi + \eps \partial_\xi (\xi w^\pm_\xi) + \partial_\theta w^\pm_\theta
	=0
	\quad \text{in } \Omega_\pm(t),
\end{equation*}
and the boundary conditions transform into
\begin{equation}\label{eq:velocity_interface}
    \begin{cases}
	w_\xi^-(0,\theta) = 0, \quad w_\theta^-(0,\theta) = 0 
	&\\	
	\eps^2 w_\xi^- \partial_\theta h  +  w_\theta^- (1+\eps h)
	=
	\eps^2 w_\xi^+ \partial_\theta h  +  w_\theta^+(1+\eps h)\\
	\eps w_\xi^- (1+\eps h) - \eps w_\theta^- \partial_\theta h
	=
	\eps w_\xi^+ (1+\eps h) - \eps w_\theta^+ \partial_\theta h
	=
	(1+\eps h) \partial_t h
	&\\
	\bigl(\sigma_{\xi\xi}]^+_- - \sigma_{\theta\theta}]^+_-\bigr) \eps \partial_\theta h (1+\eps h)
	+
	\sigma_{\xi\theta}]^+_- \left((1+\eps h)^2 - \eps^2 (\partial_\theta h)^2\right)
	=
	0
    &	\\
	\sigma_{\xi\xi}]^+_- (1+\eps h)^2 
	+  
	\sigma_{\theta\theta}]^+_- \eps^2 (\partial_\theta h)^2
	-
	2 
	\sigma_{\xi\theta}]^+_-	(1+\eps h) \eps \partial_\theta h
	=
	\left((1+\eps h)^2 - \eps^2 (\partial_\theta h)^2\right) \gamma \kappa.
	\end{cases}
\end{equation}
For the elements of the stress tensors we obtain
\begin{equation*}
\begin{cases}
	\sigma_{\xi\xi}^\pm
	=
	-\frac{1}{\eps} P^\pm + \frac{2\eps}{\text{Re}} c^\pm \partial_{\xi} w_{\xi}^\pm
	&\\
	\sigma_{\xi\theta}^\pm
	=
	\frac{1}{\text{Re}} c^\pm \left(\partial_{\xi} w_{\theta}^\pm - \frac{\eps}{1+\eps \xi} w_{\theta}^\pm
	+ \frac{\eps^2}{1+\eps \xi} \partial_\theta w_\xi^\pm\right)
	&\\
	\sigma_{\theta\theta}^\pm
	=
	-\frac{1}{\eps} P^\pm + \frac{2}{\text{Re}(1+\eps \xi)} c^\pm 
	\left(\eps \partial_{\theta} w_\theta^\pm + \eps^2 w_\xi^\pm\right), &
\end{cases}
\end{equation*}
with $c^- = \mu \mu_-(\tau|\partial_\xi w_\theta^-|)$ and $c^+ = 1$.
Note that $\kappa$ is the rescaled mean curvature of the interface $r = 1 + \eps h(t,\vartheta)$, given by
\begin{equation} \label{eq:curvature}
	\kappa
	=
	\frac{2 \eps^2 (\partial_\theta h)^2 - \eps(1+\eps h) \partial_\theta^2 h + (1+\eps h)^2}
	{\left((1+\eps h)^2 - \eps^2 (\partial_\theta h)^2\right)^\frac{3}{2}}.
\end{equation}
In order to determine the equation for the evolution of the interface $h$, separating the two fluids, we keep only the terms of order one in $\eps$ in the system derived above.
\bigskip

\noindent\textsc{The leading order system. } In this paragraph we consider the formal asymptotic limit $\eps \to 0$. In the literature this limiting process is also referred to as lubrication approximation. For a rigorous justification of the lubrication approximation in the Newtonian case we refer the reader to the work \cite{GO:2003}.
Taking the formal asymptotic limit $\eps \to 0$, we obtain the following system. The Navier--Stokes equations reduce to
\begin{equation}\label{eq:LOS}
	\begin{cases}
		\partial_\xi P^\pm = 0 
		& 
		\text{in } \Omega_\pm(t)
		\\
		-\partial_\theta P^\pm +\frac{1}{\text{Re}} \partial_\xi (c^\pm \partial_\xi w_\theta^-) = 0
		& 
		\text{in } \Omega_\pm(t)
		\\
		\partial_\xi w_\xi^\pm + \partial_\theta w_\theta^\pm = 0
		& 
		\text{in } \Omega_\pm(t).
	\end{cases}
\end{equation}
Moreover, for the boundary conditions we obtain
\begin{equation}\label{eq:boundary_cond:limit}
	\begin{cases}
	    w_\xi^-(0,\theta) = 0, \quad w_\theta^-(0,\theta) = 0
	    & \text{for } \theta \in S^1 \\
	    w_\theta^-
	    =
	    w_\theta^+ 
	    &\text{on } \Gamma\\
	    w^-_\xi  =
	    w^+_\xi 
	    & \text{on } \Gamma
	   
	    \\
	   	\frac{1}{\text{Re}} c^\pm \partial_{\xi} w_{\theta}^\pm]^+_-
	    =
	    0 & \text{on } \Gamma
	    \\
	    -P^\pm]^+_- 
	    =
	    \eps\gamma\kappa & \text{on } \Gamma,
	\end{cases}
\end{equation}
where we used the coefficients of the stress tensors in the leading order. Note that we keep the $\eps$ in the last boundary condition in \eqref{eq:boundary_cond:limit} since we will choose $\gamma$, depending on $\eps$, later. The particular choice is made in such a way that the contribution of the term $\eps\gamma\kappa$ is of order one.
\bigskip

\noindent\textsc{Determination of the pressure and the velocity field by matched asymptotic expansions. }
With this reduced system we are able to determine explicit expressions for the pressure $P^\pm$ and the velocity field $(w_\xi^\pm,w_\theta^\pm)$. Indeed, from $\eqref{eq:LOS}_1$ we deduce $P^+ = P^+(\theta)$.
Thus, integrating $\eqref{eq:LOS}_2$ twice with respect to $\xi$ yields
\begin{equation*}
	w^+_\theta(\xi,\theta) = \frac{\text{Re}}{2} \partial_\theta P^+(\theta) \xi^2 + A^+(\theta) \xi + B^+(\theta)
\end{equation*}
for functions $A^+ = A^+(\theta)$ and $B^+=B^+(\theta)$ that are to be determined.
In view of the conservation of mass equation $\eqref{eq:LOS}_3$ we are finally able to derive an equation for $w_\xi^+(\xi,\theta)$. Summarising, for the Newtonian fluid next to the external boundary we obtain
\begin{equation}\label{eq:velocity+}
	\begin{cases}
		w_\theta^+(\xi,\theta)
		=
		\frac{\text{Re}}{2} \partial_\theta P^+(\theta) \xi^2 + A^+(\theta) \xi + B^+(\theta)
		& 
		\\
		w_\xi^+(\xi,\theta)
		=
		-\frac{\text{Re}}{6} \partial_\theta^2 P^+(\theta) \xi^3 - \partial_\theta A^+(\theta) \frac{\xi^2}{2}
		- \partial_\theta B^+(\theta) \xi + C^+(\theta).
		& 
	\end{cases}
\end{equation}
For the fluid film next to the internal cylinder we proceed similarly. Recall that the fluid filling $\Omega_-(t)$ is assumed to be non-Newtonian. To leading order its viscosity is a function $\mu_- = \mu_-(\tau|\partial_\xi w_\theta^-|)$. In order to derive a well-posed parabolic equation for the interface separating the two fluids, we assume the shear stress $\mu_-(\tau|\partial_\xi w_\theta^-|)\partial_\xi w_\theta^-$ to be a monotonically increasing function of the shear rate $\partial_\xi w_\theta^-$.
In the leading-order approximation of the Navier-Stokes system this yields
\begin{equation}\label{eq:reduced-}
	\begin{cases}
		\partial_\xi P^- = 0 
		& 
		\text{in } \Omega_-(t)
		\\
		-\partial_\theta P^- 
		+ \frac{\mu}{\text{Re}} \partial_\xi \left(\mu_-(\tau|\partial_\xi w_\theta^-|)\partial_\xi w_\theta^-\right)
		= 0
		& 
		\text{in } \Omega_-(t)
		\\
		\partial_\xi w_\xi^- + \partial_\theta w_\theta^- = 0
		& 
		\text{in } \Omega_-(t).
	\end{cases}
\end{equation}
As for the outer fluid, the first equation implies that $P^- = P^-(\theta)$. 
Thus, by integration of $\eqref{eq:reduced-}_2$ with respect to $\xi$ we obtain 
\begin{equation*}
    \mu_-(\tau|\partial_\xi w_\theta^-|)\partial_\xi w_\theta^- = \frac{\text{Re}}{\mu} \partial_\theta P^-(\theta) \xi + A^-(\theta), \quad
    (\xi,\theta) \in \Omega_{-}(t).
\end{equation*}
Since $s \mapsto \mu_-(|s|)s$ is monotonically increasing on $\R$, we can define a function $\psi$ such that $\psi(\mu_-(|s|)s) = s$. Hence, for all $\xi < h$ we have
\begin{equation*}
    \partial_\xi w_\theta^- = \frac{1}{\tau}\psi\left(\frac{\tau \text{Re}}{\mu} \partial_\theta P^-(\theta) \xi + \tau A^-(\theta)\right).
\end{equation*}
Integration of this equation with respect to $\xi$ and exploiting the boundary condition $\eqref{eq:boundary_cond:limit}_1$ we obtain
\begin{equation}\label{eq:velocity-}
	w_\theta^{-}(\xi,\theta)
	=
	\frac{1}{\tau} \int_0^\xi \psi\left(\frac{\tau \text{Re}}{\mu} \partial_\theta P^-(\theta) s + \tau A^-(\theta)\right)\, ds.
\end{equation}

We can now use the boundary conditions in order to determine $A^\pm(\theta), B^+(\theta)$ and $C^+(\theta)$ by matched asymptotics. 
Since the non-Newtonian liquid film next to the internal cylinder is very thin compared the Newtonian fluid film,
we may assume the velocity of the outer fluid being a small perturbation of the Taylor--Couette flow for one single fluid confined between the two cylinders. This means we assume that the angular velocity, before the change of variables, is given by
\begin{equation*}
	u^+_{\theta}(r) = D_1 r + \frac{D_2}{r}
	\quad
	\text{with}
	\quad
	D_1 = \frac{\eta^2}{\eta^2-1}
	\quad
	\text{and} 
	\quad
	D_2 = -D_1.
\end{equation*}
Here we used that the radius of the internal cylinder is $1$, the radius of the external cylinder is $\eta$ and the external cylinder is rotating at angular velocity $\omega=1$, while the internal cylinder is at rest. The second-order Taylor series expansion of the Taylor--Couette flow $u^+_{\theta}(r)$ around $r=1$ is given by
\begin{equation*}
	u^+_{\theta}(r) 
	=
	(D_1 - D_2) (r-1) + D_2 (r-1)^2 + \mathcal{O}\bigl((r-1)^3\bigr).
\end{equation*}
With the change of variables introduced in \eqref{eq:change_of_var} this becomes
\begin{equation*}
	w^+_{\theta}(\xi) 
	=
	(D_1 - D_2) \xi + D_2 \eps \xi^2 + \mathcal{O}\bigl(\eps^2 \xi^3\bigr).
\end{equation*}
Thus, matching \eqref{eq:velocity+} with $w^+_\theta$, we get 
\begin{equation*}
	\frac{\text{Re}}{2} \partial_\theta P^+(\theta) \xi^2 + A^+(\theta) \xi + B^+(\theta)
	=
	w^+_\theta(\xi,\theta)
	=
	(D_1 - D_2) \xi + D_2 \eps \xi^2 + \mathcal{O}\bigl(\eps^2 \xi^3\bigr),
\end{equation*}
and consequently, $A^+(\theta) = D_1 - D_2 = 2\eta^2/(\eta^2-1)$.
Moreover, since the pressure in the Taylor--Couette flow is constant, that is $\partial_\theta P^+(\theta) = 0$, we have
\begin{equation*}
	w^+_\theta(\xi,\theta) = (D_1 - D_2) \xi + B^+(\theta) =\frac{2\eta^2}{\eta^2-1} \xi + B^+(\theta). 
\end{equation*}
Next, we determine $A^-(\theta)$. To this end, observe that 
\begin{equation*}
	\partial_\xi w^+_\theta(\xi,\theta) = (D_1 - D_2) = \frac{2\eta^2}{\eta^2-1},
\end{equation*}
thanks to the fact that the pressure $P^+$ is constant. Therefore, the tangential-stress balance condition $\eqref{eq:boundary_cond:limit}_4$, given by
\begin{equation*}
    \frac{1}{\text{Re}}\bigl(\partial_\xi w^+_\theta - \mu \mu_-\bigl(\tau|\partial_\xi w^-_\theta|\bigr) \partial_\xi w^-_\theta\bigr) = 0
    \quad \text{on } \Gamma,
\end{equation*}
yields, to leading order,
\begin{equation} \label{eq:A_1}
	A^-(\theta) = \frac{D}{\mu} - \frac{\text{Re}}{\mu} \partial_\theta P^-(\theta) h(\theta),
\end{equation}
where we set $D = D_1 - D_2 = 2D_1 = 2\eta^2/(\eta^2 - 1) > 0$.
Using this expression in $\eqref{eq:velocity-}$ yields

\begin{equation}\label{eq:w_theta^-}
	w^-_\theta(\xi,\theta)
	=
	\frac{1}{\tau}
	\int_0^\xi
	\psi\left(\frac{\tau D}{\mu} - \frac{\tau \text{Re}}{\mu} \partial_\theta P^1 \bigl(h(\theta) - s\bigr)\right)\, ds.
\end{equation}
Proceeding similarly as above, we may further determine the functions $B^+(\theta)$ and $C^+(\theta)$ by exploiting the different boundary conditions. 

However, we do not compute the explicit expressions since they are not needed in order to determine the equation for the evolution of the interface.

\bigskip

\noindent\textsc{Derivation of the evolution equation for $h$. }
We first recall from $\eqref{eq:boundary_cond:limit}_5$ that to leading order the normal-stress balance condition is $\sigma_{\xi\xi}]^+_- = \gamma \kappa$. This yields $-P^+ + P^- = \eps \gamma \kappa$.
Using the first-order Taylor approximation $\kappa = 1 - \eps(h + \partial^2_\theta h) + \mathcal{O}(\eps^2)$ of the mean curvature \eqref{eq:curvature} of the interface around $\eps = 0$ and the fact that $\partial_\theta P^+(\theta) = 0$, we obtain
\begin{equation}\label{eq:der_P_1}
	\partial_\theta P^-(\theta) 
	=
	-\eps^2 \gamma (\partial_\theta h + \partial^3_\theta h).
\end{equation}
Moreover, up to order $\eps$ the boundary condition \eqref{eq:velocity_interface} for the normal velocity of the interface may be written as
\begin{equation*}
    \partial_t h - \eps\bigl(w_\xi^-(h(\theta),\theta) - w_\theta^-(h(\theta),\theta) \partial_\theta h\bigr) = 0.
\end{equation*}
Inserting into this equation the identity
\begin{equation*}
    \partial_\theta \left(\int_0^{h(\theta)} w^-_\theta(\xi,\theta)\, d\xi\right)
    =
    w^-_\theta(h(\theta),\theta) \partial_\theta h(\theta)
	-
	w^-_\xi(h(\theta),\theta)
\end{equation*}
which follows from the conservation of mass, we obtain that the interface evolves according to the evolution equation
\begin{equation} \label{eq:evolu_h_2}
	\partial_t h+ \eps \partial_\theta \left(\int_0^{h(\theta)} w^-_\theta(\xi,\theta)\, d\xi\right) = 0, \quad
	t > 0,\ \theta \in S^1.
\end{equation}
Using the representation of $w_\theta^-(\xi,\theta)$, derived in \eqref{eq:w_theta^-}, and the representation of $\partial_\theta P^-(\theta)$, derived in \eqref{eq:der_P_1}, this equation may further be rewritten as
\begin{equation} \label{eq:PDEbeforesParam}
	\partial_t h + \frac{\eps}{\tau} \partial_\theta 
	\left(\int_0^{h(\theta)} \int_0^\xi 
	\psi\left(\frac{\tau D}{\mu} + \frac{\eps^2 \gamma \tau \text{Re}}{\mu} \bigl(\partial_\theta h + \partial_\theta^3 h\bigr) \bigl(h(\theta) -s\bigr)\right)\, ds\, d\xi\right) = 0, \quad t > 0,\ \theta \in S^1.
\end{equation}
In view of Fubini's theorem we find that
\begin{equation*} \label{eq:Fubini}
    \begin{split}
    \int_0^{h(\theta)} \int_0^\xi &
	\psi\left(\frac{\tau D}{\mu} + \frac{\tau \eps^2 \gamma \text{Re}}{\mu} \bigl(\partial_\theta h(\theta) + \partial_\theta^3 h(\theta)\bigr) \bigl(h(\theta) -s\bigr) \right)\, ds\, d\xi 
	\\&=
	\int_0^{h(\theta)} \int_s^{h(\theta)} 
	\psi\left(\frac{\tau D}{\mu} + \frac{\tau \eps^2 \gamma \text{Re}}{\mu} \bigl(\partial_\theta h(\theta) + \partial_\theta^3 h(\theta)\bigr) \bigl(h(\theta) -s\bigr)\right)\, d\xi\,ds
	\\&=
	\int_0^{h(\theta)} \bigl(h(\theta)-s\bigr)
	\psi\left(\frac{\tau D}{\mu} + \frac{\tau \eps^2 \gamma \text{Re}}{\mu} \bigl(\partial_\theta h(\theta) + \partial_\theta^3 h(\theta)\bigr) \bigl(h(\theta) -s\bigr)\right)\,ds
	\\&=
	\int_0^{h(\theta)} \tilde{s}
	\psi\left(\frac{\tau D}{\mu} + \frac{\tau \eps^2 \gamma \text{Re}}{\mu} \bigl(\partial_\theta h(\theta) + \partial_\theta^3 h(\theta)\bigr) \tilde{s}\right)\,d\tilde{s}
	\\&=
	h(\theta)^2 \int_0^1 z\, \psi\left(\frac{\tau D}{\mu} + z\, \frac{\tau \eps^2 \gamma \text{Re}}{\mu} h(\theta)\bigl(\partial_\theta h(\theta) + \partial_\theta^3 h(\theta)\bigr)\right)\,dz.
    \end{split}
\end{equation*}
Consequently, the evolution equation \eqref{eq:PDEbeforesParam} reads
\begin{equation} \label{eq:PDE2}
    \partial_t h + \frac{\eps}{\tau} \partial_\theta 
	\left(h(\theta)\e{2}\int_0^1 z\, \psi\left(\frac{\tau D}{\mu} + z\, \frac{\tau \eps^2 \gamma \text{Re}}{\mu} h(\theta)\bigl(\partial_\theta h(\theta) + \partial_\theta^3 h(\theta)\bigr)\right)\,dz\right) = 0, \quad t > 0,\ \theta \in S^1.
\end{equation}
\medskip

\noindent \textsc{The evolution equation for different scaling limits. } We now determine the evolution equation for different scaling limits of the surface tension forces and the shear forces induced by the rotation of the cylinder. To this end, we
define the parameters
\begin{equation*}
    A = \frac{\tau D}{\mu}, \quad B = \frac{\tau \eps^2 \gamma \text{Re}}{\mu} \quad \text{and} \quad \beta = \frac{A}{B} = \frac{D}{\gamma \eps^2 \text{Re}},
\end{equation*}
and rewrite the evolution equation \eqref{eq:PDEbeforesParam} in terms of $B$ and $\beta$ as
\begin{equation} \label{eq:B_beta}
    \partial_t h + \frac{\eps}{\tau} \partial_\theta \left(h(\theta)^2 \int_0^1 z\, \psi\Bigl(B \left[\beta + z\, h(\theta)\bigl(\partial_\theta h(\theta) + \partial_\theta^3 h(\theta)\bigr)\right]\Bigr)\,dz\right) = 0, \quad t > 0,\, \theta \in S^1.
\end{equation}
We recall that, in this paper, we consider only the regime in which the radii of the cylinders are of the same order but such that the two cylinders are not too close to each other. Therefore, $D > 0$ is just a non-dimensional geometrical constant.
We now discuss the structure of the equation for different ranges of the parameters $B$ and $\beta$. To this end, note first that, in physical variables, $B$ and $\beta$ are given by
\begin{equation*}
    B = \frac{\tilde{\gamma}\eps^2 \tau_{\text{char}}}{\mu_0 R_-} 
    \quad \text{and} \quad \beta = \frac{D R_- \omega \mu_+}{\eps^2 \tilde{\gamma}},
\end{equation*}
respectively, where we used the scaling for $\tilde{\gamma}, \tau, \text{Re}$ and $\mu$ introduced in \eqref{eq:scaling}. Thus, the parameter $B$ reflects the ratio of the shear forces induced by the surface tension over the characteristic shear $\frac{\mu_0}{\tau_{\text{char}}}$ associated to the non-Newtonian fluid and $\beta$ reflects the ratio of the shear forces induced by the rotating cylinder over the surface tension forces.

We first distinguish the asymptotic limits $B>0$ of order one, $B\to 0$ and $B \to \infty$, respectively. \medskip

\noindent (I) The case $B > 0$ of order one. In this case the effects of the surface tension are comparable with those of the characteristic stresses of the non-Newtonian fluid. Changing the variables via 
\begin{equation}\label{rescale}
    \tilde{h}=\sqrt{B}\,h 
    \quad \text{and} \quad
    \tilde{t}= \frac{\eps}{\tau}\frac{1}{\sqrt{B}}\,t,
\end{equation}
and then dropping the tildes for convenience, we have that the evolution equation \eqref{eq:B_beta} for the interface is given by
\begin{equation}\label{equevol_general}
    \partial_t h + \partial_\theta 
	\left(h(\theta)^2 \int_0^1 z \psi\left(\tilde{\beta} + 
	z\, h(\theta) \bigl(\partial_\theta h(\theta) + \partial^3_\theta h(\theta)\bigr)\right)\, dz\right) = 0, \quad t > 0,\ \theta \in S^1
\end{equation}
with $\tilde{\beta}=B\beta = \frac{D R_- \omega^2 \mu_+ \tau_{\text{char}}}{\mu_0}$.\medskip

\noindent (II) The cases $B\to 0$ and $B\to \infty$, respectively. The asymptotic limit $B\to 0$ corresponds to the situation in which the surface tension effects are dominated by the effects of the characteristic stresses of the non-Newtonian fluid. Conversely, the limit $B\to \infty$ represents the regime in which the  surface tension forces dominate the characteristic stresses of the non-Newtonian rheology. Suppose that the function $\psi$ is given such that
\begin{equation} \label{eq:psi}
    \psi(s) =
    \begin{cases}
        |s|^{\frac{1-p}{p}} s, \quad s \to 0, & \text{if } B\searrow 0
        \\
        |s|^{\frac{1-p}{p}} s, \quad s \to \infty, & \text{if } B\nearrow \infty
    \end{cases}
\end{equation}
where $p > 0$. Then, changing the times scale via
\begin{equation*}
    \tilde{t}= \frac{\eps}{\tau}B^\frac{1}{p}\,t
\end{equation*}
and dropping again the tilde for convenience, the evolution equation \eqref{eq:B_beta} becomes
\begin{equation}\label{equevol_power-law}
    \partial_t h + \partial_\theta 
	\left(h(\theta)^2 \int_0^1 z \left|\beta + z h(\theta) \bigl(\partial_\theta h(\theta) + \partial_\theta^3 h(\theta)\bigr)\right|^\frac{1-p}{p} \left(\beta + z h(\theta) \bigl(\partial_\theta h(\theta) + \partial_\theta^3 h(\theta)\bigr)\right)\, dz\right) = 0
\end{equation}
for all $t > 0,\ \theta \in S^1.$
\medskip

For both equations, \eqref{equevol_general} and \eqref{equevol_power-law} we now distinguish different asymptotic limits of the parameter $\beta = \frac{D R_- \omega \mu_+}{\eps^2 \tilde{\gamma}}$.\medskip

\noindent (I) The case $\frac{D R_- \omega \mu_+}{\tilde{\gamma}} \approx \eps^2$. In this asymptotic limit the surface tension forces and the shear forces induced by the rotating cylinder are comparable. Thus, we obtain the evolution equations
\begin{equation}\label{eq:CaseI_1}
    \partial_t h + \partial_\theta 
	\left(h(\theta)^2 \int_0^1 z \psi\left(\tilde{\beta} + 
	z\, h(\theta) \bigl(\partial_\theta h(\theta) + \partial^3_\theta h(\theta)\bigr)\right)\, dz\right) = 0 
\end{equation}
	and
\begin{equation}\label{eq:CaseI_2}
	\partial_t h + \partial_\theta 
	\left(h(\theta)^2 \int_0^1 z \left|\beta + z h(\theta) \bigl(\partial_\theta h(\theta) + \partial_\theta^3 h(\theta)\bigr)\right|^\frac{1-p}{p} \left(\beta + z h(\theta) \bigl(\partial_\theta h(\theta) + \partial_\theta^3 h(\theta)\bigr)\right)\, dz\right) = 0
\end{equation}
for $t > 0,\ \theta \in S^1$,
with $\tilde{\beta}, \beta > 0$ being positive constants. These two equations are studied in Section \ref{sec:comparable_timescales}, where we prove that if the initial interface is close to a circle, then the solution is globally defined and converges to a circle which is not necessarily concentric with the two cylinders. However, as time tends to infinity, the center of the circle spirals towards the common center of the cylinders.\medskip

\noindent(II) The case $\frac{D R_- \omega \mu_+}{\tilde{\gamma}} \ll \eps^2$. This corresponds to the asymptotic limit $\beta \to 0$ in which the effects of surface tension on the flow are dominating the shear effects induced by the rotation of the cylinders are negligible. The evolution equations for this setting read
\begin{equation}\label{eq:CaseII_1}
    \partial_t h + \partial_\theta 
	\left(h(\theta)^2 \int_0^1 z \psi\left(
	z\, h(\theta) \bigl(\partial_\theta h(\theta) + \partial^3_\theta h(\theta)\bigr)\right)\, dz\right) = 0, \quad t > 0,\ \theta \in S^1  
\end{equation}
and
\begin{equation}\label{eq:CaseII_2}
	\partial_t h + \partial_\theta 
	\left(h(\theta)^\frac{2p-1}{p} \left| \partial_\theta h(\theta) + \partial_\theta^3 h(\theta)\right|^\frac{1-p}{p} \bigl(\partial_\theta h(\theta) + \partial_\theta^3 h(\theta)\bigr)\right) = 0, \quad t > 0,\ \theta \in S^1,
\end{equation}
respectively.
Equation \eqref{eq:CaseII_2}, corresponding to the function $\psi$ defined in \eqref{eq:psi}, is studied in Section \ref{sec:different_time_scales}. We prove existence of positive weak solutions for short times. For $p>1$, we show that solutions that are originally close to a circle, converge to a circle in finite time. We recall that, as discussed in the introduction, in the regions where $\de{}h+\de{3}h$ is small, boundary layer effects can arise. 
\medskip

\noindent(III) The case $\frac{D R_- \omega \mu_+}{\tilde{\gamma}} \gg \eps^2$. This reflects the situation in which the shear stress induced by the rotation of the cylinders dominates the surface tension such that $\beta \to \infty$. We remark that this asymptotic limit is not studied in the present paper.

\bigskip

\noindent\textsc{Power-law fluids. }
 In this paper we are particularly interested in the case in which the viscous behaviour of the thin non-Newtonian fluid film is governed by a power-law. That is, for the effective viscosity $\mu_-$ we use the constitutive law
\begin{equation}\label{eq:power-law}
	\mu_-(\tau|\partial_\xi w^-_\theta|)
	=
	\tau^{p-1}|\partial_\xi w^-_\theta|^{p-1}
\end{equation}
with $p > 0$.
Fluids with such a viscosity are usually called power-law fluids or Ostwald--de Waele fluids.
Recall that a flow-behaviour exponent $p=1$ corresponds to a Newtonian fluid. Moreover, 
for $p<1$ the fluid is shear-thinning, while it is shear-thickening for $p>1$.

\bigskip
\section{The case $\beta \to 0$ -- Existence result and asymptotic behaviour} \label{sec:different_time_scales}

In this section we deal with the asymptotic limit $\beta \to 0$ in which we have derived the approximation \eqref{eq:CaseII_2}. We first prove local in time existence of weak solutions in the shear-thickening, as well as in the shear-thinning regime. For the latter case, we then study the asymptotic behaviour of solutions that are initially not too far from a circle. We observe that they converge to a circle in finite time and then continue to exist as a circle forever. The center of the circle is not necessarily the origin, differently from the case in which $\beta$ is of order one that is discussed in Section \ref{sec:comparable_timescales}. 

We recall that, as discussed in the introduction,  equation \eqref{eq:CaseII_2} cannot expected to be a good approximation of \eqref{equevol_power-law} if $\abs{(\partial_\theta h + \partial_\theta^3 h)}\ll\beta$. Since for a circular interface we have that $(\partial_\theta h + \partial_\theta^3 h)=0$, and the model \eqref{eq:CaseII_2} predicts that the interface becomes a circle in finite time, it follows that \eqref{eq:CaseII_2} cannot describe the solutions of \eqref{equevol_power-law} for long times. Therefore, \eqref{eq:CaseII_2} describes only the intermediate asymptotics of the interfaces when they are not yet very close to circles.

Before proving local existence of positive weak solutions, we briefly introduce the notation used throughout the paper. We identify $S^1$ with the interval $[0,2\pi]$. Moreover, we identity functions $\phi \in L_p(S^1)$ with functions $\phi \in L_{p,\text{loc}}(\R)$ which are periodic with period $2\pi$. Here, $L_p(\R)$ denotes the usual Lebesgue space. 
Finally, by $W^k_p(S^1)$ and $H\e{k}(S\e{1})$ we denote the usual Sobolev spaces. They are defined as the closure of the restriction of $2\pi$-periodic functions in $C\e{\infty}(\R)$ to the interval $[0,2\pi]$ with respect to the norm in $W^k_p((0,2\pi))$, respectively $H\e{k}((0,2\pi))$. 
In order to simplify notation, we consider $W^k_p(S^1)$ and $H^k(S^1)$ as a closed subspaces of the complex spaces $W\e{k}_p(S\e{1};\C)$ and  $H\e{k}(S\e{1};\C)$, respectively. In particular, we can represent any $f\in H\e{k}(S\e{1})$ by its Fourier series
\begin{equation}\label{espacioswzeta}
	f(\theta)=\sum_{n\in\Z} a_n e\e{in\theta},
	\quad \theta \in S^1,\quad\text{with}\quad a_n=\overline{a_n}.
\end{equation}


\subsection{Local existence of positive weak solutions} \label{sec:local_existence}

In this section we prove existence of local weak solutions to the problem
\begin{equation}\tag{P}\label{eq:pde}
\begin{cases}
\partial_t h + \de{}\left(h^{\alpha+2} \left|\de{}h+\de{3}h\right|^{\alpha-1} \bigl(\de{}h+\de{3}h\bigr)\right)
=
0,
&
t > 0,\, \theta \in S^1
\\
h(0,\cdot)
=
h_0(\cdot), &
\theta \in S^1,
\end{cases}
\end{equation}
with periodic boundary conditions and for all flow behaviour exponents $\alpha > 0$. We use the notation
\begin{equation*}
    \psi(s) = \left|s\right|^{\alpha-1} s,
    \quad s \in \R,
\end{equation*}
such that the evolution equation \eqref{eq:pde} can be written as
\begin{equation*}
\partial_t h + \de{}\left(h^{\alpha+2} \psi\bigl(\de{}h+\de{3}h\bigr) \right)
=
0,
\quad
t > 0,\, \theta \in S^1.
\end{equation*}
Note that this equation is a quasilinear equation of fourth order that may degenerate in $h$ and $\de{}h+\de{3}h$. In the non-degenerate case of a positive film height the equation is parabolic. Moreover, the coefficients of the highest-order terms depend only $(\alpha-1)$-H\"older-continuously on the lower-order terms. In order to prove the existence of local positive weak solutions, we follow the usual ansatz of regularising the equation and showing that the sequence of solutions to the regularised problem has an accumulation point $h$ which is a weak solution to the original problem.
The compactness arguments mainly rely on a-priori estimates that are derived from the functional
\begin{equation*}
	E[v] = \frac{1}{2} \int_{S^1} \left((\de{}v)^2 -v^2\right)\, d\theta.
\end{equation*}
Even if $E[v](t)$ is not necessarily non-negative, we refer to it as an energy functional. 
To be able to pass to the limit in the nonlinear terms we use lower semicontinuity and apply Minty's trick.  


The main result of this subsection is the following theorem on the existence of weak solutions to \eqref{eq:pde}.


\begin{theorem} \label{thm:existence}
Given an initial film height $h_0 \in  H^1(S^1)$ with $0 < C_0 \leq h_0(\theta)$ for all $\theta \in S^1$, there exist a positive time $T > 0$ and a positive weak solution $h$ of \eqref{eq:pde} on $[0,T]$ in the sense that $0 < C_1 \leq h(t,\theta) \leq C_2$ for all $t \in [0,T],\, \theta \in S^1$,
\begin{itemize}
    \item[(i)] $h$ has the regularity
        $$
            h \in L_{\alpha+1}\bigl((0,T);W^3_{\alpha+1}(S^1)\bigr) \cap C\bigl([0,T];H^1(S^1)\bigr),
            \quad
            \partial_t h \in L_\frac{\alpha+1}{\alpha}\bigl((0,T);(W^1_{\alpha+1}(S^1))'\bigr);
        $$
    \item[(ii)] $h$ satisfies the integral equation
    $$
	    \int_0^T \langle \partial_t h(t),\varphi(t)\rangle_{W^1_{\alpha+1}(S^1)}\, dt
	    =
	    \int_0^T \int_{S^1} h^{\alpha+2} \psi(\de{}h+\de{3}h) \de{}\varphi\, d\theta\, dt
	$$
	for all test functions $\varphi \in L_{\alpha+1}\bigl((0,T);W^1_{\alpha+1}(S^1)\bigr)$; 
	\item[(iii)] $h$ satisfies the initial condition $h(0,\theta) = h_0(\theta)$ for all $\theta \in S^1$.
\end{itemize}
In addition, this solution has the following properties:
\begin{itemize}
    \item[(iv)] (Conservation of mass) The mass of the fluid is conserved in the sense that
    $$
        \left\|h(t)\right\|_{L_1(S^1)} = \left\|h_0\right\|_{L_1(S^1)}
    $$
    for all $t \in [0,T]$.
    \item[(v)] (Energy dissipation) The solution dissipates energy in the sense that
    \begin{equation*}
        E[h](t) + \int_0^T \int_{S^1} \left|h\right|^{\alpha+2} \left|\de{}h+\de{3}h\right|^{\alpha+1}\, d\theta\, dt = E[h_0].
    \end{equation*}
\end{itemize}
\end{theorem}


\noindent\textsc{The mollified problem. } 
To overcome the problems caused by the degeneracy and the lack of regularity of \eqref{eq:pde}, we introduce a mollified version of \eqref{eq:pde} as follows. Let
\begin{equation*}
	\rho \in C^\infty_c(\R)
	\quad
	\text{with}
	\quad
	\int_{\R} \rho(\theta)\, d\theta = 1
\end{equation*}
be the usual mollifier such that, for $\eps \in (0,1)$, 
\begin{equation*}
	\rho_\eps(\theta) = \frac{1}{\eps} \rho\left(\frac{\theta}{\eps}\right)
	\quad
	\text{and}
	\quad
	\eta_\eps(v)(\theta) = (\rho_\eps \ast v)(\theta),
\end{equation*}
where $\ast$ denotes convolution. Note that the parameter $\eps$ in this section is a regularisation parameter that is not related to the average dimensioless height $\eps$ of the film, defined in \eqref{eq:scaling}.

Moreover, we use the notation
\begin{equation*}
	\bar{v} = \frac{1}{2\pi} \int_{S^1} v(\theta)\, d\theta 
\end{equation*}
for the average of a function $v \in L_2(S^1)$. Since our proofs strongly rely on Fourier analysis, we use this notation frequently for the zeroth Fourier mode. Therewith, for a fixed $\eps \in (0,1)$, we replace the mobility $h^{\alpha+2}$ by a function 
\begin{equation}\label{eq:def_m_eps}
	m_\eps \in C^\infty\bigl(\R;\R_{\geq 0}\bigr) 
	\quad \text{with} \quad
	m_\eps(s) = |s|^{\alpha+2} 
	\quad \text{for} \quad
	\frac{\bar{h}_0}{2} \leq |s|
	\quad \text{and} \quad 
	m_\eps(s) \leq |s|^{\alpha+2},\ s \in \R.
\end{equation}
and the nonlinear term $\psi(\partial_\theta h + \partial_\theta^3 h)= |\partial_\theta h + \partial_\theta^3 h|^{\alpha-1}(\partial_\theta h + \partial_\theta^3 h)$ by 
\begin{equation*}
	\psi_\eps(s) = \bigl(s^2 + \eps^2\bigr)^\frac{\alpha-1}{2} s.
\end{equation*}
Then, we also have $\psi_\eps \in C^\infty\bigl(\R_{\geq 0};\R_{\geq 0}\bigr)$. Finally, we introduce the regularised / mollified problem 
\begin{equation}\tag{$P_\eps$}\label{eq:pde_eps}
	\begin{cases}
	\partial_t h^\eps + \de{}\Bigl(\eta_\eps \left[m_\eps(h^\eps)\, \psi_\eps\bigl(\eta_\eps\bigl(\de{}h\e{\eps}+\de{3}h\e{\eps}\bigr)\bigr)\right]\Bigr)
	=
	0,
	&
	t > 0,\, \theta \in S^1
	\\
	h^\eps(0,\cdot)
	=
	h_0(\cdot), &
	\theta \in S^1,
	\end{cases}
\end{equation}
with periodic boundary conditions. We are interested in solving \eqref{eq:pde_eps} for initial values $h_0$ that are close to their constant average $\bar{h}_0 > 0$.

In order to solve \eqref{eq:pde} different regularisation techniques are certainly possible. In \cite{BF:1990} the authors propose two different regularisation techniques for the Newtonian thin-film equation. In both approaches the regularisation is restricted to the mobility coefficient. Moreover in \cite{AG:2004}, where the authors prove global existence for a doubly nonlinear equation the nonlinearity of which contains only the third-order derivative, a two-step regularisation is chosen. The authors regularise the mobility coefficient and introduce an artificial lower-order term, to guarantee positivity of the regularised solution on the one hand and to obtain sufficient regularity of the limit problem with the non-regularised mobility.

It follows from a standard fixed-point argument (see for instance \cite{majda}) that the regularised problem \eqref{eq:pde_eps} possesses a local solution as stated in the following theorem. 


\begin{theorem} \label{thm:existence_eps}
	Let $\eps \in (0,1)$. Given an initial value $h^\eps_0 = h_0 \in H^1(S^1)$, there exists a positive time $T_\eps > 0$, possibly depending on $\eps$, and a unique solution $h^\eps$ of \eqref{eq:pde_eps} on $[0,T_\eps]$ such that
	\begin{equation*}
	    h^\eps \in C\bigl([0,T_\eps);H^1(S^1)\bigr) \cap C\bigl((0,T_\eps);H^4(S^1)\bigr) \cap C^1\bigl([0,T_\eps);L_2(S^1)\bigr)
	\end{equation*}
	and $h^\eps$ satisfies the integral equation
	\begin{equation}\label{eq:weak_solution_eps}
	    \int_0^{T_\eps} \langle \partial_t h^\eps(t),\varphi(t)\rangle_{W^1_{\alpha+1}(S^1)}\, dt
	    =
	    \int_0^{T_\eps} \int_{S^1} 
	    \eta_\eps \Bigl(m_\eps(h^\eps)\, \psi_\eps\bigl(\eta_\eps\bigl(\de{}h\e{\eps}+\de{3}h\e{\eps}\bigr)\bigr)\Bigr)
	    \de{}\phi\,
	    d\theta\, dt
	\end{equation}
	for all test functions $\phi \in L_{\alpha+1}\bigl((0,T_\eps);W^1_{\alpha+1}(S^1)\bigr)$. 
\end{theorem}


Note that, if in Theorem \ref{thm:existence_eps} the initial value $h_0 \in H^1(S^1)$ satisfies $0 < \frac{\bar{h}_0}{2}\leq h_0 \in H^1(S^1)$, the continuity of the solution $h^\eps \in C\bigl([0,T_\eps]\times S^1\bigl)$ implies its positivity of $h^\eps$ for very small times $t > 0$. However, in general the solution $h^\eps$ does not necessarily remain positive on the whole time interval $[0,T_\eps]$ on existence, even if we require $h_0 > 0$ initially.


We now prove that the sequence $(h^\eps)_\eps$ has an accumulation point $h$ which is in turn a weak solution to the original problem \eqref{eq:pde}. To this end, note first that \eqref{eq:weak_solution_eps} may be rewritten equivalently as
\begin{equation*}
	\int_0^{T_\eps} \langle \partial_t h^\eps(t),\varphi(t)\rangle_{W^1_{\alpha+1}(S^1)}\, dt
	=
	\int_0^{T_\eps} \int_{S^1} 
	m_\eps(h^\eps)\, \psi_\eps\bigl(\eta_\eps\bigl(\de{}h\e{\eps}+\de{3}h\e{\eps}\bigr)\bigr)\,
	\partial_\theta(\eta_\eps \phi)\, d\theta\, dt
\end{equation*}
for all $\phi \in L_{\alpha+1}\bigl((0,T);W^1_{\alpha+1}(S^1)\bigr)$. We start by collecting some important properties of the solution $h^\eps$ to the regularised problem \eqref{eq:pde_eps}. To this end, we denote by $T_\eps$ be the maximal time of existence of the solution $h^\eps$ to \eqref{eq:pde_eps}. In general, $T_\eps$ depends on the parameter $\eps$.

\begin{remark} \label{rem:extension}
It is worthwhile to mention that if $h^\eps$ is defined in some space $C\bigl([0,\tau];H^1(S^1)\bigr)$ for some $\tau > 0$, then the solution can be extended to a larger time interval. In particular, this implies that $\tau < T_\eps$. We will use this result frequently in order to prove that the solutions $h^\eps$ obtained in Theorem \ref{thm:existence_eps} can be defined in some small time interval $[0,T]$, where $T > 0$ is independent of $\eps$.
\end{remark}

In the next lemma, we observe that solutions $h^\eps$ to \eqref{eq:pde_eps} conserve their mass. 


\begin{lemma} \label{lem:cons_mass}
	Let $h^\eps$ be the solution to the mollified problem \eqref{eq:pde_eps} on $[0,T_\eps)$, corresponding to the initial value $h_0$. Then $h^\eps$ conserves its mass in the sense that
	\begin{equation*}
	\left\|h^\eps(t)\right\|_{L_1(S^1)}
	=
	\|h^\eps_0\|_{L_1(S^1)},
	\quad
	t \in [0,T_\eps).
	\end{equation*}
\end{lemma}


\begin{proof}
	This follows by testing the regularised evolution equation with $\phi=1$ and using the periodic boundary conditions.
\end{proof}

In addition to the conservation of mass property, solutions $h^\eps$ to the mollified problem dissipate energy, in the sense that the functional $E$ introduced above is decreasing along solutions. 


\begin{lemma}[Energy dissipation] \label{lem:dissipation}
Let $h^\eps$ be a solution to \eqref{eq:pde_eps} on $[0,T_\eps)$, emanating from an initial value $h_0 \in H^1(S^1)$. Then $h^\eps$ complies with the functional equation
\begin{equation} \label{eq:energy_equ}
	E[h^\eps](T) + 2 D^\eps_{T}[h^\eps] = E[h^\eps](0),
	\quad
	T \in [0,T_\eps),
\end{equation}
where the non-negative dissipation $D_{T}[h^\eps]$ is given by
\begin{equation*}
	D^\eps_{T}[h^\eps] 
	= 
	\int_0^{T} \int_{S^1} 
	m_\eps(h^\eps)\,
	\left(\bigl(\eta_\eps(\partial_\theta h^\eps + \partial_\theta^3 h^\eps\bigr)^2 + \eps^2\right)^\frac{\alpha-1}{2} \eta_\eps(\partial_\theta h^\eps + \partial_\theta^3 h^\eps\bigr)^2\, d\theta\, dt.
\end{equation*}
This, in particular, implies the a-priori estimate
\begin{equation} \label{eq:diss_est}
	D^\eps_T[h^\eps] 
	\leq 
	\|h_0\|_{H^1(S^1)} + \pi \left|\bar{h}_0\right|^2
	\leq
	C \|h_0\|_{H^1(S^1)}.
\end{equation}
\end{lemma}


\begin{proof}
(i)	That $h^\eps$ satisfies the functional equation \eqref{eq:energy_equ} follows by testing the equation with $(h^\eps + \partial_\theta^2 h^\eps)$.
	
(ii) In order to prove the a-priori estimate \eqref{eq:diss_est}, we use Fourier analysis and write
\begin{equation*}
	h^\eps(t,\theta) 
	=
	\sum_{n\in \Z} a_n(t) e^{i n \theta}
	=
	\bar{h}^\eps
	+
	\sum_{n\in\Z, n \neq 0} a_n(t) e^{i n \theta},
	\quad t \in [0,T_\eps),
\end{equation*}
where the Fourier coefficients $a_n(t),\ n \in \Z$, are, for $t \in [0,T_\eps)$, given by
\begin{equation*}
	a_n(t) =
	\frac{1}{2\pi} \int_{S^1} h^\eps(t,\theta) e^{-i n \theta}\, d\theta
	\quad
	\text{and}
	\quad
	a_0(t)
	=
	\frac{1}{2\pi} \int_{S^1} h^\eps(t,\theta)\, d\theta
	=
	\bar{h}^\eps.
\end{equation*}
Using Plancherel's theorem, i.e. the identity
\begin{equation*}
	\sum_{n\in \Z} |a_n(t)|^2 
	=
	\frac{1}{2\pi} \left\|\sum_{n\in\Z} a_n(t) e^{i n \theta}\right\|_{L_2(S^1)}^2,
	\quad t \in [0,T_\eps),
\end{equation*}
we obtain 
\begin{align*}
	2 E[h^\eps](t)
	&=
	\int_{S^1} \left(\left|\de{}h\e{\eps}(t)\right|^2 - \left|h^\eps(t)\right|^2\right)\, d\theta
	\\
	&=
	\left\|\de{}h\e{\eps}(t)\right\|_{L_2(S^1)}^2 - \left\|h^\eps(t)\right\|_{L_2(S^1)}^2
	\\
	&=
	2\pi 
	\left(\sum_{n\in\Z, n \neq 0} (n^2-1) |a_n(t)|^2
	-
	\left|\bar{h}_0\right|^2\right),
	\quad t \in [0,T_\eps),
\end{align*}
and consequently, we end up with
\begin{equation*}
	E[h^\eps](t) + \pi |\bar{h}_0|^2
	=
	\pi 
	\sum_{n\in\Z, n \neq 0} (n^2-1) |a_n(t)|^2
	\geq 
	0, \quad t \in [0,T_\eps).
\end{equation*}
This yields the desired estimate and the proof is complete.
\end{proof}

\medskip

By means of the energy balance \eqref{eq:energy_equ} we can now derive a uniform (in $\eps$) $L_\infty\bigl([0,T];H^1(S^1)\bigr)$  estimate for $h^\eps$.
It is worthwhile to mention that, although the energy estimate \eqref{eq:energy_equ} is the natural estimate for the equation \eqref{eq:pde}, it does not provide any information on the Fourier modes $a_{\pm 1}$.

Throughout the paper we frequently use the following elementary inequality. Given $\eps > 0$ and $\alpha \in (0,1)$ it holds that
\begin{equation}\label{eq:elementary_inequ}
	\bigl(|x|^2 + \eps^2\bigr)^\frac{\alpha-1}{2} |x|
	=
	\bigl(|x|^2 + \eps^2\bigr)^\frac{\alpha(\alpha-1)}{2(\alpha+1)} |x|^\frac{2\alpha}{\alpha+1} 
	\frac{|x|^\frac{1-\alpha}{\alpha+1}}{\bigl(|x|^2 + \eps^2\bigr)^\frac{1-\alpha}{2(\alpha+1)}}
	\leq
	\bigl(|x|^2 + \eps^2\bigr)^\frac{\alpha(\alpha-1)}{2(\alpha+1)} |x|^\frac{2\alpha}{\alpha+1},
	\quad x \in \R.
\end{equation}


\begin{lemma}\label{lem:H^1}
Let $\eps \in (0,1)$ be fixed and  let $h^\eps$ be the solution to \eqref{eq:pde_eps} on $[0,T_\eps)$, emanating from the initial value $h^\eps_0 \in H^1(S^1)$. Then there exist a positive time $T > 0$ and a positive constant $C > 0$, both independent of $\eps$, such that
\begin{equation*}
	\left\|h^\eps\right\|_{L_\infty((0,T);H^1(S^1))}
	\leq
	C.
\end{equation*}
\end{lemma}

\medskip

The main issue in the proof of this lemma is to observe that the $H^1(S^1)$-norm of $h^\eps$ is equivalent to the sum of $E[h^\eps]$ and the low Fourier modes. In virtue of the Lemma \ref{lem:cons_mass} and Lemma \ref{lem:dissipation}, we thus need to derive estimates for the Fourier modes $n=0, \pm 1$.


\begin{proof}
As in the proof of Lemma \ref{lem:dissipation} we use the Fourier series representation of $h^\eps$ to obtain the equation
\begin{equation*}
	E[h^\eps](t)
	=
	\left(
	\left\|\de{}h\e{\eps}(t)\right\|_{L_2(S^1)}^2 - \left\|h^\eps(t)\right\|_{L_2(S^1)}^2
	\right)
	=
	- \pi \left|\bar{h}_0(t)\right|^2 + \pi \sum_{n\in\Z, n \neq 0} (n^2-1) |a_n(t)|^2
\end{equation*}
for every $t \in [0,T_\eps)$.
	
Writing $\left\|h^\eps(t)\right\|_{H^1(S^1)}$ in terms of the Fourier series of $h^\eps(t)$ yields 
\begin{equation} \label{eq:H1_est_A_1}
	\begin{split}
	\left\|h^\eps(t)\right\|_{H^1(S^1)}^2
	&=
	\left\|h^\eps(t)\right\|_{L_2(S^1)}^2 + \left\|\de{}h\e{\eps}(t)\right\|_{L_2(S^1)}^2
	\\
	&=
	2 \pi \sum_{n\in\Z} (n^2 + 1) |a_n(t)|^2
	\\
	&\leq
	C_1 \left|\bar{h}_0\right|^2 + C_2 \left(|a_1(t)|^2 + |a_{-1}(t)|^2\right) + C_3 E[h^\eps](t)
	\end{split}
\end{equation}
for all $t \in (0,T_\eps)$.
Thus, to get an estimate for $\left\|h^\eps(t)\right\|_{H^1(S^1)}$, we need estimates for the first Fourier modes $a_{-1}(t)$ and $a_1(t)$. Since $h^\eps$ is a real-valued function, we have
$a_{-1}(t) = \overline{a_1(t)}$, where the bar indicates complex conjugation. Hence, it is enough to estimate $a_1(t)$. To this end, recall that $a_1$ is given by
\begin{equation*}
	a_1(t) = \frac{1}{2\pi} \int_{S^1} h^\eps(t,\theta)\, e^{-i \theta}\, d\theta.
\end{equation*}
This immediately implies the estimate $|a_1(t)| \leq C \left\|h^\eps(t)\right\|_{L_\infty(S^1)} \leq C \left\|h^\eps(t)\right\|_{H^1(S^1)},\, t\in (0,T_\eps)$. Moreover, 
\begin{align*}
	\frac{d}{dt} a_1(t)
	&=
	\frac{1}{2\pi} \int_{S^1} \partial_t h^\eps(t,\theta)\, e^{-i \theta}\, d\theta
	\\
	&=
	\frac{1}{2\pi} \int_{S^1} 
	m_\eps(h^\eps)\,
	\psi_{\eps}\bigl(\eta_\eps\bigl(\partial_\theta h^\eps + \partial_\theta^3 h^\eps\bigr)\bigr)\, \partial_\theta\bigl(\eta_\eps e^{-i \theta}\bigr)\, d\theta
	\\
	&=
	-\frac{i}{2\pi} \int_{S^1} 
	m_\eps(h^\eps)\,
	\psi_{\eps}\bigl(\eta_\eps\bigl(\partial_\theta h^\eps + \partial_\theta^3 h^\eps\bigr)\bigr)\,\bigl(\eta_\eps e^{-i \theta}\bigr)\, d\theta.
\end{align*} 
With the elementary inequality \eqref{eq:elementary_inequ} and in view of H\"older's inequality with exponents $p=(\alpha+1)/\alpha$ and $q=\alpha+1$, we deduce the estimate
\begin{align*}
	&
	\left|\frac{d}{dt} a_1(t)\right|
	\leq
	C \int_{S^1} m_\eps(h^\eps)\,
	\left|\bigl(\eta_\eps\bigl(\partial_\theta h^\eps + \partial_\theta^3 h^\eps\bigr)\bigr)^2 + \eps^2\right|^\frac{\alpha(\alpha-1)}{2(\alpha+1)}\,
	\left|\eta_\eps\bigl(\partial_\theta h^\eps + \partial_\theta^3 h^\eps\bigr)\right|^\frac{2\alpha}{\alpha+1}
	\left|\eta_\eps e^{-i \theta}\right|\, d\theta
	\\
	&\leq
	C
	\left(\int_{S^1} m_\eps(h^\eps)\, \left|\bigl(\eta_\eps\bigl(\partial_\theta h^\eps + \partial_\theta^3 h^\eps\bigr)\bigr)^2 + \eps^2\right|^\frac{\alpha-1}{2}\,
	\left|\eta_\eps\bigl(\partial_\theta h^\eps + \partial_\theta^3 h^\eps\bigr)\right|^2
	d\theta\right)^\frac{\alpha}{\alpha+1}
	\left(\int_{S^1} m_\eps(h^\eps)\, \left|\eta_\eps e^{-i \theta}\right|^{\alpha+1}\right)^\frac{1}{\alpha+1}
	\\
	&\leq
	C\, \left\|h^\eps\right\|_{L_\infty(S^1)}^\frac{\alpha+2}{\alpha+1}
	\left(\int_{S^1} m_\eps(h^\eps)\, \left|\bigl(\eta_\eps\bigl(\partial_\theta h^\eps + \partial_\theta^3 h^\eps\bigr)\bigr)^2 + \eps^2\right|^\frac{\alpha-1}{2}\,
	\left|\eta_\eps\bigl(\partial_\theta h^\eps + \partial_\theta^3 h^\eps\bigr)\right|^2
	d\theta\right)^\frac{\alpha}{\alpha+1}
\end{align*}
for all $t \in (0,T_\eps)$.
Consequently, for the derivative with respect to time of $|a_1(t)|^2$ we find that
\begin{equation*}
\begin{split}
	\left|\frac{d}{dt} \left|a_1(t)\right|^2\right|
	&\leq
	2\, \left|a_1(t)\right|\cdot \left|\frac{d}{dt} a_1(t)\right|
	\\
	&\leq
	C\, \left\|h^\eps(t)\right\|_{H^1(S^1)}^{\beta+1} 
	\left(\int_{S^1} m_\eps(h^\eps)\, \left|\bigl(\eta_\eps\bigl(\partial_\theta h^\eps + \partial_\theta^3 h^\eps\bigr)\bigr)^2 + \eps^2\right|^\frac{\alpha-1}{2}\,
	\left|\eta_\eps\bigl(\partial_\theta h^\eps + \partial_\theta^3 h^\eps\bigr)\right|^2
	d\theta\right)^\frac{\alpha}{\alpha+1}
\end{split}
\end{equation*}
for all $t \in (0,T_\eps)$, with $\beta = \frac{\alpha+2}{\alpha+1} > 1$. Due to Young's inequality with exponents $p = \frac{\alpha+1}{\alpha}$ and $q=\alpha+1$ we find that, for all $t \in (0,T_\eps)$,
\begin{align*}
	\left|\frac{d}{dt} |a_1(t)|^2\right|
	&\leq
	\delta\,
	\int_{S^1} m_\eps(h^\eps)\, \left|\bigl(\eta_\eps\bigl(\partial_\theta h^\eps + \partial_\theta^3 h^\eps\bigr)\bigr)^2 + \eps^2\right|^\frac{\alpha-1}{2}\,
	\left|\eta_\eps\bigl(\partial_\theta h^\eps + \partial_\theta^3 h^\eps\bigr)\right|^2
	d\theta
	+
	C_{\delta,\alpha} \left\|h^\eps(t)\right\|_{H^1(S^1)}^\gamma
\end{align*}
for arbitrarily small $\delta > 0$, a constant $C_{\delta,\alpha} > 0$ that depends on $\delta$ and $\alpha$ and with $\gamma = (\alpha+2)+(\alpha+1) > 1$. Integration with respect to time yields
\begin{equation*}
	|a_1(t)|^2 
	\leq
	\delta\, D^\eps_{t}[h^\eps] + C_{\delta,\alpha}\, \int_0^t \left\|h^\eps(s)\right\|_{H^1(S^1)}^\gamma\, ds
	+
	|a_1(0)|^2,
	\quad t \in (0,T_\eps).
\end{equation*}
The estimate for $|a_{-1}(t)|^2 = |a_1(t)|^2$ is the same. In addition, Lemma \ref{lem:cons_mass} implies $|\bar{h}^\eps(t)| = |\bar{h}^\eps_0|$.
Inserting this into \eqref{eq:H1_est_A_1} yields for all $t \in (0,T_\eps)$
\begin{align*}
	\left\|h^\eps(t)\right\|_{H^1(S^1)}^2
	&\leq
	C_1 \left|\bar{h}_0\right|^2 + C_2 \left(|a_1(t)|^2 + |a_{-1}(t)|^2\right) 
	+ C_3 E[h^\eps](t)
	\\
	&\leq
	C \bigl(\left|\bar{h}_0\right|^2 + |a_1(0)|^2 + |a_{-1}(0)|^2\bigr)
	+
	\tilde{\delta} D^\eps_t[h^\eps] 
	+
	\tilde{C}_{\delta,\alpha} \int_0^t \left\|h^\eps(s)\right\|_{H^1(S^1)}^\gamma\, ds
	\\
	&\quad
	+
	C_3 E[h^\eps](t)
	\\
	&\leq
	C \bigl(\left|\bar{h}_0\right|^2 + |a_1(0)|^2 + |a_{-1}(0)|^2\bigr)
	+
	\tilde{\delta} D^\eps_t[h^\eps] 
	+
	\tilde{C}_{\delta,\alpha} \int_0^t \left\|h^\eps(s)\right\|_{H^1(S^1)}^\gamma\, ds
	\\
	&\quad
	+
	C_3 \bigl(E[h^\eps](0) - D^\eps_t[h^\eps]\bigr)
	\\
	& 
	\leq 
	C \bigl(\left|\bar{h}_0\right|^2 + |a_1(0)|^2 + |a_{-1}(0)|^2 + E[h^\eps](0)\bigr)
	+
	\tilde{C}_{\delta,\alpha} \int_0^t \left\|h^\eps(s)\right\|_{H^1(S^1)}^\gamma\, ds
	\\
	&\quad
	-(C_3 - \tilde{\delta}) D^\eps_t[h^\eps]
	\\
	&\leq
	C\, \|h^\eps_0\|_{H^1(S^1)}^2 - (C_3 - \tilde{\delta}) D^\eps_t[h^\eps]
	+
	\tilde{C}_{\delta,\alpha} \int_0^t \left\|h^\eps(s)\right\|_{H^1(S^1)}^\gamma\, ds.
\end{align*}
Choosing $\tilde{\delta}$ small enough, such that $C_3 - \tilde{\delta} = \frac{1}{2}$, we obtain
\begin{equation*}
	\left\|h^\eps(t)\right\|_{H^1(S^1)}^2
	+
	\frac{1}{2} D^\eps_t[h^\eps]
	\leq
	C_1\, \|h_0\|_{H^1(S^1)}^2 
	+
	C_2 \int_0^t \left\|h^\eps(s)\right\|_{H^1(S^1)}^\gamma\, ds,
	\quad t \in (0,T_\eps),
\end{equation*}
with $\gamma = (\alpha+2)+(\alpha+1) > 2$.
Thus, in view of Lemma \ref{lem:dissipation}, we have derived the estimate
\begin{equation*}
	\left\|h^\eps(t)\right\|_{H^1(S^1)}^2
	\leq
	C_1\, \|h_0\|_{H^1(S^1)}^2 
	+
	C_2 \int_0^t \left\|h^\eps(s)\right\|_{H^1(S^1)}^\gamma\, ds,
\end{equation*}
for $t \in (0,T_\eps)$. Finally, using a Gronwall type argument, we find that there exists a time $T > 0$, independent of $\eps$, such that
\begin{equation*}
	\left\|h^\eps(t)\right\|_{H^1(S^1)}^2
	\leq C_{T,h_0},
	\quad t \in (0,T).
\end{equation*}
In addition, we have $h^\eps \in C\bigl([0,T];H^1(S^1)\bigr)$ due to the regularity of the divergence-term in the regularised equation \eqref{eq:pde_eps}. Thus, it follows that $T \leq T_\eps$ (cf. Remark \ref{rem:extension}).
This completes the proof.
\end{proof}

\medskip

Our next goal is to prove that solutions $h^\eps$ of \eqref{eq:pde_eps}, that emerge from a positive initial film height, do not immediately drop to zero. For this purpose we need the following auxiliary result.
\medskip

\begin{lemma}\label{lem:interpolation_est}
Let $p > 1$ and $\psi \in W^2_p(S^1)$. Then for all $\delta > 0$ there exists a constant $C_\delta > 0$, such that the estimate
\begin{equation*} 
	\|\de{}\psi\|_{L_p(S^1)}
	\leq
	\delta \|\de{2}\psi\|_{L_p(S^1)}
	+
	C_\delta \|\psi\|_{L_2(S^1)}
\end{equation*}
holds true.
\end{lemma}

\begin{proof}
Let $\phi \in L_q(S^1)$ with $q > 1$ such that $\tfrac{1}{p}+\frac{1}{q}=1$. As usual, we identify $S^1$ with the interval $[0,2\pi]$ and we can identity functions $\phi \in L_q(S^1)$ with functions $\phi \in L_{q,\text{loc}}(\R)$ which are periodic with period $2\pi$. Let $\eta_\eps$ be a standard mollifier, i.e. let
\begin{equation*}
	\eta_\eps(\theta) = \frac{1}{\eps} \eta\bigl(\tfrac{\theta}{\eps}\bigr),
	\quad \text{where} \quad 
	\eta \in C_c^\infty(\R)
	\quad \text{with} \quad
	\text{supp}(\eta) \subset [-1,1]
	\quad \text{and} \quad 
	\int_{\R} \eta(\theta)\, d\theta = 1.
\end{equation*}
Moreover, we require that
\begin{equation*} 
	\eta(\theta) = \eta(-\theta)
	\quad \text{and} \quad 
	\eta \geq 0.
\end{equation*}
Since the argument of $\eta_\eps$ might be negative in some of the following calculations, we extend the function $\eta_\eps$ periodically by assuming that $\eta_\eps(\theta) = \eta_\eps(\theta+2\pi)$.
We define $\phi_\eps = \eta_\eps \ast \phi$, where $\ast$ denotes convolution. Then we have $\phi_\eps \in C^\infty(S^1)$ and
\begin{equation*}
	\phi - \phi_\eps = \partial_\theta G_\eps,\quad \text{where} \quad 
	G_\eps(\theta) = \int_0^\theta \bigl(\phi(s) - \phi_\eps(s)\bigr)\, ds.
\end{equation*}
Moreover, since $\int_{S^1} (\phi - \phi_\eps)\, d\theta = 0$, it follows that $G_\eps \in W^1_q(S^1)$ and we may use integration by parts to obtain, for every $\psi \in W^2_p(S^1)$, the equation
\begin{equation*}
		\int_{S^1} \phi\, \partial_\theta\psi\, d\theta
		=
		\int_{S^1} \bigl(\phi - \phi_\eps)\, \partial_\theta \psi\, d\theta
		+
		\int_{S^1} \phi_\eps\, \partial_\theta \psi\, d\theta
		=
		- \int_{S^1} G_\eps \partial_\theta^2 \psi\, d\theta
		- \int_{S^1} \partial_\theta \phi_\eps\, \psi\, d\theta.
\end{equation*}
In view of H\"older's inequality, this allows us to derive the estimate
\begin{equation} \label{eq:Holder_aux}
	\begin{split}
		\left|\int_{S^1} \phi\, \partial_\theta\psi\, d\theta\right|
		&\leq
		\|G_\eps\|_{L_q(S^1)} 
		\|\partial_\theta^2 \psi\|_{L_p(S^1)} 
		+
		\|\partial_\theta \phi_\eps\|_{L_2(S^1)}
		\|\psi\|_{L_2(S^1)} .
	\end{split}
\end{equation}
For the second summand on the right-hand side of \eqref{eq:Holder_aux} we have the estimate
\begin{equation} \label{eq:interpolation_2nd}
	\|\partial_\theta \phi_\eps\|_{L_2(S^1)} 
	\leq 
	C\, \|\partial_\theta \phi_\eps\|_{L_\infty(S^1)} 
	=
	C\,\sup_{\theta \in S^1}\left|\int_{S^1} \partial_\theta\eta_\eps(\theta - s)\, \phi(s)\, ds\right|
	\leq 
	C_\eps \|\phi\|_{L_q(S^1)}.
\end{equation}
Thus, in order to prove the desired inequality, we are left with estimating the first summand in \eqref{eq:Holder_aux}. To this end, we write  
\begin{equation*}
	\begin{split}
		G_\eps(\theta) 
		&=
		\int_0^\theta \int_{0}^{2\pi} \bigl(\phi(\xi) - \phi(s)\bigr)\, \eta_\eps(\xi-s)\, ds\, d\xi
		\\
		&=
		\int_{0}^{2 \pi} \chi_{\{0\leq\xi \leq \theta\}}\, \int_{0}^{2\pi} \bigl(\phi(\xi) - \phi(s)\bigr)\, \eta_\eps(\xi-s)\, ds\, d\xi
		\\
		&=
		\int_{0}^{2\pi} \phi(s)\, 
		\int_{0}^{2\pi} \bigl(\chi_{\{0\leq s \leq \theta\}} - \chi_{\{0\leq \xi \leq \theta\}}\bigr)\, \eta_\eps(\xi - s)\, d\xi\, ds
	\end{split}	
\end{equation*}
for $\theta \in S^1$, where we used symmetry condition of the mollifier.
From this equation we may then derive the estimate
\begin{equation} \label{eq:G_eps}
	|G_\eps(\theta)|
	\leq
	C
	\|\phi\|_{L_q(S^1)} \left(\int_0^{2\pi} \left|W(s,\theta)\right|^p\,ds\right)^\frac{1}{p}
\end{equation}
for $(s,\theta) \in [0,2\pi]^2$, where $W(s,\theta) = \int_0^{2\pi} \bigl|\chi_{\{s \leq \theta\}} - \chi_{\{\xi \leq \theta\}}\bigr|\, \eta_\eps(\xi - s)\, d\xi$. Using that the support of $\eta_\eps(\xi-s)$ is contained in the region $\{|\xi-s| \leq \eps, |\xi - s \pm 2\pi|\leq \eps\}$, we conclude that the support of $W(s,\theta)$ is contained in the region $\{|s-\theta|\leq \eps, |s - \theta \pm 2\pi|\leq \eps\}$. Since in addition $0 \leq W(s,\theta) \leq 1$ for all $(s,\theta) \in [0,2\pi]^2$, we find that
\begin{equation*}
	\int_0^{2\pi} \left|W(s,\theta)\right|^p\, ds
	=
	\int_{\{|s-\theta|\leq \eps, |s-\theta\pm 2\pi|\leq \eps\}} 1\, ds
	\leq
	2\eps \longrightarrow 0
	\quad \text{as} \quad \eps \to 0.
\end{equation*}
Inserting this into \eqref{eq:G_eps}, we find that the first summand in \eqref{eq:Holder_aux} may be estimated by
\begin{equation}\label{eq:interpolation_1st}
	\|G_\eps\|_{L_q(S^1)}
	\leq 
	C \|G_\eps\|_{L_\infty(S^1)} 
	\leq 
	C\, (2\eps)^\frac{1}{p} \|\phi\|_{L_q(S^1)}.
\end{equation}
Finally, inserting \eqref{eq:interpolation_2nd} and \eqref{eq:interpolation_1st} into \eqref{eq:Holder_aux}, we end up with
\begin{equation*}
	\begin{split}
		\|\phi\, \partial_\theta \psi\|_{L_1(S^1)}
		\leq
		C (2\eps)^\frac{1}{p} \|\phi\|_{L_q(S^1)}
		\|\partial_\theta^2 \psi\|_{L_p(S^1)}
		+
		C_\eps \|\phi\|_{L_q(S^1)} \|\psi\|_{L_q(S^1)}. 
	\end{split}
\end{equation*}
Since the constant $C(2\eps)^{1/p}$ becomes arbitrarily small, as $\eps \to 0$, the statement follows after choosing $\phi = (\partial_\theta \psi)^{p/q} \in L_q(S^1)$.
\end{proof}

\medskip

We are now able to prove the following lemma.

\begin{lemma} \label{lem:bound_from_zero}
Let $s \in (0,1)$. Given $\eps \in (0,1)$, let $h^\eps$ be the corresponding solution to \eqref{eq:pde_eps} on $[0,T_\eps)$ with initial value $h^\eps_0 \in H^1(S^1)$.
There exist a time $T_0 > 0$ and a function $K_s \in C(\R_+;\R_+)$ with $\lim_{t\searrow 0} K_s(t)=0$, both independent of $\eps$, such that for any $T \in (0,T_0)$ such that $T_\eps \leq T$, we have
\begin{equation*}
	\left\|h^\eps(t) - h_0\right\|_{L_\infty((0,T_\eps);H^s(S^1))}
	\leq
	K_s(T),
	\quad
	s \in [1/2,1).
\end{equation*}
\end{lemma}

\medskip


Note that, due to the embedding $H^s(S^1) \hookrightarrow L_\infty(S^1)$ for $s \geq 1/2$, the estimate obtained in Lemma \ref{lem:bound_from_zero} implies in particular that
\begin{equation*}
	\left\|h^\eps(t) - h_0\right\|_{L_\infty((0,T_\eps)\times S^1)}
	\leq
	K(T).
\end{equation*}

\begin{proof}
For convenience we work with functions $u^\eps = h^\eps - \bar{h}_0$ with zero average, that is $\int_{S^1} u^\eps(t,\theta) \, d\theta = 0$.
Given an initial value $u_0$, consider the equation
\begin{equation} \label{eq:pde_difference}
	\partial_t\left(u^\eps-u^\eps_0\right) 
	+ 
	\de{}\Bigl(\eta_\eps \left[m_\eps(u^\eps)\, \psi_\eps\bigl(\eta_\eps\bigl(\de{}u\e{\eps}+\de{3}u\e{\eps}\bigr)\bigr)\right]\Bigr) = 0,
	\quad
	t \in [0,T_\eps),\, \theta \in S^1.
\end{equation}
In order to derive suitable estimates, we test the equation with the function
\begin{equation} \label{eq:test_selfadjoint}
	\phi = \partial_\theta^{-1} S \bigl(\partial_\theta + \de{3}\bigr) \bigl(u^\eps - u^\eps_0\bigr),
\end{equation}
where the operator $S\colon L_2(S^1) \to L_2(S^1)$ is defined by 
\begin{equation*}
	Sf = \sum_{n\in \Z, n \neq 0} (Sf)_n e^{in\theta}
	\quad \text{with} \quad
	(S f)_n = \frac{1}{|n|} f_n
	\quad
	\text{for}
	\quad
	f = \sum_{n\in \Z, n\neq 0} f_n e^{i n \theta}.
\end{equation*}
Therefore, we obtain 
\begin{equation}\label{eq:operator_M}
	\bigl(\partial_\theta^{-1} S \bigl(\partial_\theta + \partial^3_{\theta}\bigr)f\bigr)_n
	=
	\frac{1}{i n} \frac{1}{|n|} (-i n^3 + i n)f_n
	=
	-\frac{1}{|n|} (n^2 - 1)f_n
	=:
	(-Mf)_n.
\end{equation}
The operator $M$ is now a nice operator in the sense that $M$ is nonnegative, self-adjoint and may be written as $M = A^2$, where $A\colon H^1(S^1) \to L_2(S^1)$ is defined by
\begin{equation*}
	Af = \sum_{n \in \Z, n \neq 0} (Af)_n e^{in\theta}
	\quad \text{with} \quad
	(Af)_n := \sqrt{\frac{n^2 - 1}{|n|}} f_n
	\quad
	\text{for}
	\quad
	f = \sum_{n\in \Z, n\neq 0} f_n e^{i n \theta}.
\end{equation*}
On the other hand, we can identify $S (\partial_\theta + \partial^3_{\theta})=H(\de{2}+I)$, where $H$ is the periodic Hilbert operator given by
\begin{displaymath}
	H(f)(\theta)
	=
	\sum_{n\in\Z}h_ne\e{in\theta}
	\quad\text{with}\quad 
	h_n=-i\, \text{sgn}(n)f_n
	\quad\text{for}\quad f(\theta)=\sum_{n\in\Z}f_n e\e{in\theta},
	\quad \theta \in S^1. 
\end{displaymath}
Testing \eqref{eq:pde_difference} with the function $\varphi = -M(u^\eps - u_0)$, introduced in \eqref{eq:test_selfadjoint}, respectively \eqref{eq:operator_M}, we obtain
\begin{equation*} 
	\begin{split}
	& \quad
	\int_{S^1} \partial_t\left(u^\eps - u_0\right) \varphi\, d\theta 
	=
	\int_{S^1}
	m_\eps(u^\eps)\, \psi_\eps\bigl(\eta_\eps\bigl(\de{}u\e{\eps}+\de{3}u\e{\eps}\bigr)\bigr)\, \eta_\eps(\de{}\varphi)\,
	d\theta
	\\
	\Longleftrightarrow 
	&\quad
	-\int_{S^1} \partial_t u^\eps M(u^\eps - u_0)\, d\theta 
	=
	\int_{S^1}
	m_\eps(u^\eps)\, \psi_\eps\bigl(\eta_\eps\bigl(\de{}u\e{\eps}+\de{3}u\e{\eps}\bigr)\bigr)\, \left[S (\partial_\theta + \de{3})\eta_\eps(u^\eps-u_0)\right]\,
	d\theta
	\\
	\Longleftrightarrow 
	&\quad
	\int_{S^1} \partial_t u^\eps A(u^\eps-u_0)\, A(u^\eps-u_0)\, d\theta 
	=
	-\int_{S^1}
	m_\eps(u^\eps)\, \psi_\eps\bigl(\eta_\eps\bigl(\de{}u\e{\eps}+\de{3}u\e{\eps}\bigr)\bigr)\, \left[S (\partial_\theta + \de{3})\eta_\eps(u^\eps-u_0)\right]\,
	d\theta
	\end{split}
\end{equation*}
for $t \in [0,T_\eps)$. As in Lemma \ref{lem:H^1}, the inequality \eqref{eq:elementary_inequ} and Young's inequality with exponents $p=\frac{\alpha+1}{\alpha}$ and $q = \alpha+1$ yields, for all $\delta > 0$ the existence of some constant $C_{\delta,\alpha} > 0$ such that
\begin{equation} \label{eq:test_with_selfadjoint}
	\begin{split}
	&
	\frac{d}{dt} \left(\frac{1}{2} \int_{S^1} \left|A(u^\eps-u_0)\right|^2\, d\theta\right)
	\\
	&=
	-\int_{S^1}
	m_\eps(u^\eps)\, \psi_\eps\bigl(\eta_\eps\bigl(\de{}u\e{\eps}+\de{3}u\e{\eps}\bigr)\bigr)\, \left[S (\partial_\theta + \de{3})\eta_\eps(u^\eps-u_0)\right]\,
	d\theta
	\\
	&
	\leq
	\int_{S^1} m_\eps(u^\eps) \left|\bigl(\eta_\eps\bigl(\de{}u\e{\eps}+\de{3}u\e{\eps}\bigr)\bigr) + \eps^2\right|^\frac{\alpha(\alpha-1)}{2(\alpha+1)} \left|\eta_\eps\bigl(\de{}u\e{\eps}+\de{3}u\e{\eps}\bigr)\right|^\frac{2\alpha}{\alpha+1}
	\left|S(\partial_\theta + \partial_\theta^3) \eta_\eps(u^\eps - u_0)\right|\, d\theta
	\\
	&
	\leq
	\delta \int_{S^1}
	m_\eps(h^\eps)\, \left|\bigl(\eta_\eps\bigl(\partial_\theta h^\eps + \partial_\theta^3 h^\eps\bigr)\bigr)^2 + \eps^2\right|^\frac{\alpha-1}{2}
	\left|\eta_\eps\bigl(\de{}u\e{\eps}+\de{3}u\e{\eps}\bigr)\right|^2 d\theta
	\\
	&\quad
	+ C_{\delta,\alpha}
	\int_{S^1}
	m_\eps(u^\eps)\,
	\left|S (\partial_\theta + \de{3})\eta_\eps(u^\eps-u_0)\right|^{\alpha+1}
	d\theta
	\\
	&\leq
	\delta \int_{S^1}
	m_\eps(h^\eps)\, \left|\bigl(\eta_\eps\bigl(\partial_\theta h^\eps + \partial_\theta^3 h^\eps\bigr)\bigr)^2 + \eps^2\right|^\frac{\alpha-1}{2}
	\left|\eta_\eps\bigl(\de{}u\e{\eps}+\de{3}u\e{\eps}\bigr)\right|^2 d\theta
	\\
	&\quad
	+ C_{\delta,\alpha}
	\int_{S^1}
	m_\eps(u^\eps)\,
	\left|H (\de{2} +I)\eta_\eps(u^\eps-u_0)\right|^{\alpha+1}\,
	d\theta
	\end{split}
\end{equation}
for all $t \in [0,T_\eps)$.
Next, we estimate the second integral on the right-hand side of \eqref{eq:test_with_selfadjoint}. Using the definition of the operator $S$ and Lemma \ref{lem:interpolation_est}, we find that
\begin{equation*}
\begin{split}
	&
    \int_{S^1}
	m_\eps(u^\eps)\, 
	\left|H (\de{2} +I)\eta_\eps(u^\eps-u_0)\right|^{\alpha+1}\,
	d\theta
	\\
	&=
	C\left\|u\e{\eps}\right\|_{L_{\infty}(S\e{1})}\e{\alpha+2}
	\left\|H(\de{2}+I)\eta_\eps(u\e{\eps}-u_0)\right\|_{L_{\alpha+1}(S\e{1})}\e{\alpha+1}
	\\
	&\leq C\left\|u\e{\eps}\right\|_{L_{\infty}(S\e{1})}\e{\alpha+2}
	\left\|(\de{2}+I)\eta_\eps(u\e{\eps}-u_0)\right\|_{L_{\alpha+1}(S\e{1})}\e{\alpha+1}
	\\
	&\leq 
	C\left\|u\e{\eps}\right\|_{L_{\infty}(S\e{1})}\e{\alpha+2}
	\left(
	\tilde{\delta}\left\|(\de{}+\de{3})\eta_\eps(u\e{\eps}-u_0)\right\|_{L_{\alpha+1}}\e{\alpha+1}
	+
	C_{\tilde{\delta}}\left\|(\de{}+\de{-1})\eta_\eps(u\e{\eps}-u_0)\right\|_{L_{2}(S\e{1})}\e{\alpha+1}
	\right)
	\\
	&\leq
	\tilde{\delta}\left\|u\e{\eps}\right\|_{H\e{1}(S\e{1})}\e{\alpha+2}\left\|(\de{}+\de{3})\eta_\eps(u\e{\eps}-u_0)\right\|_{L_{\alpha+1}(S^1)}\e{\alpha+1}
	+ C_{\tilde{\delta}}\left\|u\e{\eps}\right\|_{H\e{1}(S^1)}\e{\alpha+2}\left\|u\e{\eps}-u_0\right\|_{H\e{1}(S\e{1})}\e{\alpha+1}
	\end{split}
\end{equation*}
for arbitrarily small $\tilde{\delta} > 0$.
Integrating \eqref{eq:test_with_selfadjoint} with respect to time and using that at the initial time $t=0$ we have $\left\|A\bigl(u^\eps(0)-u_0\bigr)\right\|_{L_2(S^1)} = 0$ yield
\begin{equation*}
\begin{split}
	&
	\frac{1}{2} \left\|A\bigl(u^\eps(t)-u_0\bigr)\right\|_{L_2(S^1)}^2
	\\
	&
	\leq
	\delta\int_0\e{T_\eps} \int_{S^1}
	m_\eps(h^\eps)\, \left|\bigl(\eta_\eps\bigl(\partial_\theta h^\eps + \partial_\theta^3 h^\eps\bigr)\bigr)^2 + \eps^2\right|^\frac{\alpha-1}{2}
	\left|\eta_\eps\bigl(\de{}u\e{\eps}+\de{3}u\e{\eps}\bigr)\right|^2 d\theta\, dt
	\\
	&\quad 
	+
	C_{\delta,\tilde{\delta},\alpha} \int_0\e{T_\eps}\norm{u\e{\eps}}_{H\e{1}(S\e{1})}\e{\alpha+2}\norm{u\e{\eps}-u_0}_{H\e{1}(S\e{1})}^{\alpha+1}\, dt
	+
	C\tilde{\delta}
	\int_0\e{T_\eps}\norm{u\e{\eps}}\e{\alpha+2}_{H\e{1}(S\e{1})}
	\int_{S\e{1}}
	\left|(\de{}+\de{3})\eta_\eps(u\e{\eps}-u_0)\right|\e{\alpha+1}d\theta dt
	\\
	&\leq
	C \delta 
	+
	C_{\delta,\tilde{\delta},\alpha} T
	+ C \tilde{\delta},
\end{split}
\end{equation*}
where we use Lemma \ref{lem:H^1}, the fact that the dissipation is bounded thanks to Lemma \ref{lem:dissipation} and $T_\eps \leq T$.
Note that the right-hand side of this inequality can be made arbitrarily small by choosing $\delta, \tilde{\delta}$, and then $T$ sufficiently small.
Hence, by definition of $A$, we find that there exists a function $K_\frac{1}{2}(T) \geq 0$ with $\lim_{T \to 0} K_\frac{1}{2}(T) = 0$ such that
\begin{equation*}
	\left\|u^\eps - u_0\right\|_{L_\infty((0,T);H^{1/2}(S^1))} \leq K_\frac{1}{2}(T).
\end{equation*}
In view of Lemma \ref{lem:H^1}, interpolation between $H^{1/2}(S^1)$ and $H^1(S^1)$ leads us to the desired estimate
\begin{equation*}
	\left\|u^\eps - u_0\right\|_{L_\infty((0,T);H^s(S^1))} \leq K_s(T)
\end{equation*}
	for all $s \in [1/2,1]$. This completes the proof for $s \in (0,1/2]$.
\end{proof}

\medskip 

Note that Lemma \ref{lem:bound_from_zero} implies that the solution $h^\eps$ stays bounded away from zero in the sense that
\begin{equation*}
	0 < \frac{\bar{h}_0}{2} \leq h(t,\theta) \leq 2 \bar{h}_0,
	\quad
	t \in [0,T],\, \theta \in S^1,
\end{equation*}
for $T > 0$ sufficiently small, i.e. with a bound independent of $\eps$, if we require $\|h_0 - \bar{h}_0\|_{H^1(S^1)}\leq \delta$ for $\delta > 0$ sufficiently small (and independent of $\eps$).

\medskip

In the following lemma we collect the uniform a-priori estimates for the approximations $h^\eps$.

\medskip

\begin{lemma}[Uniform bounds] \label{lem:uniform_bounds}
Let $\eps \in (0,1)$ be given and let $h^\eps$ be the corresponding solution to \eqref{eq:pde_eps} on $[0,T_\eps)$ with initial value $h_0 \in H^1(S^1)$. There exist $\delta > 0$ sufficiently small and $T > 0$, both independent of $\eps$ such that, if $\|h_0 - \bar{h}_0\|_{H^1(S^1)} \leq \delta$, then $T_\eps > T$ and the functions $h^\eps$ have the following properties.
\begin{itemize}
	\item[(i)] The family $(h^\eps)_\eps$ is uniformly bounded in $L_\infty\bigl((0,T);H^1(S^1)\bigr)$; 
	\item[(ii)] the family $\bigl(m_\eps(h^\eps)\, \psi_\eps\bigl(\eta_\eps\bigl(\de{}h\e{\eps}+\de{3}h\e{\eps}\bigr)\bigr)\bigr)_\eps$ is uniformly bounded in $L_\frac{\alpha+1}{\alpha}\bigl((0,T)\times S^1\bigr)$;
	\item[(iii)] the family $(\partial_t h^\eps)_\eps$ is uniformly bounded in $L_\frac{\alpha+1}{\alpha}\bigl((0,T);(W^1_{\alpha+1}(S^1))'\bigr)$;
	\item[(iv)] the family $\bigl(\eta_\eps\bigl(\de{}h\e{\eps}+\de{3}h\e{\eps}\bigr)\bigr)_\eps$ is uniformly bounded in $L_{\alpha+1}\bigl((0,T)\times S^1\bigr)$;
	\item[(v)] the family $(\eta_\eps h^\eps)_\eps$ is uniformly bounded in $L_{\alpha+1}\bigl((0,T);W^3_{\alpha+1}(S^1)\bigr)$;
	\item[(vi)] the family $\bigl(\partial_t(\de{}h\e{\eps})\bigr)_\eps$ is uniformly bounded in $L_\frac{\alpha+1}{\alpha}\bigl((0,T);\bigl(W^1_{\alpha+1,0}(S^1)\cap W^2_{\alpha+1}(S^1)\bigr)'\bigr)$.
\end{itemize}
\end{lemma}


\begin{proof}
(i) Uniform boundedness of $(h^\eps)_\eps$ in $L_\infty\bigl((0,T);H^1(S^1)\bigr)$ has already been proved in Lemma \ref{lem:H^1}. 

(ii) This follows by applying \eqref{eq:elementary_inequ}, H\"older's inequality with exponents $p=\alpha+1$ and $q=(\alpha+1)/\alpha$ and using the uniform bounds on the dissipation term (Lemma \ref{lem:dissipation}) and on the $L_\infty\bigl((0,T_\eps)\times S^1\bigr)$-norm (cf. part (i) of this lemma):
\begin{equation*}
    \left\|m_\eps(h^\eps)\, \psi_\eps\bigl(\eta_\eps\bigl(\de{}h\e{\eps}+\de{3}h\e{\eps}\bigr)\bigr)\right\|_{L_\frac{\alpha+1}{\alpha}((0,T)\times S^1)}
    \leq
    C \left\|h^\eps\right\|_{L_\infty((0,T)\times S^1)}^{\alpha+2} \left(D_T^\eps[h^\eps]\right)^\frac{\alpha}{\alpha+1}
    \leq
    C(h_0).
\end{equation*}

(iii) Since $h^\eps$ is a weak solution to \eqref{eq:pde_eps}, we have that
\begin{equation*}
	\int_0^T \langle \partial_t h^\eps(t), \varphi(t) \rangle_{W^1_{\alpha+1}(S^1)}\, dt
	=
	\int_0^T \int_{S^1} m_\eps(h^\eps)\, \psi_\eps\left(\eta_\eps\bigl(\de{}h\e{\eps}+\de{3}h\e{\eps}\bigr)\right) \eta_\eps(\de{}\varphi)\, d\theta\, dt,
\end{equation*}
for all $\varphi \in L_{\alpha+1}\bigl((0,T);W^1_{\alpha+1}(S^1)\bigr)$. Using again the elementary inequality \eqref{eq:elementary_inequ} and H\"older's inequality with exponents $p=\alpha+1$ and $q=(\alpha+1)/\alpha$, we obtain
\begin{align*}
	&
	\left|\int_0^T \langle \partial_t h^\eps(t), \varphi(t) \rangle_{W^1_{\alpha+1}(S^1)}\, dt
	\right|
	\\
	& \leq
	\int_0^T \int_{S^1} m_\eps(h^\eps)\, \left|\bigl(\eta_\eps\bigl(\partial_\theta h^\eps + \partial_\theta^3 h^\eps\bigr)\bigr)^2 + \eps^2\right|^{\frac{\alpha(\alpha-1)}{2(\alpha+1)}} \left|\eta_\eps\bigl(\partial_\theta h^\eps + \partial_\theta^3 h^\eps\bigr)\right|^\frac{2\alpha}{\alpha+1} \left|\eta_\eps(\de{}\varphi)\right|\, d\theta\, dt
	\\
	&\leq
	C\,	
	\bigl(D^\eps_T[h^\eps]\bigr)^\frac{\alpha}{\alpha+1}
	\left(\int_0^T \int_{S^1} m_\eps(h^\eps)\, \left|\eta_\eps(\de{}\varphi)\right|^{\alpha+1}\, d\theta\, dt
	\right)^\frac{1}{\alpha+1}.
\end{align*}
Young's inequality for convolutions and the fact that the mollifier has mass $1$, thus lead to the uniform estimate
\begin{equation*}
    \left|\int_0^T \langle \partial_t h^\eps(t), \varphi(t) \rangle_{W^1_{\alpha+1}(S^1)}\, dt\right|
    \leq
    C \left\|h^\eps\right\|_{L_\infty((0,T)\times S^1)}^\frac{\alpha+2}{\alpha+1} \left(D_T^\eps[h^\eps]\right)^\frac{\alpha}{\alpha+1}
    \leq
    C(h_0)
\end{equation*}
with a positive constant $C(h^\eps_0)$ that does not depend on $\eps$. 

(iv) We prove that $\eta_\eps\bigl(\de{}h\e{\eps}+\de{3}h\e{\eps}\bigr)$ is uniformly bounded in $L_{\alpha+1}\bigl((0,T)\times S^1\bigr)$. To this end, recall from Lemma \ref{lem:bound_from_zero} that $h^\eps$ is, for each $\eps \in (0,1)$, bounded from away from zero for short times. For $\eps >0$, we split
\begin{equation*}
	\begin{split}
		&
		\int_0^T \int_{S^1} \left|\eta_\eps\bigl(\de{}h\e{\eps}+\de{3}h\e{\eps}\bigr)\right|^{\alpha+1}\, d\theta\, dt
		\\
		=\ &
		\iint_{\{|\eta_\eps(\de{}h\e{\eps}+\de{3}h\e{\eps})| \leq \eps\}}
		\left|\eta_\eps\bigl(\de{}h\e{\eps}+\de{3}h\e{\eps}\bigr)\right|^{\alpha+1}\, d\theta\, dt
		+
		\iint_{\{|\eta_\eps(\de{}h\e{\eps}+\de{3}h\e{\eps})| > \eps\}}
		\left|\eta_\eps\bigl(\de{}h\e{\eps}+\de{3}h\e{\eps}\bigr)\right|^{\alpha+1}\, d\theta\, dt.
	\end{split}
\end{equation*}
Using the inequality
\begin{equation*}
	|x|^{\alpha+1} 
	=
	\left(\tfrac{1}{2} |x|^2 + \tfrac{1}{2} |x|^2\right)^\frac{\alpha-1}{2} |x|^2
	\leq
	\left(\tfrac{1}{2}\right)^\frac{\alpha-1}{2}
	\left(|x|^2 +  \sigma^2\right)^\frac{\alpha-1}{2} |x|^2,
	\quad
	|x| > \eps,
\end{equation*}
this leads us to the estimate
\begin{equation*}
	\begin{split}
		&
		\int_0^T \int_{S^1} \left|\eta_\eps\bigl(\de{}h\e{\eps}+\de{3}h\e{\eps}\bigr)\right|^{\alpha+1}\, d\theta\, dt
		\\
		\leq\ &
		2\pi T \eps^{\alpha+1}
		+
		\int_0^T \int_{S^1} \left|\eta_\eps\bigl(\de{}h\e{\eps}+\de{3}h\e{\eps}\bigr) + \sigma^2\right|^\frac{\alpha-1}{2}
		\left|\eta_\eps\bigl(\de{}h\e{\eps}+\de{3}h\e{\eps}\bigr)\right|^2 d\theta\, dt
		\\
		\leq\ &
		2\pi T \eps^{\alpha+1}
		+
		D_T[h^\eps].
	\end{split}
\end{equation*}
%
Using again the uniform bound for the dissipation functional derived in Lemma \ref{lem:dissipation}, we obtain the desired bound
\begin{equation*}
    \int_0^T \int_{S^1} \left|\eta_\eps\bigl(\de{}h\e{\eps}+\de{3}h\e{\eps}\bigr)\right|^{\alpha+1}\, d\theta\, dt
    \leq
    C(h_0).
\end{equation*}

(v)
We prove the estimate
\begin{equation*}
    \int_0^T\int_{S\e{1}}\left|\eta_\eps\bigl(\de{3}h^\eps\bigr)\right|\e{\alpha+1}d\theta\, dt \leq \int_0^T\int_{S\e{1}}\left|\eta_\eps\bigl(\de{}h^\eps+\de{3}h^\eps\bigr)\right|\e{\alpha+1}d\theta\, dt + C(T) \norm{\eta_\eps(h^\eps)}_{H\e{1}(S\e{1})}\e{\alpha+1}.
\end{equation*}    
To this end, 
we define $V_0= \text{span}\{\cos(\theta),\sin(\theta)\}\subset L_{2}(S^1)$ and $V_1$ as the orthogonal complement of $V_0\oplus \text{span}\{1\}$ in $L_{2}(S^1)$.
Given $\eta_\eps(h^\eps)\in H\e{3}(S^1)$, we can decompose $\eta_\eps(h^\eps)$ as 
\begin{equation}\label{descomposicion}
    \eta_\eps(h^\eps)
    =
    a_0+\Phi+v
\end{equation}
with $a_0\in\R$, $\Phi\in V_0\cap H\e{3}(S^1)$ and $v\in V_1\cap H\e{3}(S^1)$.
In terms of the corresponding Fourier series we may write
\begin{equation*}
  (\de{3}v)_n
  =
  -in\e{3}v_n=m(n)i(n-n\e{3})v_n,
  \quad
  \text{where}
  \quad
  m(n)
  =
  \frac{-in\e{3}}{i(n-n\e{3})}
  =
  \frac{n\e{2}}{n\e{2}-1} 
  \quad \text{for} \quad 
  n\neq 0,\pm 1.
\end{equation*}
Since $m(n)$ is bounded, we can apply the Littlewood--Paley Theory (c.f. \cite[Chapter 4]{stein}) to obtain 
\begin{equation}\label{littlewood}
    \norm{\de{3}v}_{L_{\alpha+1}(S^1)}\e{\alpha+1}\leq A\norm{(\de{}+\de{3})v}_{L_{\alpha+1}(S^1)}\e{\alpha+1}
\end{equation}
with a positive constant $A$ which is independent of $v$.
Therefore, using \eqref{littlewood} and the fact that
$\Phi=a_1(t)\cos\theta+a_{-1}(t)\sin\theta$, we can write
\begin{equation}
\begin{split}
    \int_0^T\int_{S\e{1}}\left|\eta_\eps\bigl(\de{3}h^\eps\bigr)\right|\e{\alpha+1}d\theta\, dt
    & \leq C\left(\int_0^T\int_{S\e{1}}|\de{3}\Phi|\e{\alpha+1}d\theta\, dt
    +
    \int_0^T\int_{S\e{1}}|\de{3}v|\e{\alpha+1}d\theta\, dt\right)\\\label{eq86}
    &\leq C\left(\int_0^T\abs{a_1(t)}\e{\alpha+1}+\abs{a_{-1}(t)}\e{\alpha+1}\, dt\right) \\ &\quad + CA\left(\int_0^T\int_{S\e{1}}\abs{\de{}v+\de{3}v}\e{\alpha+1}\, d\theta\, dt\right).  
\end{split}
\end{equation}
Indeed, for the first term on the right-hand side of \eqref{eq86} we use the structure of $\Phi$ and Young's inequality for convolutions to derive the pointwise estimate 
\begin{equation*}
\begin{split}
    \abs{a_1(t)}\e{\alpha+1}+\abs{a_{-1}(t)}\e{\alpha+1}&\leq C\left(\abs{a_1(t)}\e{2}+\abs{a_{-1}(t)}\e{2}\right)\e{\frac{\alpha+1}{2}} \\ &= C\norm{\Phi(t)}_{L_2(S\e{1})}\e{\alpha+1}\\ &\leq C\norm{\eta_\eps(h^\eps(t))}_{H\e{1}(S\e{1})}\e{\alpha+1}.
\end{split}
\end{equation*}
For the second integral on the right-hand side of \eqref{eq86} we use that $\de{}\Phi+\de{3}\Phi=0$ and we obtain that
\begin{equation*}
    \int_0^T\int_{S\e{1}}\abs{\de{}v+\de{3}v}\e{\alpha+1}\, d\theta\, dt
    =
    \int_0^T\int_{S\e{1}}\abs{\de{}\eta_\eps(h^\eps)+\de{3}\eta_\eps(h^\eps)}\e{\alpha+1}\, d\theta\, dt.
\end{equation*}
Thus, we can conclude that
\begin{equation*}
\begin{split}
    \int_0^T\int_{S\e{1}}\left|\eta_\eps\bigl(\de{3}h^\eps\bigr)\right|\e{\alpha+1}d\theta\, dt &\leq C \int_0^T \norm{h^\eps(t)}_{H\e{1}(S\e{1})}\e{\alpha+1}\, dt
    + C_A \int_0^T\int_{S\e{1}}\bigl|\eta_\eps\bigl(\de{} h^\eps+\de{3}h^\eps\bigr)\bigr|\e{\alpha+1}\, d\theta\, dt
    \\ &
    \leq C(T,h_0)
\end{split}
\end{equation*}
and we have proved the desired result.

(vi) This follows similarly as in (iii).
\end{proof}


Next, we prove that the approximations $h^\eps$ converge in a suitable sense.

\begin{lemma}[Convergence of approximations] \label{lem:convergence}
Let $\eps \in (0,1)$ be given and let $h^\eps$ be the corresponding solution to \eqref{eq:pde_eps} on $[0,T_\eps)$ with initial value $h_0 \in H^1(S^1)$.
There exist $\delta > 0$ sufficiently small and $T > 0$, both independent of $\eps$ such that, if $\|h_0 - \bar{h}_0\|_{H^1(S^1)} \leq \delta$, then $T_\eps > T$ and we may extract a subsequence $(h^\eps)_\eps$ (not relabelled) such that, as $\eps \searrow 0$,
\begin{itemize}
	\item[(i)] $h^\eps \to h$ strongly in $C\bigl([0,T];C^{\rho}(S^1)\bigr)$; 
	\item[(ii)] 
	$m_\eps(h^\eps) \psi_\eps\bigl(\eta_\eps\bigl(\de{}h\e{\eps}+\de{3}h\e{\eps}\bigr)\bigr) \rightharpoonup \Phi$ weakly in $L_\frac{\alpha+1}{\alpha}\bigl([0,T]\times S^1\bigr)$ for some limit function $\Phi$;
	\item[(iii)] $\partial_t h^\eps \rightharpoonup \partial_t h$ weakly in 		$L_\frac{\alpha+1}{\alpha}\bigl((0,T)\times (W^1_{\alpha+1}(S^1))'\bigr)$;
	\item[(iv)] $\eta_\eps(\de{}h\e{\eps}+\de{3}h\e{\eps}) \rightharpoonup (\de{}h+\de{3}h)$ weakly in $L_{\alpha+1}\bigl([0,T]\times S^1\bigr)$;
	\item[(v)] $\partial_t(\de{}h\e{\eps}) \rightharpoonup \partial_t\de{}h$ weakly in $L_\frac{\alpha+1}{\alpha}\bigl((0,T);\bigl(W^1_{\alpha+1,0}(S^1)\cap W^2_{\alpha+1}(S^1)\bigr)'\bigr)$.
\end{itemize}
\end{lemma}


\begin{proof}
(i) In the previous Lemma \ref{lem:uniform_bounds} (i), (iii) we have proved that 
\begin{equation*}
	\begin{cases}
		(h^\eps)_\eps \text{ is uniformly bounded in } L_\infty\bigl((0,T);H^1(S^1)\bigr) 
		& \\
		(\partial_t h^\eps)_\eps \text{ is uniformly bounded in } L_\frac{\alpha+1}{\alpha}\bigl((0,T);(W^1_{\alpha+1}(S^1))'\bigr).&
	\end{cases}
\end{equation*}
Moreover, thanks to the Rellich-Kondrachov theorem, cf. for instance in \cite[Thm. 6.3]{adams_fournier}, we know that
\begin{equation*}
	H^1(S^1) \xhookrightarrow[]{c} C^{\rho}(S^1) \hookrightarrow (W^1_{\alpha+1}(S^1))', \quad \rho \in [0,1/2),
\end{equation*}
where $\xhookrightarrow[]{c}$ indicates compactness of the embedding.
This allows us to invoke \cite[Cor. 4]{S:1987} in order to conclude that the sequence 
\begin{equation*}
		(h^\eps)_\eps \text{ is relatively compact in } C\bigl([0,T];C^\rho(S^1)\bigr) 
\end{equation*}
with $\rho \in [0,1/2)$ as above.
		    
(ii) This is an immediate consequence of Lemma \ref{lem:uniform_bounds} (ii).
	
(iii) Thanks to Lemma \ref{lem:uniform_bounds} (iii), we may extract a subsequence $(\partial_t h^\eps)_\eps$ such that
\begin{equation*}
	\partial_t h^\eps \rightharpoonup v 
	\quad \text{weakly in } L_\frac{\alpha+1}{\alpha}\bigl((0,T);(W^1_{\alpha+1}(S^1))'\bigr) \hookrightarrow \Dcal'\bigl((0,T);(W^1_{\alpha+1}(S^1))'\bigr)
\end{equation*}
for some limit function $v \in L_\frac{\alpha+1}{\alpha}\bigl((0,T);(W^1_{\alpha+1}(S^1))'\bigr)$.
Since we know in addition that
\begin{equation*}
	h^\eps \longrightarrow h \quad \text{in }  C\bigl([0,T];C^\rho(S^1)\bigr) \hookrightarrow \Dcal'\bigl((0,T);(W^1_{\alpha+1}(S^1))'\bigr),
	\quad \rho \in [0,1/2),
\end{equation*}
we conclude that 
\begin{equation*}
	\partial_t h^\eps \longrightarrow \partial_t h
	\quad \text{in }
	\Dcal'\bigl((0,T);(W^1_{\alpha+1}(S^1))'\bigr)\bigr),
\end{equation*}
and consequently, $v = \partial_t h \in L_\frac{\alpha+1}{\alpha}\bigl((0,T);(W^1_{\alpha+1}(S^1))'\bigr)$.
	
(iv) The strong convergence $h^\eps \to h$ in $C\bigl([0,T];C^\rho(S^1)\bigr),\, \rho \in [0,1/2),$ in Lemma \ref{lem:convergence} (i) in particular implies uniform convergence $h^\eps \to h$ in $C\bigl([0,T]\times S^1\bigr)$ and then, by definition of the mollifier,
	\begin{equation} \label{eq:conv_mollifier_1}
	    \eta_\eps h^\eps \longrightarrow h
	    \quad
	    \text{in } C\bigl([0,T]\times S^1\bigr).
	\end{equation}
	Moreover, Lemma \ref{lem:uniform_bounds} (v) guarantees the existence of some $\hat{h} \in L_{\alpha+1}\bigl((0,T);W^3_{\alpha+1}(S^1)\bigr)$ such that
	\begin{equation} \label{eq:conv_mollifier_2}
	    \eta_\eps h^\eps \rightharpoonup \hat{h}
	    \quad
	    \text{in } L_{\alpha+1}\bigl((0,T);W^3_{\alpha+1}(S^1)\bigr).
	\end{equation}
    In virtue of the uniqueness of the limit function, \eqref{eq:conv_mollifier_1} and \eqref{eq:conv_mollifier_2} imply
	\begin{equation*}
	    \eta_\eps h^\eps \rightharpoonup h
	    \quad
	    \text{in } 
	    L_{\alpha+1}\bigl((0,T);W^3_{\alpha+1}(S^1)\bigr).
	\end{equation*}
	Thanks to the weak lower semicontinuity of the norm and Lemma \ref{lem:uniform_bounds} (iv), (v), we finally obtain
\begin{equation} \label{eq:lsc}
	\begin{cases}
	    \left\|\de{}h+\de{3}h\right\|_{L_{\alpha+1}((0,T)\times S^1)}
	    \leq
	    \liminf_{\eps \to 0} \left\|(\partial_\theta + \partial_\theta^3)(\eta_\eps h^\eps)\right\|_{L_{\alpha+1}((0,T)\times S^1)}
	    \leq 
	    C
	    & \\
	    \left\|h\right\|_{L_{\alpha+1}((0,T);W^3_{\alpha+1}(S^1))}
	    \leq
	    \liminf_{\eps \to 0} \left\|\eta_\eps h^\eps\right\|_{L_{\alpha+1}((0,T);W^3_{\alpha+1}(S^1))}
	    \leq 
	    C &
	\end{cases}
\end{equation}
for some positive generic constant $C > 0$ that does not depend on $\eps$. 

(v) This follows similarly as in (iii) and the proof is complete.
\end{proof}


It remains to prove the convergence of the nonlinear flux term 
$m_\eps(h^\eps)\, \psi_\eps\bigl(\eta_\eps\bigl(\de{}h\e{\eps}+\de{3}h\e{\eps}\bigr)\bigr) \rightharpoonup \left|h\right|^{\alpha+2} \psi\bigl(\de{}h+\de{3}h\bigr)$
in $L_\frac{\alpha+1}{\alpha}\bigl((0,T)\times S^1\bigr)$. This is the content of the next lemma. The main idea of the proof is to use lower semicontinuity of the norm and to apply Minty's trick in order to be able to identify the nonlinear limit flux.


\begin{lemma}\label{lem:limit_flux}
Let $\eps \in (0,1)$ be given and let $h^\eps$ be the corresponding solution to \eqref{eq:pde_eps} on $[0,T_\eps)$ with initial value $h_0 \in H^1(S^1)$.
There exist $\delta > 0$ sufficiently small and $T > 0$, both independent of $\eps$ such that, if $\|h_0 - \bar{h}_0\|_{H^1(S^1)} \leq \delta$, then $T_\eps > T$ and we may extract a subsequence $(h^\eps)_\eps$ (not relabelled) such that
\begin{equation*}
	m_\eps(h^\eps) \psi_\eps\bigl(\eta_\eps\bigl(\partial_\theta h^\eps + \partial^3_\theta h^\eps\bigr)\bigr) \rightharpoonup 
	\left|h\right|^{\alpha+2} \psi\bigl(\de{}h+\de{3}h\bigr) 
	\quad
	\text{weakly in }
	L_\frac{\alpha+1}{\alpha}\bigl([0,T]\times S^1\bigr)
\end{equation*}
as $\eps \searrow 0$.
\end{lemma}


\begin{proof}
We divide the proof in several steps. For convenience, we pass to a subsequence where necessary without explicitly mentioning.
	
(i) 
First, in virtue of Lemma \ref{lem:convergence} (ii), we know that $m_\eps(h^\eps)\, \psi_\eps\bigl(\eta_\eps\bigl(\partial_\theta h^\eps + \partial^3_\theta h^\eps\bigr)\bigr)$ is weakly sequentially compact, i.e. there is an element $\Phi \in L_\frac{\alpha+1}{\alpha}\bigl((0,T)\times S^1)\bigr)$ such that
\begin{equation*}
	m_\eps(h^\eps) \psi_\eps\bigl(\eta_\eps\bigl(\partial_\theta h^\eps + \partial^3_\theta h^\eps\bigr)\bigr)
	\rightharpoonup \Phi
    \quad
	\text{weakly in } 
	 L_\frac{\alpha+1}{\alpha}\bigl((0,T)\times S^1)\bigr).
\end{equation*}
It remains to identify the limit flux $\Phi$.
	
(ii) Next, we prove that $h$ is bounded in $C\bigl([0,T];H^1(S^1)\bigr)$.
We already know from Lemma \ref{lem:convergence} (i) that
\begin{equation*}
	h \in C\bigl([0,T];C^\rho(S^1)\bigr) \hookrightarrow C\bigl([0,T];L_2(S^1)\bigr).
\end{equation*}
    Moreover,
    \begin{equation*}
	    \de{}h \in L_{\alpha+1}\bigl((0,T);W^1_{\alpha+1,0}(S^1) \cap W^2_{\alpha+1}(S^1)\bigr) \quad \text{and} \quad
	    \partial_t\de{}h\in L_\frac{\alpha+1}{\alpha}\bigl((0,T);\bigl(W^1_{\alpha+1,0}(S^1)\cap W^2_{\alpha+1}(S^1)\bigr)'\bigr)
    \end{equation*}
	thanks to Lemma \ref{lem:convergence} (iv) and (v) and lower semicontinuity of the norm. Using \cite[Remark 3.4]{Bernis:1988}, this implies that $\partial_\theta h \in C\bigl([0,T];L_2(S^1)\bigr)$. Consequently, $h \in C\bigl([0,T];H^1(S^1)\bigr)$.
	
(iii) In view of the previous steps we may choose $\varphi = h + \de{2}h \in L_{\alpha+1}\bigl((0,T);W^1_{\alpha+1}(S^1)\bigr)$ as a test function in the weak formulation \eqref{eq:weak_solution_eps} for $h^\eps$. This yields
\begin{equation*}
	\int_0^T \int_{S^1} \partial_t h^\eps (h + \de{2}h)\, d\theta\, dt + \int_0^T \int_{S^1} m_\eps(h^\eps) \psi_\eps\left(\eta_\eps\bigl(\partial_\theta h^\eps + \partial_\theta^3 h^\eps\bigr)\right) \eta_\eps\bigl(\de{}h+\de{3}h\bigr)\, d\theta\, dt
	= 0.
\end{equation*}
As $\eps \searrow 0$, the first term satisfies
\begin{equation*}
	\int_0^T \int_{S^1} \partial_t h^\eps (h + \de{2}h)\, d\theta\, dt \longrightarrow 
	\int_0^T \int_{S^1} \partial_t h\, (h + \de{2}h)\, d\theta\, dt
	=
	E[h](T) - E[h](0),
\end{equation*}
where we have used that the limit function satisfies $h \in C\bigl([0,T];H^1(S^1)\bigr)$.
Moreover, since $(\de{}h+\de{3}h)$ is bounded in $L_{\alpha+1}\bigl((0,T)\times S^1\bigr)$ by \eqref{eq:lsc}, the definition of the mollifier yields the strong convergence
\begin{equation*}
	\eta_\eps\bigl(\de{}h+\de{3}h\bigr)
	\longrightarrow \de{}h+\de{3}h
	\quad \text{in } L_{\alpha+1}\bigl((0,T)\times S^1\bigr).
\end{equation*}
Together with Lemma \ref{lem:convergence} (ii) this implies
\begin{equation*}
	\int_0^T \int_{S^1} 
	m_\eps(h^\eps)\, \psi_\eps\left(\eta_\eps\bigl(\partial_\theta h^\eps + \partial_\theta^3 h^\eps\bigr)\right) \eta_\eps\bigl(\de{}h+\de{3}h\bigr)\, d\theta\, dt
	\longrightarrow 
	\int_0^T \int_{S^1} \Phi \bigl(\de{}h+\de{3}h\bigr)\, d\theta\, dt,
\end{equation*}	
and consequently, we obtain the identity
\begin{equation*}
	E[h](t) + \left\langle \Phi | \de{}h+\de{3}h \right\rangle
	=
	E[h_0]
\end{equation*}
for almost every $t \in [0,T]$.
	
(iv) Monotonicity and identification of the limit flux $\Phi$ by Minty's trick. Observe that the operator 
\begin{equation*}
	\begin{cases}
	    \psi_\eps : L_{\alpha+1}\bigl((0,T)\times S^1\bigr) \longrightarrow L_{\frac{\alpha+1}{\alpha}}\bigl((0,T)\times S^1\bigr), & \\
    	\psi_\eps(u) = \left(|u|^2 + \eps^2\right)^\frac{\alpha-1}{2} u &
    \end{cases}
\end{equation*}
is monotone, i.e. we have
\begin{equation*}
	\langle \psi_\eps(u) - \psi_\eps(v) | u-v\rangle_{L_{\alpha+1}}
	=
	\int_0^T \int_{S^1} \bigl(\psi_\eps(u) - \psi_\eps(v)\bigr) (u-v)\, 	d\theta\, dt
	> 0
\end{equation*}
for all $u, v \in L_{\alpha+1}\bigl([0,T]\times S^1\bigr)$ with $u\neq v$. This follows immediately from the monotonicity of the function $\psi_\eps\colon \R \to \R,\ s\mapsto (s^2 + \eps^2)^\frac{\alpha-1}{2}s$.
Let now $\phi \in W^3_{\alpha+1}\bigl((0,T)\times S^1\bigr)$. For better readability, henceforth we simply write $\langle u | v\rangle$ for the dual pairing $\langle u |v \rangle_{L_{\alpha+1}((0,T)\times S^1)}$ between $u \in L_\frac{\alpha+1}{\alpha}\bigl((0,T)\times S^1\bigr)$ and $v \in L_{\alpha+1}\bigl((0,T)\times S^1\bigr)$. Thanks to the monotonicity of $\psi_\eps$ we have 
\begin{align*}
	0
	&\leq
	\left\langle m_\eps(h^\eps) \psi_\eps\bigl(\eta_\eps\bigl(\partial_\theta h^\eps + \partial_\theta^3 h^\eps\bigr)\bigr) - m_\eps(h^\eps) \psi_\eps \bigl(\partial_\theta \phi + \partial_\theta^3 \phi\bigr) | (\partial_\theta + \partial_\theta^3) (\eta_\eps h^\eps - \phi)
	\right\rangle
	\\
	&=	
	\left\langle 
	m_\eps(h^\eps) \psi_\eps\bigl(\eta_\eps\bigl(\partial_\theta h^\eps + \partial_\theta^3 h^\eps\bigr)\bigr) | \eta_\eps\bigl(\partial_\theta h^\eps + \partial_\theta^3 h^\eps\bigr) \right\rangle
	-
	\left\langle m_\eps(h^\eps) \psi_\eps\bigl(\eta_\eps\bigl(\partial_\theta h^\eps + \partial_\theta^3 h^\eps\bigr)\bigr) | \partial_\theta \phi + \partial_\theta^3 \phi \right\rangle
	\\
	&\quad
	-
	\left\langle m_\eps(h^\eps) \psi_\eps\bigl(\partial_\theta \phi + \partial_\theta^3 \phi\bigr) | \eta_\eps\bigl(\partial_\theta h^\eps + \partial_\theta^3 h^\eps) \right\rangle
	+
	\left\langle m_\eps(h^\eps) \psi_\eps\bigl(\partial_\theta \phi + \partial_\theta^3 \phi\bigr) | \partial_\theta \phi + \partial_\theta^3 \phi \right\rangle.
\end{align*}
Let us consider the different pairings on the right-hand side separately.
	
First, we have proved in Lemma \ref{lem:dissipation} that $h^\eps$ satisfies the energy dissipation formula for the problem \eqref{eq:pde_eps}. We rewrite it here as
\begin{equation*}
	\left\langle m_\eps(h^\eps) \psi_\eps\bigl(\eta_\eps\bigl(\partial_\theta h^\eps + \partial_\theta^3 h^\eps\bigr)\bigr) | \eta_\eps\bigl(\partial_\theta h^\eps + \partial_\theta^3 h^\eps\bigr) 
	\right\rangle
	= 
	E[h_0] - E[h^\eps](t),\quad\text{for a.e. $t \in [0,T]$.}
\end{equation*}
Thanks to Lemma \ref{lem:convergence} (i) we know that $h^\eps(t) \to h(t)$ in $H^1(S^1)$ for almost every $t \in [0,T]$. Thus, as $\eps$ tends to zero, we find that
\begin{equation} \label{eq:Minty_1}
	\left\langle m_\eps(h^\eps) \psi_\eps\bigl(\eta_\eps\bigl(\partial_\theta h^\eps + \partial_\theta^3 h^\eps\bigr)\bigr) | \eta_\eps\bigl(\partial_\theta h^\eps + \partial_\theta^3 h^\eps\bigr) 
	\right\rangle
	\longrightarrow
	E[h_0] - E[h](t)
\end{equation}
for almost every $t \in [0,T]$. Moreover, Lemma \ref{lem:convergence} (ii) yields
\begin{equation}\label{eq:Minty_2}
	\left\langle m_\eps(h^\eps) \psi_\eps\bigl(\eta_\eps\bigl(\partial_\theta h^\eps + \partial_\theta^3 h^\eps\bigr)\bigr) | \partial_\theta \phi + \partial_\theta^3 \phi\bigr) 
	\right\rangle
	\longrightarrow
	\left\langle \Phi | \partial_\theta \phi + \partial_\theta^3 \phi \right\rangle
	\quad
	\text{as } \eps \searrow 0.
\end{equation}
For the third pairing we use that
\begin{equation*}
	\begin{cases}
		m_\eps(h^\eps) \longrightarrow h^{\alpha+2} \quad \text{strongly in } C\bigl([0,T]\times S^1\bigr) &
		\\
		\eta_\eps\bigl(\partial_\theta h^\eps + \partial_\theta^3 h^\eps\bigr) 
		\xrightharpoonup[\phantom{wea}]{}
		\partial_\theta h + \partial_\theta^3 h
		\quad \text{weakly in } L_{\alpha+1}\bigl([0,T]\times S^1\bigr), &
	\end{cases}
\end{equation*}
thanks to Lemma \ref{lem:convergence} (i), respectively Lemma \ref{lem:convergence} (iv). Here we also use the fact that $m_\eps(h^\eps) = |h^\eps|^{\alpha+2}$ in the range of values attained by $h^\eps$ for times $t \in [0,T]$, cf. \eqref{eq:def_m_eps}.
This implies
\begin{equation}\label{eq:Minty_3}
	\left\langle m_\eps(h^\eps) \psi_\eps\bigl(\partial_\theta \phi + \partial_\theta^3 \phi\bigr)\bigr) | \eta_\eps\bigl(\partial_\theta h^\eps + \partial_\theta^3 h^\eps\bigr) 
	\right\rangle
	\longrightarrow
	\left\langle \left|h\right|^{\alpha+2} \psi \bigl(\partial_\theta \phi + \partial_\theta^3 \phi\bigr) | \partial_\theta h + \partial_\theta^3 h \right\rangle.	
\end{equation}
Finally, for the fourth term, we obtain
\begin{equation}\label{eq:Minty_4}
	\left\langle m_\eps(h^\eps) \psi_\eps\bigl(\partial_\theta \phi + \partial_\theta^3 \phi\bigr) | \partial_\theta \phi + \partial_\theta^3 \phi \right\rangle
	\longrightarrow 
	\left\langle |h|^{\alpha+2} \psi\bigl(\partial_\theta \phi + \partial_\theta^3 \phi\bigr) | \partial_\theta \phi + \partial_\theta^3 \phi \right\rangle,
\end{equation} 
where we use again the convergence $m_\eps(h^\eps) \rightarrow h^{\alpha+2}$ strongly in $C\bigl([0,T]\times S^1\bigr)$.
Combining \eqref{eq:Minty_1}--\eqref{eq:Minty_4} leads to the inequality
\begin{equation*}
	0 \leq
	E[h_0] - E[h](t)
	-
	\left\langle \Phi | \partial_\theta \phi + \partial_\theta^3 \phi \right\rangle
	-
	\left\langle \left|h\right|^{\alpha+2} \psi\bigl(\partial_\theta \phi + \partial_\theta^3 \phi\bigr) | (\partial_\theta + \partial_\theta^3) (h - \phi) \right\rangle
\end{equation*}
Thus, using the identity 
\begin{equation*}
	E[h](t) + \left\langle \Phi | \partial_\theta h + \partial_\theta^3 h \right\rangle
	=
	E[h_0],
\end{equation*}
proved in step (iii), for almost every $t \in [0,T]$, we discover that
\begin{equation*}
	0 
	\leq
	\left\langle \Phi - \left|h\right|^{\alpha+2} \psi\bigl(\partial_\theta \phi + \partial_\theta^3 \phi\bigr) | (\partial_\theta + \partial_\theta^3) (h - \phi) \right\rangle.
\end{equation*}
By choosing $\phi = h - \lambda v$ for some arbitrary $v \in W^3_{\alpha+1}\bigl((0,T)\times S^1\bigr)$ and $\lambda > 0$, we obtain the inequality
\begin{equation*}
	\left\langle \Phi - \left|h\right|^{\alpha+2} \psi\bigl((\partial_\theta + \partial_\theta^3) (h - \lambda v)\bigr) | \partial_\theta v + \partial_\theta^3 v \right\rangle
	\geq 0
\end{equation*}
and hence, in the limit $\lambda \searrow 0$:
\begin{equation*}
	\left\langle \Phi - \left|h\right|^{\alpha+2} \psi\bigl(\partial_\theta h + \partial_\theta^3 h\bigr) | \partial_\theta v + \partial_\theta^3 v \right\rangle
	\geq 0,
	\quad
	v \in W^3_{\alpha+1}\bigl((0,T)\times S^1\bigr),
\end{equation*}	
for almost every $t \in [0,T]$. Similarly, choosing $\phi = h + \lambda v$, we discover that
\begin{equation*}
	\left\langle \Phi - \left|h\right|^{\alpha+2} \psi\bigl(\partial_\theta h + \partial_\theta^3 h\bigr) | \partial_\theta v + \partial_\theta^3 v \right\rangle
	\leq 0,
	\quad
	v \in W^3_{\alpha+1}\bigl((0,T)\times S^1\bigr),
\end{equation*}
and consequently	
\begin{equation*}
	\left\langle \Phi - \left|h\right|^{\alpha+2} \psi\bigl(\partial_\theta h + \partial_\theta^3 h\bigr) | \partial_\theta v + \partial_\theta^3 v \right\rangle
	= 0,
	\quad
	v \in W^3_{\alpha+1}\bigl((0,T)\times S^1\bigr).
\end{equation*}
Since $v \in W^3_{\alpha+1}\bigl((0,T)\times S^1\bigr)$ is arbitrary, we are finally able to identify
\begin{equation*}
	\Phi 
	=
	\left|h\right|^{\alpha+2} \psi\bigl(\partial_\theta h + \partial_\theta^3 h\bigr)
	\in
	L_\frac{\alpha+1}{\alpha}\bigl((0,T)\times S^1\bigr).
\end{equation*}
This completes the proof.
\end{proof}
\medskip

With the previous convergence results at hand, we are now able to prove Theorem \ref{thm:existence}.

\medskip

\begin{proof}[\textbf{Proof of Theorem \ref{thm:existence}}]
(i) 
It follows from Lemma \ref{lem:convergence} (iii) and (iv) and step (ii) of the proof of Lemma \ref{lem:limit_flux} that 
\begin{equation*}
    h \in C\bigl([0,T];H^1(S^1)\bigr) \cap L_{\alpha+1}\bigl((0,T);W^3_{\alpha+1}(S^1)\bigr)\quad\text{and}\quad \partial_t h \in L_\frac{\alpha+1}{\alpha}\bigl((0,T);(W^1_{\alpha+1}(S^1))'\bigr).
\end{equation*}

(ii) We now prove that $h$ complies with the integral equation in Theorem \ref{thm:existence} (ii). To this end, recall that 
\begin{equation*}\label{eq:weak_solution_eps}
    \int_0^T \langle \partial_t h^\eps(t),\varphi(t)\rangle_{W^1_{\alpha+1}(S^1)}\, dt
	=
	\int_0^T \int_{S^1} \eta_\eps \left[m_\eps(h^\eps) \psi_\eps\left(\eta_\eps\bigl(\de{}h\e{\eps}+\de{3}h\e{\eps}\bigr)\right)\right]\, \de{}\phi\,
	d\theta\, dt
\end{equation*}
for all test functions $\phi \in L_{\alpha+1}\bigl((0,T);W^1_{\alpha+1}(S^1)\bigr)$. Since $\de{}\phi \in L_{\alpha+1}\bigl((0,T)\times S^1)$, we may invoke Lemma \ref{lem:limit_flux} to deduce that, on the one hand, 
\begin{equation*}
    \int_0^T \langle \partial_t h^\eps(t),\varphi(t)\rangle_{W^1_{\alpha+1}(S^1)}\, dt \longrightarrow 
    \int_0^T \int_{S^1} h^{\alpha+2} \psi\bigl(\de{}h+\de{3}h\bigr)\, d\theta\, dt.
\end{equation*}
On the other hand, Lemma \ref{lem:convergence} (iii) implies that
\begin{equation*}
    \int_0^T \langle \partial_t h^\eps(t),\varphi(t)\rangle_{W^1_{\alpha+1}(S^1)}\, dt \longrightarrow \int_0^T \langle \partial_t h(t),\varphi(t)\rangle_{W^1_{\alpha+1}(S^1)}\, dt.
\end{equation*}
Combining both, we find that $h$ satisfies the desired integral identity
\begin{equation*}
    \int_0^T \langle \partial_t h(t),\varphi(t)\rangle_{W^1_{\alpha+1}(S^1)}\, dt = \int_0^T \int_{S^1} h^{\alpha+2} \psi\bigl(\de{}h+\de{3}h\bigr)\, d\theta\, dt
\end{equation*}
for all $\phi \in L_{\alpha+1}\bigl((0,T);W^1_{\alpha+1}(S^1)\bigr)$.

(iii) That the initial condition is satisfied is clear since we chose them identically in \eqref{eq:pde} and \eqref{eq:pde_eps}.

(iv) That a solution conserves its mass follows from the convergence $h^\eps \to h$ in $C\bigl([0,T];C^\rho(S^1)\bigr)$, cf. Lemma \ref{lem:convergence} (i), and from the conservation of mass property 
\begin{equation*}
    \int_{S^1} h^\eps(t,\theta)\, d\theta = \int_{S^1} h_0(\theta)\, d\theta, \quad t \in [0,T],
\end{equation*}
for the approximations $h^\eps$.

(v) That the solution satisfies the energy equality has already been proved in step (iii) of the proof of Lemma \ref{lem:limit_flux}.

(vi) To see that the solution is bounded away from zero for short times, we just argue as in the proof of Lemma \ref{lem:H^1} and Lemma \ref{lem:bound_from_zero}. 
\end{proof}

\medskip

\subsection{Global existence and convergence to a circle in finite time in the shear-thickening regime} \label{sec:global_existence}

In this section we study the setting in which the thin film next to the internal cylinder is occupied by a shear-thickening fluid. This corresponds to the regime of flow-behaviour exponents $\alpha < 1$. We consider solutions emerging from initial values that are close to a circle. For these solutions we show that they converge to a circle in finite time. This circle does not necessarily have to be centered at the origin. Note that this behaviour clearly differs from the Newtonian case (c.f. \eqref{eq:intro_Newtonian}).

The main result of this section is stated in the following theorem.


\begin{theorem}\label{thm:global_ex}
Let $\alpha < 1$. There exists a $\delta > 0$ such that for all initial values $h_0 \in H^1(S^1)$ with $\left\|h_0 - \bar{h}_0\right\|_{H^1(S^1)} < \delta$, there is a weak solution $h \in C\bigl([0,\infty);H^1(S^1)\bigr) \cap L_{\alpha+1,\text{loc}}\bigl((0,\infty);W^3_{\alpha+1}(S^1)\bigr)$ that exists globally in time. Moreover, there exists a time $0 < t^\ast < \infty$ such that
\begin{equation*}
    h(t,\theta) = \bar{h}_0 + v(\theta)
    \quad
    \text{with}
    \quad
    v(\theta) = v_{-1}(t^\ast) e^{-i\theta} + v_1(t^\ast) e^{i\theta}, \quad t \geq t^\ast,\ \theta \in S^1,
\end{equation*}
Here $v_{\pm 1}$ denote the Fourier coefficients corresponding to the Fourier modes $n=\pm 1$ that are constant for times $t \geq t^\ast$. 
\end{theorem}

\medskip

The idea of the proof is to derive a differential inequality for the energy $E$ that guarantees that the energy drops down to zero in finite time. Using Fourier analysis, we may then prove that solutions with zero average and zero energy are necessarily given by a circle.  
To perform the proof, let first \begin{equation*}
     h_0(\theta) = \bar{h}_0 + v_0(\theta),
     \quad
     \theta \in S^1.
\end{equation*} 
Then, the condition $\|h_0 - \bar{h}_0\|_{H^1(S^1)} < \delta$ for some $\delta > 0$ is equivalent to the condition
$\left\|v_0\right\|_{H^1(S^1)} < \delta$
for some $\delta > 0$.

Thanks to Theorem \ref{thm:existence} there exists a time $T > 0$ and a positive weak solution $h \in C\bigl([0,T];H^1(S^1)\bigr) \cap L_{\alpha+1}\bigl((0,T);W^3_{\alpha+1}(S^1)\bigr)$ on $[0,T]$. We write this solution in terms of its Fourier series as
\begin{equation*}
    h(t,\theta) = c + v(t,\theta)
    \quad \text{with} \quad
    v(t,\theta) = \sum_{n\in \Z, n \neq 0} v_n(t) e^{in\theta},
    \quad \theta \in S^1,
\end{equation*}
for all $t \in (0,T]$,
where the constant $c$ is given by the average 
\begin{equation*}
    c = \frac{1}{2\pi} \int_{S^1} h(t,\theta)\, d\theta.
\end{equation*} 
Then it clearly holds that $\bar{v}=0$, i.e. $v = h-c$ has average zero.

\medskip

\begin{lemma}\label{lem:diff_inequ}
Let $\alpha < 1$. For all $v \in C\bigl([0,T];H^1(S^1)\bigr)$ with $\bar{v} = \frac{1}{2\pi} \int_{S^1} v(t,\theta)\, d\theta = 0$, the energy $E[v](t)$ satisfies the differential inequality
\begin{equation*}
    \frac{d}{dt} E[v](t) \leq -C \bigl(E[v](t)\bigr)^\frac{\alpha+1}{2},
    \quad 
    t \in [0,T].
\end{equation*}
\end{lemma}

\medskip

\begin{proof}
Using the Fourier series representation of $v$ we have that
\begin{equation*}
    v(t,\theta) = \sum_{n\in\Z, n\neq 0} v_n(t) e^{in\theta}
    \quad \text{and} \quad
    E[v](t) = \pi \sum_{n\in\Z, n\neq 0, 1} (n^2-1) \left|v_n(t)\right|^2 \geq 0
\end{equation*}
for all $t\in [0,T]$ and $\theta \in S^1$. We define 
\begin{equation*}
    \Phi(t,\theta) = \de{}v(t,\theta)+\de{3}v(t,\theta)
    =
    \sum_{n\in\Z,n\neq 0,\pm 1} \Phi_n(t) e^{in\theta}
    \quad \text{with} \quad
    \Phi_n(t) = \bigl(in - in^3\bigr) v_n(t)
\end{equation*}  
and write
\begin{equation*}
    v(t,\theta) = (A\Phi)(t,\theta) +  w(t,\theta)
    \quad \text{with} \quad
    w(t,\theta) = a(t) \cos(\theta) + b(t) \sin(\theta)
\end{equation*}
and $a(t), b(t) \in \R$ for all $t \in [0,T]$.
Then the Fourier series representation of $A\Phi$ is
\begin{equation*}
\begin{split}
    (A\Phi)(t,\theta) 
    &= \sum_{n\in\Z, n\neq 0,\pm 1} v_n(t) e^{in\theta} 
    \\
    &= \sum_{n\in\Z, n\neq 0,\pm 1} \frac{\Phi_n(t)}{(in - in^3)} e^{in\theta}
    \\
    &= \frac{1}{2\pi} \sum_{n\in\Z, n\neq 0,\pm 1} \int_{S^1} \frac{\Phi_n(t)}{(in - in^3)} e^{in(\theta-\xi)}\, d\xi,
\end{split}
\end{equation*}
that is, $A\Phi$ is given by
\begin{equation*}
    (A\Phi)(t,\theta) = \int_{S^1} \Phi(t,\xi) G(\theta - \xi)\, d\xi = (\Phi(t)\ast G)(\theta),
    \quad \text{where} \quad
    G(\theta) = \frac{1}{2\pi} \sum_{n\in\Z, n\neq 0,\pm 1} \frac{e^{in\theta}}{(in - in^3)}
\end{equation*}
for $\theta \in S^1$. Moreover, since $G \in C^1(S^1)$, we observe that the derivative $\de{}(A\Phi)$ satisfies
\begin{equation*}
    \left|\de{}(A\Phi(t))\right| \leq 
    \int_{S^1} \left|\Phi(t,\xi)\right|\cdot \sup_{\theta\in S^1} \left|\de{}G(\theta)\right|\, d\xi 
    \leq
    C \left\|\Phi(t)\right\|_{L_1(S^1)} \leq C \left\|\Phi(t)\right\|_{L_{\alpha+1}(S^1)}, \quad t \in [0,T].
\end{equation*}
Consequently, using the identities 
\begin{equation*}
    E[v](t) =
    \frac{1}{2} \int_{S^1} (\de{}v(t))^2 - v(t)^2\, d\theta
    =
    \frac{1}{2} \int_{S^1} \bigl(\de{}(A\Phi(t))\bigr)^2 - (A\Phi(t))^2\, d\theta,
\end{equation*}
we discover that
\begin{equation*}
\begin{split}
    0 &\leq 
    E[v](t)
    \leq
    \frac{1}{2} \left\|\de{}(A\Phi(t))\right\|_{L_2(S^1)}^2
    \leq
    C \left(\int_{S^1} \left|\de{}v(t)+\de{3}v(t)\right|^{\alpha+1}\, d\theta\right)^\frac{2}{\alpha+1}
    \\
    &\leq
    C \left(\int_{S^1} \left|c+v(t)\right|^{\alpha+2} \left|\de{}v(t)+\de{3}v(t)\right|^{\alpha+1}\, d\theta\right)^\frac{2}{\alpha+1}
    =
    C \left(-\frac{d}{dt} E[v](t)\right)^\frac{2}{\alpha+1}.
\end{split}
\end{equation*}
This yields the desired estimate.
\end{proof}

\medskip

Using this differential inequality for the energy, we may prove that the energy converges to zero in finite time. We obtain an upper bound of the \emph{extinction time} $t^\ast$ in terms of the distance $\delta$ from the initial value $h_0$ to its average $\bar{h}_0$. As time tends to $t^\ast$, the first Fourier modes converge to a constant, while the remaining terms converge to zero in the $H^1$-norm. For times $t \geq t^\ast$ the solution can be extended globally in time as the  limit circle at time $t^\ast$, i.e. it does not change its shape at later times. In this way we obtain a solution that is defined for arbitrarily large times.
\medskip

\begin{proof}[\textbf{Proof of Theorem \ref{thm:global_ex}}]
We divide the proof into several steps.\\
(i) We first prove that there exists a positive time $t^\ast > 0$ such that $E[v](t)=0,\ t \geq t^\ast$. In the previous Lemma \ref{lem:diff_inequ} we have derived the differential inequality 
\begin{equation*}
    \frac{d}{dt} E[v](t) \leq -C \bigl(E[v](t)\bigr)^\frac{\alpha+1}{2},
    \quad
    t \in [0,T],
\end{equation*}
for all $v \in C\bigl([0,T];H^1(S^1)\bigr)$ with zero average and for all flow behaviour exponents $\alpha < 1$. This inequality implies that $E[v](\cdot)$ is decreasing and thus
\begin{equation*}
	\frac{d}{dt} \bigl(E[v](t)^\frac{1-\alpha}{2}\bigr) 
	\leq
	-C_\alpha,
	\quad t \in [0,T],
\end{equation*}
as long as $E[v](t) > 0$, where $C_\alpha= \frac{C(1-\alpha)}{2}$.
Therefore,
\begin{equation} \label{eq:diff_inequ_II}
    \bigl(E[v](t)\bigr)^\frac{1-\alpha}{2} \leq \bigl(E[v_0]\bigr)^\frac{1-\alpha}{2} - C_\alpha t,
    \quad
    t \in [0,T], 
    \quad \text{if} \quad E[v](t) > 0.
\end{equation}
This in turn implies that
\begin{equation*}
    E[v](t) \leq \left(\bigl(E[v_0]\bigr)^\frac{1-\alpha}{2} - C_\alpha t\right)^\frac{2}{1-\alpha},
    \quad
    t \in [0,T], \quad \text{if} \quad E[v](t) > 0.
\end{equation*}
Thus, we can conclude that there exists a time $t^\ast \geq 0$ such that
\begin{equation*}
    E[v](t) = 0, \quad t \geq t^\ast,
    \quad \text{where} \quad
    t^\ast \leq \frac{(E[v_0])^\frac{1-\alpha}{2}}{C_\alpha}.
\end{equation*}
Note that $t^\ast > 0$ is strictly positive, if $E[v_0]>0$, as a consequence of the continuity of the energy. Moreover, recall that we assume $h$ to be initially close to a circle in the sense that $\left\|h_0 - c\right\|_{H^1(S^1)} = \left\|v_0\right\|_{H^1(S^1)} < \delta$ for some small but positive $\delta > 0$. If we choose $\delta < (C_\alpha T)^\frac{1}{1-\alpha}$, then we find that
\begin{equation*}
	E[v](t) = 0, \quad t \geq t^\ast,
	\quad \text{where} \quad
	0 < t^\ast
	<
	\frac{\left\|v_0\right\|_{H^1(S^1)}^{1-\alpha}}{C_\alpha}
	< \frac{\delta^{1-\alpha}}{C_\alpha}
	< T.
\end{equation*}
It is worthwhile to mention that the extinction time $t^\ast$ is smaller that the maximal time $T$ of existence if $\delta$ is small enough.


(iii) We can write $h$ in terms of its Fourier series as
\begin{equation*}
    h(t,\theta) = c + v_{-1}(t) e^{-i\theta} + v_1(t) e^{i\theta} + \Phi(t,\theta) \quad \text{with} \quad \Phi(t,\theta) = \sum_{n\in\Z, n \neq 0,\pm 1} v_n(t) e^{in\theta}
\end{equation*}
for $t \in [0,T]$ and $\theta \in S^1$. We show that $\Phi$ converges to zero as time $t$ tends to the extinction time $t^\ast$. It holds that
\begin{align*}
    \left\|\Phi(t)\right\|_{H^1(S^1)}^2
    &=
    \left\|\Phi(t)\right\|_{L_2(S^1)}^2 + \left\|\de{}\Phi(t)\right\|_{L_2(S^1)}^2\\
    &=
    2\pi \sum_{n=\in\Z,n \neq 0,\pm 1} (n^2+1) \left|v_n(t)\right|^2
    \\
    &\leq C \pi \sum_{n\in\Z, n\neq 0,\pm 1} (n^2-1) \left|v_n(t)\right|^2.
    \\
    &= C E[v](t). 
\end{align*}
Consequently, we discover that
\begin{equation*}
    \left\|\Phi(t)\right\|_{H^1(S^1)}^2 \leq C \delta^2, \quad t \in [0,T], \quad \text{and} \quad \left\|\Phi(t)\right\|_{H^1(S^1)}^2 \longrightarrow 0 \quad \text{as } t \to t^\ast
\end{equation*}

(iv) Next, we prove that $\left|v_{\pm 1}(t)\right|$ is bounded by $\delta$ for all $t \in [0,t^\ast)$. To this end, observe first that
\begin{equation*}
    \left|v_{\pm 1}(0)\right| \leq \left\|v_0\right\|_{H^1(S^1)} \leq \delta.
\end{equation*}
Moreover, as in the proof of Lemma \ref{lem:H^1} we can derive the estimate
\begin{equation*}
    \begin{split}
    \left|\frac{d}{dt} v_{\pm 1}(t)\right|
    &=
    \left|\frac{d}{dt} \left(\frac{1}{2\pi} \int_{S^1} h(t,\theta) e^{\mp i\theta}\, d\theta\right)\right| \\
    &=
    \left|\frac{i}{2\pi}\int_{S^1} h^{\alpha+2}  \psi(\de{}h+\de{3}h) e^{\mp i\theta}\, d\theta\right|
    \\
    &\leq
    C \int_{S^1} \left|h\right|^{\alpha+2} \left|\de{}h+\de{3}h\right|^\alpha\, d\theta.
\end{split}
\end{equation*}
Integration with respect to time and applying H\"older's inequality lead to the estimate
\begin{equation*}
\begin{split}
    \left|v_{\pm 1}(t) - v_{\pm 1}(0)\right|
    &\leq
    C \left(\int_0^t \int_{S^1} \left|h\right|^{\alpha+2} \left|\de{}h+\de{3}h\right|^{\alpha+1}\, d\theta\, ds\right)^\frac{\alpha}{\alpha+1} \left(\int_0^t\int_{S^1}\left|h\right|^{\alpha+2}\, d\theta\, ds\right)^\frac{1}{\alpha+1}
    \\
    &\leq C t^\frac{1}{\alpha+1} \sup_{t \in (0,T)} \left\|h(t)\right\|_{L_\infty(S^1)}^\frac{\alpha+2}{\alpha+1} 
    \left(\int_0^t\int_{S^1} \left|h\right|^{\alpha+2} \left|\de{}h+\de{3}h\right|^{\alpha+1}\, d\theta\, ds\right)^\frac{\alpha}{\alpha+1}, \quad t \leq t^\ast,
    \end{split}
\end{equation*}
whence we find that
\begin{equation*}
    \left|v_{\pm 1}(t) - v_{\pm 1}(0)\right| 
    \leq C^\frac{\alpha+2}{\alpha+1} (D_t[h])^\frac{\alpha}{\alpha+1} 
    \leq C \left\|v_0\right\|_{H^1(S^1)}^\frac{2\alpha}{\alpha+1} 
    \leq
    C \delta^\frac{2\alpha}{\alpha+1}.
\end{equation*}
This implies
\begin{equation*}
    \left|v_{\pm 1}(t)\right| \leq \delta + C \delta^\frac{2\alpha}{\alpha+1}, \quad t \leq t^\ast.
\end{equation*}

(v) Next, we prove that $v_{\pm 1}(t) \to v_{\pm 1}(t^\ast)$ as $t \to t^\ast$. To this end, we first observe that the bound
\begin{equation*}
    \left|v_{\pm 1}(t)\right| \leq C \sup_{t \in (0,T)} \left\|h(t,\theta)\right\|_{L_\infty(S^1)} \leq C, \quad t \in [0,T),
\end{equation*}
implies that there exists a sequence $t_n \to t^\ast$ and an element $\chi \in \R$ such that
\begin{equation} \label{eq:conv_t_n}
    \lim_{n\to \infty} v_{\pm 1}(t_n) = \chi.
\end{equation}
Moreover, as in the previous step (iv), we obtain 
\begin{equation*}
\begin{split}
    \lim_{n \to \infty}
    \left|v_{\pm 1}(t_n) - v_{\pm 1}(t)\right| 
    &\leq  \lim_{n \to \infty }C \sup_{s \in (t,t^\ast)} \left\|h(s)\right\|_{L_\infty(S^1)}^\frac{\alpha+2}{\alpha+1} \left(\int_t^{t_n}\int_{S^1} \left|h\right|^{\alpha+2} \left|\de{}h+\de{3}h\right|^{\alpha+1}\, d\theta\, ds\right)^\frac{\alpha}{\alpha+1}
    \\
    & = C \sup_{s \in (t,t^\ast)} \left\|h(s)\right\|_{L_\infty(S^1)}^\frac{\alpha+2}{\alpha+1} \left(\int_t^{t^\ast}\int_{S^1} \left|h\right|^{\alpha+2} \left|\de{}h+\de{3}h\right|^{\alpha+1}\, d\theta\, ds\right)^\frac{\alpha}{\alpha+1}.
\end{split}
\end{equation*}
Thanks to the convergence in \eqref{eq:conv_t_n} we also have
\begin{equation*}
    \left|\chi - v_{\pm 1}(t)\right|
    =
    \lim_{n \to \infty} \left|v_{\pm 1}(t_n) - v_{\pm 1}(t)\right| 
    \leq
    C \sup_{s \in (t,t^\ast)} \left\|h(s)\right\|_{L_\infty(S^1)}^\frac{\alpha+2}{\alpha+1} \left(\int_t^{t^\ast}\int_{S^1} \left|h\right|^{\alpha+2} \left|\de{}h+\de{3}h\right|^{\alpha+1}\, d\theta\, ds\right)^\frac{\alpha}{\alpha+1}
\end{equation*}
and the right-hand side converges to zero, as $t \to t^\ast$.

(vi) Now we construct a solution with infinite lifetime by defining
\begin{equation*}
	h(t,\theta) = c + v_{-1}(t) e^{-i\theta} + v_1(t) e^{i\theta} + \Phi(t,\theta),
	\quad
	t \geq 0,
\end{equation*}
where
\begin{align*}
	v_{\pm 1}(t) &=
	\begin{cases}
		v_{\pm 1}(t), & t \leq t^\ast \\
		v_{\pm 1}(t^\ast), & t > t^\ast,
	\end{cases}
	\\
	\Phi(t,\theta) &=
	\begin{cases}
		\Phi(t,\theta), & t \leq t^\ast \\
		0, & t > t^\ast.
	\end{cases}
\end{align*}
Then $h$ defines for all times $t > 0$ a weak solution of \eqref{eq:pde}.   


(vii) Finally, we prove that the extension of the solution to the time interval $[t^\ast,\infty)$ by 
\begin{equation*}
    h(t,\theta) = c + v_{-1}(t^\ast) e^{-i\theta} + v_1(t^\ast) e^{i\theta}, \quad t \geq t^\ast,\ \theta \in S^1,
\end{equation*}
where $v_{\pm 1} = v_{\pm 1}(t^\ast)$, is unique. For this purpose, assume that 
$f$ is another solution of \eqref{eq:pde} on $[0,\infty)$ such that 
\begin{equation*}
    f(t^\ast,\theta) = h(t^\ast,\theta), \quad \theta \in S^1. 
\end{equation*}
Then also $E[f](t) = 0$ for all $t \geq t^\ast$ and $f$ is given by
\begin{equation}\label{eq:f_circle}
    f(t,\theta) = c + w_{-1}(t) e^{-i\theta} + w_1(t) e^{i\theta}, \quad t \geq t^\ast,\ \theta \in S^1.
\end{equation}
In addition, $f$ satisfies the integral equation
\begin{equation}\label{eq:f_int_eq}
    \int_{t^\ast}^\infty \langle \partial_t f,\varphi\rangle\, dt = -\int_{t^\ast}^\infty \int_{S^1} |f|^{\alpha+2} \psi\bigl(\partial_\theta f + \partial_\theta^3 f\bigr)\, d\theta\ dt,
    \quad \varphi \in L_{\alpha+1}\bigl((0,T);W^1_{\alpha+1}(S^1)\bigr).
\end{equation}
Since $f$ is, for times $t \geq t^\ast$, given by a circle, cf. \eqref{eq:f_circle}, it holds that $\bigl(\partial_\theta f(t,\theta) + \partial_\theta^3 f(t,\theta)\bigr)=0$ pointwise for all $t \geq t^\ast$ and $\theta \in S^1$. Consequently, \eqref{eq:f_int_eq} implies that $f$ is constant in time. This in turn implies that the $w_{\pm 1}$ are constant in time. Finally, the continuity property of the solution implies that $w_{\pm 1} = v_{\pm 1}(t^\ast)$ and the proof is complete.
\end{proof}


\bigskip

\section{The case $\beta$ of order one -- Stability and long time behaviour} \label{sec:comparable_timescales}
In this section we study the stability of constant solutions of \eqref{eq:CaseI_1} and \eqref{eq:CaseI_2} when $\beta$ is of order one, i.e. when the surface tension forces are comparable with the shear forces induced by the rotation of the cylinder. These constant solutions describe circular interfaces which are concentric with the confining cylinders. We prove that, in the scaling limit $\beta$ of order one, solutions of (\ref{eq:CaseI_1}) and \eqref{eq:CaseI_2} with initial data close to a constant, converge to the constant with an error of order 1 as $t\rightarrow\infty.$ A more detailed analysis of the solution shows that the interface behaves, to the leading order, as a circle the center of which moves along a spiral towards the origin $O$, as $t\rightarrow\infty$ (cf. Theorem \ref{thestab}).

In the following, we denote by  $\dot{H}\e{k}(S\e{1})$ the homogeneous Sobolev spaces, consisting of functions $f\in H\e{k}(S\e{1})$ with zero average $\frac{1}{2\pi}\int_{S\e{1}}f\, d\theta=0$.
\bigskip

\subsection{The case $B>0$ of order one}\label{subsect3_1}
In the case $B > 0$ of order one the effects of the surface tension are comparable with those of the shear forces induced by the rotation of the cylinder and we have the evolution equation \eqref{eq:CaseI_1}, i.e.
\begin{equation*}
\partial_t h + \partial_\theta 
\left(h(\theta)^2 \int_0^1 z \psi\left(\tilde{\beta} + 
z\, h(\theta) \bigl(\partial_\theta h(\theta) + \partial^3_\theta h(\theta)\bigr)\right)\, dz\right) = 0,
\quad t > 0,\, \theta \in S^1. 
\end{equation*}
It turns out that, in this physical limit, the asymptotic behaviour of solutions, as $t\to\infty$, is as in the Newtonian case studied in \cite{PV1}. However, since the involved differential operators are different, we briefly sketch the arguments yielding the long-time asymptotics in order to give a complete picture of all the possible scaling limits.
\begin{theorem}\label{thglobal}
Let $c>0$. There exists $\delta>0$ (depending on $c$) such that, for any $h_0\in H\e{4}(S\e{1})$ satisfying $\norm{h_0-c}_{H\e{4}(S\e{1})}<\delta$ with $\frac{1}{2\pi}\int_{S\e{1}}h_0=c$, there exists a unique solution $h\in C([0,\infty); H\e{4}(S\e{1}))\cap C\e{1}((0,\infty);H\e{4}(S\e{1}))$ of \eqref{eq:CaseI_1}, where $h(0,\cdot)=h_0(\cdot)$.
Moreover, we have
\begin{equation}\label{cotathglobal}
    \norm{h(t,\cdot)-c}_{H\e{4}(S\e{1})}\leq C\delta,
    \quad t\geq 0,
\end{equation}
where $C$ depends only on $c$. 
\end{theorem}

In order to prove Theorem \ref{thglobal}, it is convenient to reformulate \eqref{eq:CaseI_1} in a coordinate system rotating at velocity $c$. Moreover, we linearise around the constant solution $h=c$.
More precisely, we define
\begin{equation}\label{defvfuncionh}
    v(t,\varphi)=h(t,\theta)-c\quad\textit{with}\quad\varphi=\theta-c\psi(\tilde{\beta})t.
\end{equation} 
Using the evolution equation \eqref{eq:CaseI_1}, we obtain that $v$ solves the equation
\begin{equation}\label{evoleq}
    \frac{dv}{dt}=L(v)+R(v),
\end{equation}
where $L\colon \dot{H}\e{k}(S\e{1})\to L_{2}(S\e{1})$ is a linear operator, given by
\begin{equation}\label{defl}
    \textcolor{magenta}{
    L(v)
    =
    -\frac{c^3}{3}\psi(\tilde{\beta})\bigl(\partial_\varphi^2 v+\partial_\varphi^4 v\bigr),\quad v\in\dot{H}^4(S^1),}
\end{equation}
and $R\colon\dot{H}\e{k}(S\e{1})\to L_{2}(S\e{1})$ is the non-linear operator defined by
\begin{equation}\label{defr}
    \textcolor{magenta}{
    R(v)
    =
    -\psi(\tilde{\beta})\partial_\phi(v^2) -\frac{5}{3}c\e{2}\psi^\prime(\tilde{\beta})\partial_\phi\bigl(v\,\bigl(\partial_\phi v+\partial_\phi^3 v\bigr)\bigr)+\frac{c\e{4}}{2}\psi^{\prime\prime}(\tilde{\beta})\partial_\phi\bigl(\bigl(\partial_\theta v+\partial_\phi\e{3}v\bigr)\e{2}\bigr)+h.o.t}
\end{equation}
for $v\in\dot{H}^4(S^1)$.
Note that $L$ and $R$ are well-defined bounded operators from $\dot{H}\e{k}(S\e{1})$ to $L_{2}(S\e{1})$.
Now we define 
\begin{equation*}
    V_0= \text{span}\{\cos(\phi),\sin(\phi)\}\subset L_{2}(S\e{1}) \quad \text{and}\quad V_1 = \bigl(V_0\oplus \text{span}\{1\}\bigr)^\bot \subset L_2(S^1),
\end{equation*}
where $\bot$ denotes the orthogonal complement in $L_2(S\e{1})$. 
Moreover, we introduce the following subspaces of $\dot{H}^4(S^1)$,
\begin{equation}\label{e0e1}
    \mathcal{E}_0=V_{0}\cap\dot{H}^4(S^1)\quad\text{and}\quad\mathcal{E}_1=V_{1}\cap\dot{H}^4(S^1),
\end{equation}
and we denote by $P_0$ and $P_1$ the orthogonal projections of $L_{2}(S\e{1})$ onto $V_0$, respectively $V_1$. Therefore, we have that $\dot{H}^4(S^1)=\mathcal{E}_0\oplus\mathcal{E}_1$.
Indeed, using Fourier expansion, it readily follows that $P_0(\dot{H}^4(S^1))\subset \mathcal{E}_0$, and $P_1(\dot{H}^4(S^1))\subset\mathcal{E}_1$. Consequently, we can
write each $v\in\dot{H}^4(S^1)$ as a sum
$v=v_0+v_1$ with $v_0=P_0(v)\in\mathcal{E}_0$ and $v_1=P_1(v)\in \mathcal{E}_{1}$. 
We finally introduce the quadratic operator, $Q\colon\mathcal{E}_0\times\mathcal{E}_0\to\mathcal{E}_1$ by
\begin{equation}\label{defq}
    Q(v_0)
    =
    \frac{i\psi(\tilde{\beta})}{8c^3\psi^\prime(\tilde{\beta})}a_{-1}^2 e^{-2i\phi}-\frac{i\psi(\tilde{\beta})}{8c^3\psi^\prime(\tilde{\beta})}a_{1}^2 e^{2i\phi}
\end{equation}
for each $v_0=a_{-1}e\e{-i\phi}+a_1e\e{i\phi}\in\mathcal{E}_0$, where $a_1=\overline{a_{-1}}$. 
Since the proof Theorem \ref{thglobal} follows the lines of the proof of \cite[Thm 5.1]{PV1}, we only sketch it here.
\medskip

\begin{proof}
First, one can prove existence of solutions
\begin{equation*}
    v \in C\bigl([0,T];\dot{H}^4(S^1)\bigr) \cap C^1\bigl((0,T];\dot{H}^4(S^1)\bigr)
\end{equation*}
on some small time interval $[0,T]$, cf. \cite[Prop. A.1]{PV1}. By standard parabolic theory, one can then prove that the solution even has the better regularity
\begin{equation*}
    v \in C^\infty\bigl((0,T];\dot{H}^k(S^1)\bigr), \quad k \geq 1,
\end{equation*}
cf. \cite[Lemma 5.2]{PV1}. With the improved regularity, one may then derive appropriate a-priori estimates for the operators introduced above in order to obtain global existence by a continuation argument.
\end{proof}

\medskip

A center manifold theory for quasilinear problems has been developed in \cite{Mielke}. We follow here the version of this type of theory as developed in \cite{centermanif}. Similar techniques can be also found in \cite[Chapter 9]{lunardy}.
\medskip

\begin{theorem}\label{thmmani}
Let $L$ and $R$ be defined as in \eqref{defl} and \eqref{defr}. There exists a map $\Phi\in\mathcal{C}\e{k}(\mathcal{E}_0,\mathcal{E}_1)$, where $\mathcal{E}_0=\text{Im} P_0=\text{Re} P_1\subset\dot{H}^4(S\e{1})$ and $\mathcal{E}_1=P_1\dot{H}^4\subset\dot{H}^4(S\e{1})$, with $\Phi(0)=0$ and $D\Phi(0)=0$.
Moreover, there exists a neighbourhood $\mathcal{O}$ of $0$ in $\dot{H}^4(S\e{1})$ such that the manifold 
\begin{equation}\label{manifold}
    \mathcal{M}_0=\{v_0+\Phi(v_0);v_0\in\mathcal{E}_0\}\subset\dot{H}^4(S\e{1})
\end{equation} has the following properties:
\begin{itemize}
    \item[(i)]$\mathcal{M}_0$ is locally invariant, i.e. if $v$ is a solution to \eqref{evoleq}, satisfying $v(0)\in\mathcal{M}_0\cap\mathcal{O}$
        and $v(t)\in\mathcal{O}$ for all $t\in[0,T]$, then $v(t)\in\mathcal{M}_0$ for all $t\in[0,T]$.
    \item[(ii)] $\mathcal{M}_0$ contains the set of     bounded solutions of \eqref{evoleq} that stay    in $\mathcal{O}$ for all $t\in\R$. If $v$ is a solution to $\frac{dv}{dt}=L(v)+R(v)$      that belongs to $\mathcal{M}_0$ for $t\in I$,    where $I\subset\R$ is an open interval, then $v=v_0+\Phi(v_0)$ and $v_0$ satisfies 
        \begin{equation}\label{v0}
            \frac{dv_0}{dt}=L_0(v_0)+P_0R\bigl(v_0+\Phi(v_0)\bigr),
        \end{equation}
        where $L_0$ is the restriction of $L$ to $\mathcal{E}_0$.
        Moreover, $\Phi$ satisfies 
    \begin{equation}\label{psi}
        D\Phi(v_0)\bigl(L_0(v_0)+P_0R(v_0+\Phi(v_0))\bigr)=L_1\Phi(v_0)+P_1R\bigl(v_0+\Phi(v_0)\bigr)\quad\forall v_0\in\mathcal{E}_0.
    \end{equation}
    \item[(iii)] $\mathcal{M}_0$ is locally     attracting in the following sense. If $v(0)\in\mathcal{O}$ and if the solution to \eqref{evoleq}, corresponding to this initial value, satisfies $v(t)\in\mathcal{O}$ for all $t>0$, then there exist an initial value $\tilde{v}(0)\in\mathcal{M}_0\cap\mathcal{O}$ and a constant $a > 0$ such that $\norm{v-\tilde{v}}_{\dot{H}\e{4}(S\e{1})}\leq Ce\e{-at}$, as $t\to\infty$.
\end{itemize} 
\end{theorem}

\medskip

\begin{proof}
In order to prove this theorem, we have to verify that the operators $L$ and $R$ introduced above satisfy the Hypotheses $5.1$--$5.3$ of \cite{PV1}. To this end, observe first that
\begin{itemize}
    \item[(i)] $L\in\mathcal{L}\bigl(\dot{H}^4(S\e{1}),L_{2}(S\e{1})\bigr)$.
    \item[(ii)]For some $k\geq 2$, there exists a neighbourhood $\mathcal{V}\subset\dot{H}^4(S\e{1})$ of $0$ such that $R\in\mathcal{C}^k\bigl(\mathcal{V};L_{2}(S\e{1})\bigr)$, $R(0)=0$ and $DR(0)=0$.
\end{itemize}
Moreover, the linear operator $L$ has the following spectral properties. First, its spectrum can be written as $\sigma=\sigma_0\cup\sigma_-$ where $\sigma_0=\{\lambda\in\sigma;\, \text{Re}\lambda=0\}$ and $\sigma_-=\{\lambda\in\sigma;\, \text{Re}\lambda<0\}$. More precisely, the following statements are true:
\begin{itemize}
    \item[(iii)] There exists a positive constant $\gamma>0$ such that $$\sup_{\lambda\in\sigma_-}(\text{Re}\lambda)<-\gamma.$$ \item[(iv)] $\sigma_0$ consists of a finite number of eigenvalues with finite multiplicities.
\end{itemize}
Finally, there exist positive constants $ s_0>0$ and $C>0$ such that, for all $s\in\R$ with $\abs{s}\geq s_0$, we have that
\begin{itemize}
    \item[(v)] $\norm{(i s I-L_1)^{-1}}_{\mathcal{L}(L_2(S\e{1}))}\leq\frac{C}{\abs{ s}},$
\end{itemize}
where $L_1$ is the restriction of $L$ to $P_1\dot{H}^4(S\e{1})$, $P_1$ is the projection $P_1\colon L_{2}(S\e{1})\to L_{2}(S\e{1})$ defined by $P_1=\mathbb{I}-P_0$ and $P_0$ is the spectral projection corresponding to $\sigma_0$ that is given by
\begin{equation}\label{defp0}
    P_0=\frac{1}{2\pi i}\int_{\Gamma}(\lambda\mathbb{I}-L)\e{-1}d\lambda,
\end{equation}
for a simple, counterclockwise oriented Jordan curve $\Gamma$ surrounding $\sigma_0$ and lying entirely in $\{\lambda\in\C;\, \text{Re}\lambda>-\gamma\}$.
It is worthwhile to note that, thanks to \cite[Rem. 2.16]{centermanif}, the only property that we need to check for the operator $L_1$ is the one in (v). This is the case since we are working with the Hilbert spaces $L_{2}(S\e{1})$ and $\dot{H}\e{4}(S\e{1})$. The theorem is now a minor adaptation of Theorems 2.9 and 3.22 in \cite{centermanif}.
\end{proof}

\medskip

\begin{theorem}\label{thestab}
Let $c>0$. There exist $\delta>0$ and a manifold $\mathcal{M}_0$ as in \eqref{manifold} (both of them depending on $c$) such that all the properties stated in Theorem \ref{thmmani} hold true with $\mathcal{O}=B_{\delta}(0)$. In particular, if $v(0,\cdot)\in\mathcal{M}_0\cap\mathcal{O}$, then the corresponding solution $h$ of \eqref{eq:CaseI_1} (cf. \eqref{defvfuncionh}) satisfies:
\begin{equation}\label{cotasolucmanifold}
    \left\|h(\cdot+c\psi(\tilde{\beta})t,t)-c-\tfrac{2K}{\sqrt{t}}\cos(\cdot+\tilde{K}\log{t}+C_0)\right\|_{H\e{4}(S\e{1})}\leq\frac{C}{t}\quad\textit{for all}\quad t\geq 1,
\end{equation}
where 
\begin{equation}\label{defk}
    K=\sqrt{\frac{2c\e{3}\psi^\prime(\tilde{\beta})}{\psi(\tilde{\beta})}}, \quad \tilde{K}= \frac{c\e{2}\psi^\prime(\tilde{\beta})}{2\psi(\tilde{\beta})},\quad \quad C_0=C_0(h_0),
\end{equation}
and $C$ depends only on $c$. Moreover, if $v(\cdot,0)\in\mathcal{O}$, we have that 
\begin{equation}\label{distancia}
 \text{dist}_{\dhil{4}}(v(\cdot,t),\mathcal{M}_0)\leq Ce\e{-at},
 \quad
 t>0, 
\end{equation}
with $a=a(c)$.
\end{theorem}

\medskip

\begin{proof}
Using the same techniques as in \cite{PV1} we can prove that the operators $L$ and $R$, given by \eqref{defl} and \eqref{defr}, are well-defined from $\dot{H}\e{4}(S\e{1})\to L_{2}(S\e{1})$ and $\dot{H}\e{4}(S\e{1})\times\dot{H}\e{4}(S\e{1})\to L_{2}(S\e{1})$, respectively. 
 
 Moreover, the operator $L$ satisfies the properties (i), (iii) and (iv) of the proof Theorem \ref{thmmani} and the operator $R$ satisfies property (ii) for any $k\geq 2$. 
 Therefore, the operator $P_0$ defined in \eqref{defp0} is the orthogonal projection of $L_{2}(S\e{1})$ onto $\mathcal{E}_0=\text{span}\{\cos(\theta),\sin(\theta)\}$ and there exists a manifold $\mathcal{M}_0$ with the properties stated in Theorem \ref{thmmani} that can be parametrised as in \eqref{manifold} with $\Phi\in\mathcal{C}\e{k}(\mathcal{E}_0,\mathcal{E}_1)$ for any $k\geq 2$.
Furthermore, we have $\Phi(0)=0$, $D\Phi(0)=0$ and, if $v_0=a_{-1}e\e{-i\theta}+a_1e\e{i\theta}\in\mathcal{E}_0$, then
\begin{equation}\label{defDpsi}
    D^2\Phi(0)(v_0,v_0)
    =
    \frac{i\psi(\tilde{\beta})}{4c\e{3}\psi'(\tilde{\beta})}a_{-1}^2 e^{-2i\phi}-\frac{i\psi(\tilde{\beta})}{4c\e{3}\psi'(\tilde{\beta})}a_{1}^2 e^{2i\phi}.
\end{equation}
The differential equation  \eqref{v0} that describes the dynamics of $v$ on this manifold is  reduced to
\begin{equation}\label{evolv0}
\frac{dv_0}{dt}=P_0R\bigl(v_0+\Phi(v_0)\bigr),
\end{equation}
where we used that $L_0(v_0)=0$ for $v_0\in\mathcal{E}_0$. Using \eqref{defDpsi} and \eqref{evolv0}, we can conclude \eqref{cotasolucmanifold}.
Finally, the estimate \eqref{distancia} is a consequence of the global existence result in Theorem \ref{thglobal} and Theorem \ref{thmmani}(iii). 
For more details c.f.\cite{PV1}.
\end{proof}

\begin{remark}
The asymptotic behaviour in \eqref{cotasolucmanifold} can be reformulated in terms of the original non-dimensional variables (cf. \eqref{rescale}) as:

\begin{equation}\label{defvariablesoriginales}
	h(t,\theta)
	=
	\lambda c+\tfrac{2K\lambda}{\sqrt{\frac{\eps}{\tau}\lambda t}}\cos\biggl(\theta-c\psi(\tilde{\beta})\tfrac{\eps}{\tau}\lambda t+\tilde{K}\log{\bigl(\tfrac{\eps}{\tau}\lambda t\bigr)}+C_0\biggr)+\mathcal{O}\bigl(\tfrac{1}{t}\bigr),
\end{equation}
where $\lambda=\frac{1}{\sqrt{B}}$ and $t\gg\frac{\tau}{\eps\lambda}$. 

Moreover, we recall that the interface separating the two fluids is given by the curve $r=1+\varepsilon h$. Therefore, a geometrical argument shows that the interface associated to the solution described by \eqref{defvariablesoriginales} behaves asymptotically as the circle given by 
\begin{displaymath}
\bigl(x-\sigma(t)\cos(\theta_0(t))\bigr)\e{2}+\bigl(y-\sigma(t)\sin(\theta_0(t))\bigr)\e{2}=r_0\e{2},
\end{displaymath}
where $\sigma(t)=\frac{2K\varepsilon\lambda}{\sqrt{\frac{\varepsilon}{\tau}\lambda t}}$, $\theta_0(t)=c\psi(\tilde{\beta})\frac{\varepsilon}{\tau}\lambda t+\tilde{K}\log{(\frac{\varepsilon}{\tau} \lambda t)}+C_0$ and $r_0=1+\varepsilon\lambda c$. Note that the center of this circle spirals towards the origin as $t\to\infty$.
\end{remark}
\begin{figure}[htb]
\center
\includegraphics[width=120mm]{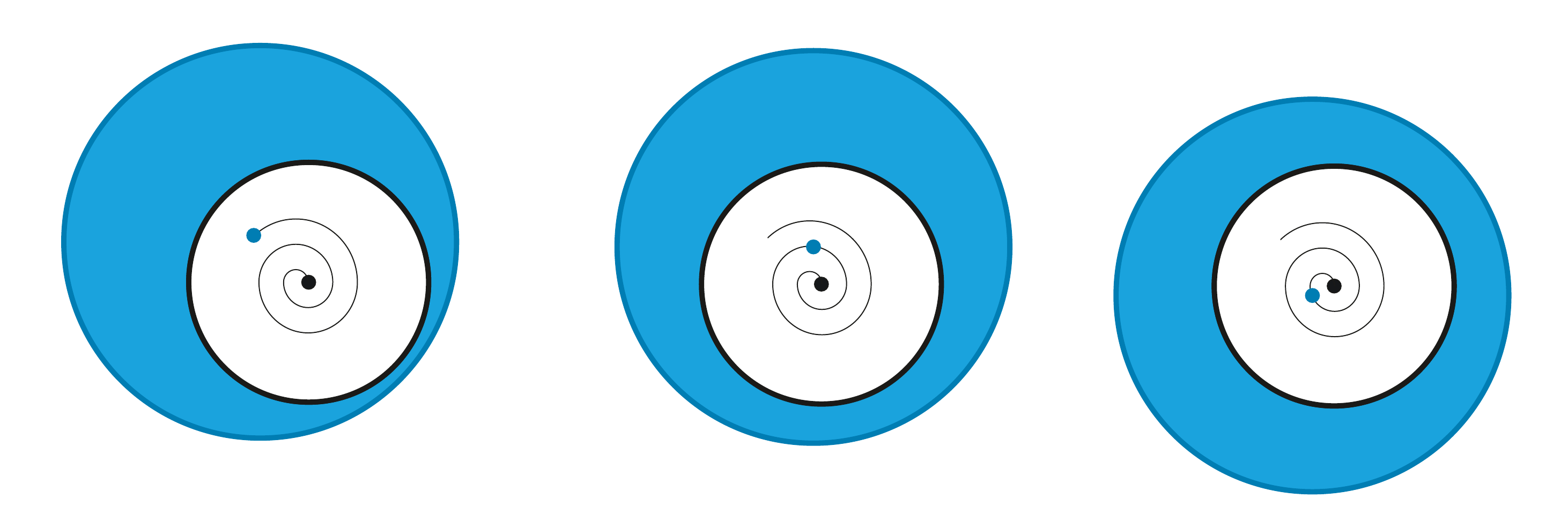}\caption{Center of the interface spiralling towards the center of the cylinders}
\end{figure}
\subsection{The cases $B\to 0$ and $B\to\infty$}
We recall that in these cases we have the evolution equation \eqref{eq:CaseI_2}. For this case we can proof exactly the analogous Theorems in Subsection \ref{subsect3_1}. For this purpose, we consider:
\begin{equation}\label{defvfuncionh2}
    v(t,\varphi)=h(t,\theta)-c\quad\textit{with}\quad\varphi=\theta-c\beta^{\frac{1}{p}}t.
\end{equation} 
Using the evolution equation \eqref{equevol_power-law}, we obtain that $v$ solves the equation
\begin{equation}\label{evoleq2}
    \frac{dv}{dt}=L(v)+R(v),
\end{equation}
where $L\colon\dot{H}\e{4}(S\e{1})\to L_{2}(S\e{1})$ is a linear operator, given by
\begin{equation}\label{defl2}
    L(v)=-\frac{c^3}{3p}\abs{\beta}\e{\frac{1}{p}-1}\bigl(\partial_\varphi^2 v+\partial_\varphi^4 v\bigr),\quad v\in\dot{H}^4(S^1),
\end{equation}
and $R\colon\dot{H}\e{4}(S\e{1})\to L_{2}(S\e{1})$ is the non-linear operator defined by
\begin{equation}\label{defr2}
    R(v)=-\beta\e{\frac{1}{p}}\dfi{}(v^2) -\frac{2}{3}c\e{2}\abs{\beta}\e{\frac{1}{p}-1}\frac{2p+1}{p\e{2}}\dfi{}\bigl(v\bigl(\partial_\varphi v+\partial_\varphi^3 v\bigr)\bigr)+h.o.t.,
    \quad v\in\dot{H}^4(S^1).
\end{equation}
Note that $L$ and $R$ are well-defined bounded operators from $\dot{H}\e{4}(S\e{1})$ to $L_{2}(S\e{1})$. Moreover, we consider the quadratic operator $Q\colon\mathcal{E}_0\times\mathcal{E}_0\to\mathcal{E}_1$ by
\begin{equation}\label{defq2}
    Q(v_0)=\frac{ip\beta}{4c^3}a_{-1}^2 e^{-2i\varphi}-\frac{ip\beta}{4c^3}a_{1}^2 e^{2i\varphi},
    \quad
    v_0=a_{-1}e\e{-i\varphi}+a_1e\e{i\varphi}\in\mathcal{E}_0,
\end{equation}
where $a_1=\overline{a_{-1}}$. Using \eqref{defvfuncionh2}-\eqref{defq2} we can prove the global existence result.

On the other hand, taking 
\begin{equation}\label{defk2}
    K=\sqrt{\frac{c\e{3}}{p\beta\e{\frac{1}{p}+1}}}, \quad \tilde{K}= \frac{2p+1}{p\e{2}}\frac{c\e{2}}{\beta},\quad \quad C_0=C_0(h_0),
\end{equation}
we can prove that \begin{equation}\label{cotasolucmanifold2}
    \left\|h(t,\cdot+ct)-c-\tfrac{2K}{\sqrt{t}}\cos(\cdot+\tilde{K}\log{t}+C_0)\right\|_{H\e{4}(S\e{1})}\leq\frac{C}{t},
    \quad t\geq 1.
\end{equation}
Therefore, we can conclude that in the these cases, we have that the center of the circle spirals towards the origin as in Subsection \ref{subsect3_1}.

\bigskip

\noindent\textsc{Acknowledgement. } The authors have been supported by the Deutsche Forschungsgemeinschaft (DFG, German Research Foundation) through the collaborative research centre 'The mathematics of emerging effects' (CRC 1060, Project-ID  211504053) and the Hausdorff Center for Mathematics (GZ 2047/1, Project-ID 390685813). 
\newpage
\bibliographystyle{plain}
\bibliography{referencias}
\end{document}